\documentclass{amsart}
\usepackage{graphicx} 
\numberwithin{equation}{section}
\usepackage{amsmath}
\usepackage{amssymb}
\usepackage{mathrsfs}
\usepackage{amsthm}
\usepackage{color}
\usepackage{enumerate}
\theoremstyle{definition}
\newtheorem{thm}{Theorem}[section]
\newtheorem{prop}[thm]{Proposition}
\newtheorem{rem}[thm]{Remark}
\newtheorem{lemma}[thm]{Lemma}

\newtheorem{example}[thm]{Example}
\newcommand{\ratio}{\mathfrak{R}}

\newcommand{\edge}{\mathcal{A}}
\newcommand{\p}{\mathcal{I}}
\newcommand{\h}{H}
\newcommand{\return}{\rho}
\newcommand{\indupot}{\tilde\phi}
\newcommand{\indumap}{\tilde f}
\newcommand{\edgesindu}{\tilde{\mathcal{A}}}
\newcommand{\inducoding}{\tilde\Sigma_B}
\newcommand{\indupote}{\tilde\phi}
\newcommand{\frat}{A}
\newcommand{\exponent}{\gamma(f)}
\newcommand{\edgepart}{\tilde{E}}
\newcommand{\inducingdomain}{\mathcal{D}}
\newcommand{\codingmap}{\tilde{\pi}}
\newcommand{\measure}{\tilde{\mu}}
\newcommand{\codingmeasure}{\tilde{\mu}'}
\newcommand{\shift}{\tilde{\sigma}}
\newcommand{\indupressure}{\tilde{p}}
\newcommand{\maps}{\mathcal{S}}
\newcommand{\conv}{\text{Conv}(\{\delta_{x_i}\}_{i\in\p})}
\newcommand{\indulimit}{\tilde{\Lambda}}
\newcommand{\LB}{\text{LB}(q)}
\newcommand{\dimension}{\delta}
\newcommand{\palpha}{p_\alpha}
\newcommand{\oricodingsp}{\Sigma}
\newcommand{\oricodingmap}{\pi}
\newcommand{\pindushift}{\tilde{P}}
\newcommand{\indufin}{\widetilde{Fin}}
\newcommand{\fin}{Fin}

\newcommand{\growth}{s(f)}

\newcommand{\indumultipote}{\tilde{\psi}}
\newcommand{\multipote}{\psi}
\newcommand{\nice}{\mathcal{G}}

\title[Thermodynamic formalism and multifractal analysis]{Thermodynamic formalism and multifractal analysis of Birkhoff averages for non-uniformly expanding R\'{e}nyi interval maps with countably many branches}
\author{Yuya Arima}
\date{\today}
\address{Graduate School of Mathematics, Nagoya University,
Furocho, Chikusa-ku, Nagoya, 464-8602, JAPAN} 
\email{yuya.arima.c0@math.nagoya-u.ac.jp}
\subjclass[2020]{28A78, 37D25, 37D35, 11K55}
\thanks{{\it Keywords}: Thermodynamic formalism, Birkhoff spectra, Multifractal analysis, Non-uniformly expanding interval maps}

\begin{document}

\begin{abstract}
In this paper, we study the multifractal spectrum of Birkhoff averages for non-uniformly expanding R\'{e}nyi interval maps with countably many branches. Our main theorem substantially strengthens the conditional variational formulas established by Jaerisch and Takahasi (2021).
Moreover, for such maps we give necessary and sufficient conditions on the potential for the associated Birkhoff spectrum to be constant.
This provides a positive answer to their conjecture.
As an application, our results enable a detailed analysis of Khinchin exponents and arithmetic means of backward continued fraction expansions in terms of the Hausdorff dimension. 
In addition, we develop the thermodynamic formalism for non-uniformly expanding R\'{e}nyi interval maps with countably many branches.    
\end{abstract}

\maketitle

\section{Introduction}
Let $f:\Lambda\rightarrow \Lambda$ be a Borel measurable dynamical system on a subset $\Lambda$ of $[0,1]$ and let $\phi$ be a continuous potential on $\Lambda$.
The Birkhoff average of $\phi$ at $x\in \Lambda$ is defined by the time average
$
\lim_{n\to\infty}\frac{1}{n}\sum_{i=0}^{n-1}\phi(f^i(x))$
whenever the limit exists. 
Birkhoff averages provide a way to characterize the dynamical system $f$.
Let $\mu$ be an $f$-invariant ergodic Borel probability measure  on $\Lambda$ with $\int|\phi|d\mu<\infty$.
Birkhoff's ergodic theorem then implies that, for $\mu$-a.e. $x\in \Lambda$ the Birkhoff average of $\phi$ at $x$ converges to the space average $\int \phi d\mu$.
Thus, for $\alpha\neq \int \phi d\mu$ the set $B(\alpha)$ of points where the Birkhoff average of $\phi$ converges to $\alpha$ is negligible with respect to $\mu$. 
However, there is still a possibility that $B(\alpha)$  might be a large set from another point of view. This raises the following natural questions: What are the typical or exceptional Birkhoff averages? How large is the set $B(\alpha)$?
To answer these questions we define the Birkhoff spectrum $\alpha\mapsto b(\alpha)$, where $b(\alpha)$ denotes  the Hausdorff dimension with respect to the Euclidean metric on $\mathbb{R}$ of the set $B(\alpha)$ and study its properties. 
We refer the reader to  the books by Pesin \cite{Pesinbook} and Barreira \cite{Barreirabook} for an introduction to the subject of dynamical systems and the dimension theory.
In the uniformly hyperbolic case, the Birkhoff spectrum for a H\"older continuous potential has been well studied  by Barreira and Saussol \cite{BarreiraSaussol}. 
For non-uniformly expanding interval maps with finitely many branches, the multifractal analysis has also been studied by many authors (see, for example, \cite{GelfertRams}, \cite{kessebohmer2004multifractaltobeappdated}, \cite{Nakaishi}, \cite{PollicottWeiss}, \cite{JJTOPnonuniformly}, \cite{simultaneous}, and the references therein). Recently, the author obtained, for such maps, results analogous to those on the multifractal analysis for uniformly expanding Markov maps with finitely many branches.

Non-uniformly expanding R\'{e}nyi interval maps with countably many branches have attracted much attention and have been studied extensively. The main reason for our interest in this class of maps is that it includes the R\'{e}nyi map introduced by R\'{e}nyi \cite{Renyi}, which generates the backward continued fraction expansion. Therefore, by investigating this class, we can, as an important application, study backward continued fraction expansions (see, for example, \cite{aaronson}, \cite{aaronsonnakada}, \cite{Iommilyapunov} and \cite{MixedJT}).
For this class of maps and a potential having certain regularity conditions, Jaerisch and Takahasi established  conditional variational formulas. 
It then follows from these results that the Birkhoff spectrum is  monotone on a certain domain. 

However, the following natural questions remain open: Is it real-analytic and strictly monotone on such a domain? For which $\alpha$ does $b(\alpha)$ attain its maximum? 
For $\alpha\in \mathbb{R}$ is there an ergodic $f$-invariant Borel probability measure $\mu$ on $\Lambda$ such that the Hausdorff dimension of $\mu$ is $b(\alpha)$ and $\int \phi d\mu=\alpha$, and, if such a measure exists, is it unique? 
Our first main theorem provides answers to these questions. 

In \cite{MixedJT}, the authors also gave an interesting example of a potential $\phi$ in the context of backward continued fractions, for which the set $\{\alpha\in \mathbb{R}: B(\alpha)\neq \emptyset\}$ is a non-degenerate interval and the Birkhoff spectrum defined by $\phi$ is constant on this set.
It is known that this phenomenon cannot occur in the case of finitely many branches, which indicates that, for such a potential, the multifractal analysis fails to capture the fine structure of the level sets.
Therefore, it is an important and natural question to give necessary and sufficient conditions on the potential that lead to this failure.
Our second main theorem gives a positive answer in a general setting to the conjecture of Jaerisch and Takahasi \cite{MixedJTkyouto} concerning such necessary and sufficient conditions in the context of backward continued fractions.
Moreover, as an important application of our main theorem, we provide a detailed analysis of Khinchin exponents and arithmetic means of backward continued fraction expansions.

Let $I:=[0,1]$. 
In this paper, for $A \subset I$, $\operatorname{Int}(A)$ and $\overline{A}$ denote its interior and closure in the Euclidean metric on $\mathbb{R}$. Moreover, we denote by $\partial A$ the Euclidean boundary of $A$.
A map $f:I\rightarrow I$ is said to be a non-uniformly expanding R\'{e}nyi interval map with countably many branches if $f$ satisfies the following conditions:
\begin{itemize}
    \item[(NERI1)] There exists a family $\{\Delta_{i}\}_{i\in \mathbb{N}}$ of  subintervals of $I$ such that for each $i,j\in\mathbb{N}$ with $i\neq j$ we have $\text{int}(\Delta_i)\cap\text{int}(\Delta_j)=\emptyset$. Moreover, for every sequence $\{x_i\}_{i\in\mathbb{N}}$ with $x_i\in \overline{\Delta_i}$ we have $\lim_{i\to\infty}x_i=1$. 
    \item[(NERI2)] For all $i\in \mathbb{N}$ the map $f|_{\Delta_i}:\Delta_i\rightarrow f(\Delta_i)$ is a $C^2$ diffeomorphism and $(0,1)\subset f(\Delta_i)\subset [0,1]$. Furthermore, 
    there exists an open set $W_i$ such that $\overline{\Delta_i}\subset W_i$ and $f|_{\Delta_i}$ extends to a $C^2$ diffeomorphism $f_i$ from $W_i$ onto its image.
    \item[(NERI3)] There exists a non-empty finite set $\p\subset \mathbb{N}$ of parabolic indices such that for each $i\in \p$, the map 
    $f_i$ has a unique fixed point $x_i\in \overline{\Delta_i}$ satisfying $|f_i'(x_i)|=1$ and $|f_i'(x)|>1$ for all $x\in W_i\setminus\{x_i\}$. Moreover, there exists $c>1$ such that for all $i\in\h:=\mathbb{N}\setminus\p$ and $x\in W_i$ we have $|f'_i(x)|>c$.
    \item[(NERI4)] $f$ satisfies the R\'{e}nyi condition, that is,
    $
    \sup_{i\in \mathbb{N}}\sup_{x\in W_i}{|f''_i(x)|}/{|f'_i(x)|^2}<\infty.
    $
    \item[(NERI5)] For each $i\in\mathbb{N}$ and each $x\in\partial\Delta_i$
with $f_i(x)=0$, we have $f(x)=0$.
\end{itemize}
Note that (NERI1) implies that $1$ is the unique accumulation point of the set of endpoints of $\{\Delta_i\}_{i\in\mathbb{N}}$. This assumption can be relaxed (see Remark \ref{rem relax}). However, we impose it for the sake of simplifying the exposition and the notation.
For simplicity of notation, we assume that $\p:=\{1,\cdots,\#\p\}$ and we write 
\begin{align}\label{eq def edge}
\edge:=\mathbb{N}.    
\end{align}

Let $f$ be a non-uniformly expanding R\'{e}nyi interval map with countably  many branches.
For each $n\in\mathbb{N}$ and $\omega=\omega_1\cdots\omega_n\in \edge^n$ we set
$
\bar{\Delta}_{\omega}:=\bigcap_{k=0}^{n-1}\overline{f^{-k}(\Delta_{\omega_{k+1}})}$.
Then, by \cite[Proposition 3.1]{MixedJT} the Euclidean diameter  $|\bar \Delta_{\omega}|$ of the set $\bar\Delta_\omega$  converges to $0$, uniformly in all sequences, that is, 
\begin{align}\label{eq uniformly decay of sylinders}
\lim_{n\to\infty}\sup_{\omega\in \edge^n}|\bar\Delta_\omega|=0.   
\end{align} 
Therefore, since for each $n\in\mathbb{N}$ and $\omega\in \edge^n$ the set $\bar\Delta_\omega$ is compact, for each $\omega\in \edge^{\mathbb{N}}$ the set $\bigcap_{n\in\mathbb{N}}\bar\Delta_{\omega_1\cdots\omega_n}$ is a singleton.
We define the coding map $\oricodingmap:\edge^{\mathbb{N}}\rightarrow I$ by 
\[
\{\oricodingmap(\omega)\}=\bigcap_{n\in\mathbb{N}}\bar\Delta_{\omega_1\cdots\omega_n}
\text{ and the limit set $\Lambda$ of $f$ by }
\Lambda:=\oricodingmap(\edge^{\mathbb{N}}).
\]
In this paper, for $J\subset [0,1]$ we always assume that $J$ is endowed with the relative topology from $[0,1]$.
We define $\delta:=\dim_H(\Lambda)$, where $\dim_H(\cdot)$ denotes the Hausdorff dimension with respect to the Euclidean metric on $\mathbb{R}$. 
As in \cite{Iommilyapunov}, for the multifractal analysis we request the following condition:
\begin{itemize}
    \item[(G)] There exist $C\geq 1$ and $\growth>1/\delta$ such that for all $i\in \edge$ and $x\in \bar\Delta_i$ we have $C^{-1}\leq {|f'_i(x)|}/{i^{\growth}}\leq C$
\end{itemize}
Since the open set condition holds (see (NERI1)), the requirement $\growth>1/\delta$ is natural.

Next, we explain conditions regarding an induced map of $f$.
For all $n\in\mathbb{N}$ and $\omega\in \edge^n$ we set 
$I_\omega:=\bar\Delta_\omega\cap\Lambda.$
Define
\[
\inducingdomain:=\bigcup_{i\in \h}I_i\cup
\bigcup_{i\in\p}\bigcup_{j\in \edge_{i}}I_{ij}, 
\text{ where } \edge_i:=\edge\setminus\{i\}.
\]
We define the return time function $\return:\inducingdomain\rightarrow \mathbb{N}\cup\{\infty\}$ by
\[
\return(x):=\inf \{n\in\mathbb{N}:f^n(x)\in \inducingdomain\}
\]
and the induced map $\indumap:\{\return<\infty\}\rightarrow \inducingdomain$ by
\[
 \indumap(x):=f^{\return(x)}(x).
\]
This construction of the induced map $\indumap$ from $f$ is analogous to the construction of an induced map commonly used in the study of Bowen–Series maps associated with Kleinian Schottky groups with parabolic generators (see \cite{JKScuspidal}, \cite{Frechet}, \cite{jaerisch2022multifractal}, \cite{kessebohmer2004multifractaltobeappdated}).
The following conditions allow us to analyze $f$ by using $\indumap$:
\begin{itemize}
    \item[(F)] There exist a constant $C\geq 1$ and an exponent $\exponent\leq1$ such that for all $n\in\mathbb{N}$ and $x\in \{\return=n\}$ we have 
    \[
    \frac{1}{C}\leq\frac{|\indumap'(x)|}{|f'(x)|n^{1+\gamma(f)}}\leq C 
    \]
\end{itemize}
Since $f$ satisfies the R\'{e}nyi condition, the requirement $\exponent\leq 1$ is natural (see \cite{Thaler}).
By slightly modifying arguments in the proofs of \cite[Lemma 2.1]{Nakaishi} and \cite[Proposition 2]{Bruintodd}, which are based on standard calculus techniques, we obtain the following: if for all $i\in\p$ there exist $b_i\in \mathbb{R}\setminus\{0\}$, $\epsilon>0$ and $\alpha\geq 1$ such that $f'(x)=1+b_i(1+\alpha)(x-x_i)^{\alpha}+O((x-x_i)^{\alpha+\epsilon})$ as $x\to x_i$ then $\indumap$ satisfies (F) with $\gamma(f)=\alpha^{-1}$.
Here, $O$ denotes the Landau notation.
    The induced map $\indumap$ is said to be admissible if $\indumap$ satisfies (F).

\begin{example}
The  R\'{e}nyi map $R:[0,1)\rightarrow[0,1)$ is given by 
\begin{align}\label{eq def renyi}
R(x):=\frac{1}{1-x}-\left[\frac{1}{1-x}\right],    
\end{align}
where $[\cdot]$ denotes the floor function. It is well-known that $R$ is a non-uniformly expanding R\'{e}nyi interval map with countably many branches and satisfies (G) with $\growth=2$ and (F) with $\exponent=1$ (see \cite[Section 6]{MixedJT}).
\end{example}

Let $\phi:\Lambda\rightarrow\mathbb{R}$ be a function.
We define the induced potential $\indupote:\Lambda\cap\{\return<\infty\}\rightarrow\mathbb{R}$ of $\phi$ by 
\begin{align}\label{eq def indu pot}
\indupote(x):=\sum_{i=0}^{\return(x)-1}\phi(f^i(x)).
\end{align}
In this paper, we always require the following conditions:
\begin{itemize}
\item[(P)] We have $\inf\{\phi(x):x\in\Lambda\}>0$.
    \item[(H)] $\phi$ is acceptable and there exists $\beta>0$ such that $\indupote$ is locally H\"older with exponent $\beta$ (see Section \ref{sec preliminary}). 
\end{itemize}
We denote by $\mathcal{R}$ the set of all Borel measurable functions $\phi:\Lambda\rightarrow\mathbb{R}$ that are continuous at each parabolic fixed point $x_i$ ($i\in\p$) and satisfy (P) and (H).
Note that our results remain valid for any Borel measurable function $\psi:\Lambda\rightarrow\mathbb{R}$ that is continuous  at each parabolic fixed point $x_i$ ($i\in\p$) and satisfies $\inf_{x\in\Lambda}{\psi(x)}>-\infty$ and (H). Indeed, it suffices to replace $\psi$ with $\phi:=\psi-\inf_{x\in\Lambda}{\psi(x)}+1$.
As in the setting of uniformly expanding maps with countably many branches, we consider the following conditions (see \cite{IommiJordanBirkhoff} and \cite{dalaklis2024multifractal}):
\begin{itemize}
    \item[(R)] The limit 
    $\ratio=\lim_{x\to1}{\phi(x)}/{\log|f'(x)|}\in[0,\infty]$ exists.
        \item[(L)] There exist $\theta>0$ and $\eta,\xi\in\mathbb{R}$ such that for all $x\in \Lambda$ we have 
    \[
   \theta\log |f'(x)|+\eta\leq \phi(x)\leq \theta\log |f'(x)|+\xi.
    \]
\end{itemize}
Notice that 
if $\phi$ satisfies (L) then we have $\ratio=\theta$.

The level set we consider is given by 
\[
\Lambda_\alpha:=\left\{x\in \Lambda:\lim_{n\to\infty}\frac{1}{n}\sum_{i=0}^{n-1}\phi(f^i(x))=\alpha\right\}\ \ (\alpha\in\mathbb{R}).
\]
We define the Birkhoff spectrum $b:\mathbb{R}\rightarrow [0,1]$ by 
$b(\alpha):=\dim_H(\Lambda_\alpha)$. We will use the following notations: For $i\in\p$ we set
\begin{align*}
&\alpha_i:=\phi(x_i),
\ \alpha_{\inf}:=\inf_{\mu\in M(f)}\left\{\int\phi d\mu\right\}\text{ and }
    \alpha_{\sup}:=\sup_{\mu\in M(f)}\left\{\int\phi d\mu\right\},
\end{align*}
where $M(f)$ denotes the set of all $f$-invariant Borel probability measures on $\Lambda$. Note that by \cite[Main theorem (a)]{MixedJT}, for $\alpha\in \mathbb{R}$, $\Lambda_\alpha\neq \emptyset$ if and only if $\alpha\in[\alpha_{\inf},\alpha_{\sup}]$.  
Since $f$ has countably many full-branches, $\Lambda$ is not compact. Hence, in general, $\alpha_{\sup}$ is not finite.  
For $\mu\in M(f)$ we define $\lambda(\mu):=\int \log |f'|d\mu$ and denote by $h(\mu)$ the measure-theoretic entropy defined in \cite{walters2000introduction}.
We set $\underline{\alpha}_\p:=\min_{i\in\p}\{\alpha_i\}$ and $\overline{\alpha}_\p:=\max_{i\in\p}\{\alpha_i\}$.
We define 
$\frat:=[\underline{\alpha}_\p,\overline{\alpha}_\p]$. 
We are now in a position to state our main theorems.
\begin{thm}\label{thm main}
Let $f$ be a non-uniformly expanding R\'{e}nyi interval map with countably  many branches having the admissible induced map $\indumap$ and let $\phi\in\mathcal{R}$ satisfy (R). We also assume that $f$ satisfies (G).
Then, for all $\alpha\in\frat$ we have $b(\alpha)=\delta$. Furthermore, we have the following: 
\begin{itemize}
    \item[(B1)] If $\ratio=0$ then there exist $a^*\in [\alpha_{\inf},\min\frat]$ and $b^*\in [\max\frat,\alpha_{\sup}]$ such that for all $\alpha\in (a^*,b^*)\setminus\frat$ we have $\growth^{-1}<b(\alpha)<\delta$ and there exists the unique ergodic measure $\mu\in M(f)$ such that $0<\lambda(\mu)<\infty$, $\int\phi d\mu=\alpha$ and $b(\alpha)=h(\mu)/\lambda(\mu)$. Moreover, $b$ is real-analytic on $(a^*,b^*)\setminus\frat$ and it is strictly increasing (resp. decreasing) on $(a^*,\min\frat)$ (resp. $(\max\frat,b^*)$), and for all $\alpha\in( (\alpha_{\inf},a^*]\cup[b^*,\alpha_{\sup}))\setminus \frat$ we have $b(\alpha)=\growth^{-1}$.
    \item[(B2)] If $\ratio=\infty$ 
    then for all $\alpha\in[\min\frat,\infty)$ we have $b(\alpha)=\delta$ and for all $\alpha\in (\alpha_{\inf},\min\frat)$ we have $0<b(\alpha)<\delta$  and there exists the unique ergodic measure $\mu\in M(f)$ such that $0<\lambda(\mu)<\infty$, $\int\phi d\mu=\alpha$ and $b(\alpha)=h(\mu)/\lambda(\mu)$. Moreover, $b$ is real-analytic and strictly increasing on $(\alpha_{\inf},\min\frat)$.
    \item[(B3)] If $\phi$ satisfies 
    (L) then for all  $\alpha\in (\alpha_{\inf},\infty)\setminus\frat$ we have $0<b(\alpha)<\delta$  and there exists the unique ergodic measure $\mu\in M(f)$ such that $0<\lambda(\mu)<\infty$, $\int\phi d\mu=\alpha$ and $b(\alpha)=h(\mu)/\lambda(\mu)$.  Moreover, $b$ is real-analytic on $(\alpha_{\inf},\infty)\setminus\frat$ and it is strictly increasing (resp. decreasing) on $(\alpha_{\inf},\min\frat)$ (resp. $(\max\frat,\infty)$).
\end{itemize}
\end{thm}
\begin{thm}\label{thm main second}
  Let $f$ be a non-uniformly expanding R\'{e}nyi interval map with countably  many branches having the admissible induced map $\indumap$ and let $\phi\in\mathcal{R}$ satisfy (R). We also assume that $f$ satisfies (G) and $\alpha_{\inf}<\alpha_{\sup}$. Then, $b$ is constant on $[\alpha_{\inf},\alpha_{\sup}]\cap \mathbb{R}$ if and only if either $\ratio=\infty$ and $\min A=\alpha_{\inf}$, or $\frat=[\alpha_{\inf},\alpha_{\sup}]$.  
\end{thm}

It is difficult to answer the natural questions mentioned above from the conditional variational formula established by Jaerisch and Takahasi \cite{MixedJT} (see Theorem \ref{thm coditional variational principle}). 
 Moreover, it is also difficult to give necessary and sufficient conditions for the multifractal spectrum $b$ to be constant on $[\alpha_{\inf},\alpha_{\sup}]$. 
 However, our main theorems answer all the natural questions mentioned above and provide necessary and sufficient conditions for the multifractal spectrum $b$ to be constant on $[\alpha_{\inf},\alpha_{\sup}]$.
In addition, Theorem \ref{thm main} generalizes the results of Iommi \cite{Iommilyapunov} on the multifractal analysis of the Lyapunov spectrum for non-uniformly expanding R\'enyi interval maps with countably many branches and a single parabolic fixed point.   

For results related to Theorem \ref{thm main} in the setting of
uniformly expanding maps with countably many branches,
see \cite{dalaklis2024multifractal, multifractalanalysisfan,
Khinchinefan, IommiJordanquotients,
IommiJordanBirkhoff, transience, JKregularityof, JKScuspidal,
mauldin2003graph, PollicottWeiss}. In particular, Jaerisch and Kesseb\"ohmer \cite{JKregularityof} studied the approximation of dimension spectra for
uniformly expanding maps with countably many branches
by those for uniformly expanding maps with finitely many branches.

The main difficulties we encounter are as follows:
First, we have to deal with the lack of uniform hyperbolicity due to the presence of parabolic fixed points. This makes it challenging to describe the thermodynamic formalism.
Recall that, even for maps with a countable Markov partition, suitable summability conditions enable us to obtain strong properties of the thermodynamic formalism (e.g., existence and uniqueness of the equilibrium measure, and real-analyticity of a pressure function). However, in general, for $\phi\in\mathcal{R}$ and $(b,q)\in \mathbb{R}^2$ with $p(b,q):=P(-q\phi-b\log|f'|)<\infty$, where $P(-q\phi-b\log|f'|)$ denotes the topological pressure for the potential $-q\phi-b\log|f'|$, an equilibrium measure $\mu$ for this potential with $\lambda(\mu)>0$ does not exist and the pressure function $(b,q)\mapsto p(b,q)$ is not real-analytic on $\text{Int}(\{(b,q)\in\mathbb{R}^2:p(b,q)<\infty\})$ in our setting. 
Therefore, many of the arguments in Iommi and Jordan \cite{IommiJordanBirkhoff} for uniformly expanding interval maps with a countable Markov partition do not work well.
To overcome this difficulty, we extend the thermodynamic formalism developed by Iommi \cite{Iommilyapunov} for non-uniformly expanding R\'{e}nyi interval maps with countably many branches and the geometric potential to the potential $-q\phi-b\log|f'|$ for $(b,q)\in\mathbb{R}^2$. More precisely, for each
\[
(b,q)\in \mathcal{N}:=\text{Int}\left(\left\{(b,q)\in \mathbb{R}^2:\max_{i\in\p}\left\{-q\alpha_i\right\}<p(b,q)<\infty\right\}\right).
\]
we establish the existence and uniqueness of an equilibrium measure, as well as the real-analyticity of the pressure function (see Section \ref{sec thermodynamic}).

Second, the symbolic model for our maps is the full-shift on an infinite alphabet.
As explained above, in our setting, obtaining results on the multifractal analysis from the thermodynamic formalism is much more difficult than in the uniformly hyperbolic setting. In a previous paper \cite{arimanonuniformly}, we performed the multifractal analysis for non-uniformly expanding interval maps with finitely many branches. 
However, to do this, in \cite{arimanonuniformly} we frequently relied on the compactness of $\Lambda$, which allows us to deduce the compactness of $M(f)$, the boundedness of continuous potentials and upper semi-continuity of the entropy map. 
Unfortunately, in our setting, $\Lambda$ is not compact and thus, we cannot directly rely on these properties.   
To overcome this difficulty, we provide a sufficient condition for a sequence of expanding equilibrium measures to be tight. Moreover, we give a sufficient condition for the limit of such a sequence to be an equilibrium measure (see Section \ref{sec convergence}).

\begin{rem}\label{rem relax}
The assumption in (NERI1) that for all sequences $\{x_i\}_{i\in\mathbb{N}}$ with $x_i\in \overline{\Delta_i}$ we have $\lim_{i\to\infty}x_i=1$ can be  generalized as follows: There exist $\ell\in\mathbb{N}$, points $y_1,\cdots,y_\ell\in [0,1]$ and pairwise disjoint infinite subsets $E_1,\cdots,E_\ell\subset\mathbb{N}$ with $\#(\mathbb{N}\setminus \bigcup_{k=1}^\ell E_k)<\infty$ such that for each $1\leq k\leq \ell$, every sequence $\{x_i\}_{i\in E_k}$ with $x_i\in \overline{\Delta_i}$ has $y_k$ as its unique accumulation point. Under this assumption, conditions (G) and (R) are replaced by 
\begin{itemize}
\item[(G')] Let $1\leq k\leq \ell$. There exist a bijection $\iota_k:E_k\rightarrow \mathbb{N}$ (note that $\#E_k=\infty$), a constant $C\geq 1$ and $s_k(f)>1/\delta$ such that for all $i\in \mathbb{N}$ and $x\in \bar\Delta_{\iota_k^{-1}(i)}$ we have $C^{-1}\leq |f'_{\iota_k^{-1}(i)}(x)|/i^{s_k(f)}\leq C$. 
    \item[(R')] There exists $\ratio \in[0,\infty]$ such that for every $1\leq k\leq \ell$ we have $\ratio=\lim_{x\to y_k}\phi(x)/\log |f'(x)|$. 
\end{itemize}
\end{rem}

\subsection{Application of the main theorem to backward continued fraction expansions}\label{sec backward}
An irrational number $x\in(0,1)$ has the following two expansions:
\begin{align}\label{eq def expansion}
    x=\cfrac{1}{a_1(x)+\cfrac{1}{a_2(x)+{\ddots}}} \text{ and } x=1-\cfrac{1}{b_1(x)-\cfrac{1}{b_2(x)-{\ddots}}},
\end{align}
where $a_i(x)\in\mathbb{N}$ and $b_i(x)\in \mathbb{N}$ with $b_i(x)\geq 2$. Moreover, for all $x\in (0,1)\setminus\mathbb{Q}$ each of these expansions is uniquely determined. 
The right-hand  expansion in \eqref{eq def expansion} is called the backward continued fraction expansion of $x\in (0,1)\setminus\mathbb{Q}$, while the left-hand expansion in \eqref{eq def expansion} is called the regular continued fraction expansion of $x\in (0,1)\setminus\mathbb{Q}$. Let $G:(0,1)\setminus\mathbb{Q}\rightarrow (0,1)\setminus\mathbb{Q}$ be the Gauss map defined by $G(x)=1/x-[1/x]$ and let $R$ be the R\'{e}nyi map defined by \eqref{eq def renyi}. For all $x\in (0,1)\setminus\mathbb{Q}$ and $n\in\mathbb{N}$ we have 
\[
a_n(x)=\left[\frac{1}{G^{n-1}(x)}\right] \text{ and } b_n(x)=\left[\frac{1}{1-R^{n-1}(x)}\right]+1.
\]
In particular, the Gauss map $G$ (resp. the R\'{e}nyi map $R$) acts as the shift map on the regular continued fraction expansion (resp. backward continued fraction expansion). Namely, for all $x\in (0,1)\setminus\mathbb{Q}$ and $n\in\mathbb{N}$ we have 
\begin{align}\label{eq important relation backward regular}
\text{$a_n(x)=a_1(G^{n-1}(x))$ and $b_n(x)=b_1(R^{n-1}(x))$.}    
\end{align}

It is well-known that for Lebesgue almost all $x\in (0,1)\setminus\mathbb{Q}$, $a_n(x)>n^{r}$ holds for infinitely many $n$ or finitely many $n$ according to whether $r\leq 1$ or $r>1$. These types of results concerning the various growth rates of $a_n$ as $n\to\infty$ in terms of the Lebesgue measure, summarized in Khinchin's book \cite{Khinchin} led to the question of quantifying the exceptional sets in terms of the Hausdorff dimension. 
In particular, the Hausdorff dimension of the following level sets has been studied in great detail by Fan et al. \cite{Khinchinefan} and by Iommi and Jordan \cite{IommiJordanBirkhoff}: For  $\alpha\in \mathbb{R}\cup\{\infty\}$,
\begin{align*}
    &K_{\text{cf}}(\alpha):=\left\{x\in(0,1)\setminus\mathbb{Q}:\lim_{n\to\infty}\frac{1}{n}\sum_{k=1}^n\log a_k(x)=\alpha \right\}  \text{ and }
    \\& M_{r,\text{cf}}(\alpha):=\left\{x\in(0,1)\setminus\mathbb{Q}:\lim_{n\to\infty}\frac{1}{n}\sum_{k=1}^n(a_k(x))^r=\alpha \right\} \ (r>0).
\end{align*}
Define, for $\alpha\in \mathbb{R}\cup\{\infty\}$, $k_{\text{cf}}(\alpha):=\dim_H K_{\text{cf}}(\alpha)$ and $b_{r,\text{cf}}(\alpha):=\dim_H M_{r,\text{cf}}(\alpha)$ ($r>0$). For $\phi \in \{a_1^r: r>0\}\cup\{\log a_1\}$ we set $\alpha_{\phi}:=\int \phi d\mu_G$, where $\mu_G$ denotes the Gauss measure defined by $d\mu_G:=\frac{dx}{\log2(1+x)}$. Note that the Gauss measure is $G$-invariant and absolutely continuous with respect to the Lebesgue measure.

\begin{thm}{\cite{Khinchinefan}} The function $k_{\text{cf}}$ is real-analytic on $(0,\infty)$, it is strictly increasing on $(0, \alpha_{\log a_1})$ and it is strictly decreasing on $(\alpha_{\log a_1},\infty)$.
\end{thm}
\begin{thm}{\cite[Proposition 6.7]{IommiJordanBirkhoff}}
    If $r\geq 1$ then $b_{r,\text{cf}}$ is real-analytic and strictly increasing on $(1,\infty)$ and $\lim_{\alpha\to\infty}b_{r,\text{cf}}(\alpha)=b_{r,\text{cf}}(\infty)=1$. If $r<1$ then $b_{r,\text{cf}}$ is real-analytic and strictly increasing on $(1,\alpha_{a_1^r})$ and for all $\alpha\in [\alpha_{a_1^r},\infty]$ we have $b_{r,\text{cf}}(\alpha)=1$.
\end{thm}

These theorems naturally lead to the question of determining the Hausdorff dimension of the following level sets:
\begin{align*}
    &K(\alpha):=\left\{x\in(0,1)\setminus\mathbb{Q}:\lim_{n\to\infty}\frac{1}{n}\sum_{k=1}^n\log b_k(x)=\alpha \right\}  \text{ and }
    \\& M_{r}(\alpha):=\left\{x\in(0,1)\setminus\mathbb{Q}:\lim_{n\to\infty}\frac{1}{n}\sum_{k=1}^n(b_k(x))^r=\alpha \right\} \ (r>0).
\end{align*}
Define, for $\alpha\in \mathbb{R}\cup\{\infty\}$, $k(\alpha):=\dim_H K(\alpha)$ and $b_{r}(\alpha):=\dim_H M_{r}(\alpha)$ ($r>0$). 

Jaerisch and Takahasi \cite[Proposition 1.2]{MixedJT} proved that for all $\alpha\in [2,\infty]$ we have $b_1(\alpha)=1$. This means that 
for arithmetic means of the backward continued fraction expansions, the multifractal analysis does not work well.
In contrast to this, the following theorem (Theorem \ref{thm khinchin}) states that the multifractal analysis is valid for Khinchin exponents of backward continued fraction expansions.

We define the partition $\{\Delta_i\}_{i\in \mathbb{N}}$ of $[0,1)$ by setting $\Delta_i:=[1-\frac{1}{i},1-\frac{1}{i+1})$.
Then, using the partition $\{\Delta_i\}_{i\in \mathbb{N}}$, we can show that $R$ is a non-uniformly expanding R\'{e}nyi interval map with countably many branches.
Moreover, $R$ satisfies (G) with $\growth=2$ and (F) with $\exponent=1$. Here, we note that for every $i\in\mathbb{N}$ and every irrational number $x\in \Delta_i$ we have $b_1(x)=i+1$. In particular, for all $r>0$ we have $\lim_{x\to 1} (b_1(x))^r/\log |R'(x)|=\infty$ and the function $\log b_1$ satisfies (L).
Therefore, the combination of Theorem \ref{thm main} and \eqref{eq important relation backward regular} yields the following:
\begin{thm}\label{thm khinchin}
    The function $k$ is real-analytic and strictly decreasing on $(\log 2,\infty)$ and we have $k(\log 2)=1$. 
\end{thm}
\begin{thm}
    For all $r>0$ and $\alpha\in [2^r,\infty)$ we have $b_r(\alpha)=1$.
\end{thm}
This leads to the following natural question: what conditions on the potential $\psi:\{2,3,\cdots\}\rightarrow \mathbb{R}$ ensure that the dimension spectrum is not constant? In other words, under what conditions on $\psi$ does the multifractal analysis become valid? Motivated by this question, Jaerisch and Takahasi \cite{MixedJTkyouto} proposed a conjecture, and the following theorem, which is a direct consequence of Theorem \ref{thm main second}, provides a positive answer.
\begin{thm}
    Let $\psi:\{2,3,\cdots\}\rightarrow \mathbb{R}$ be a monotone increasing function and let $\psi(+\infty):=\lim_{n\to\infty}\psi(n)$. We assume that the limit  $\lim_{n\to\infty}\psi(n)/\log n\in \mathbb{R}\cup\{\infty\}$ exists. Then, we have that
    \[
    \dim_H\left(\left\{x\in (0,1)\setminus\mathbb{Q}:\lim_{n\to\infty}\frac{1}{n}\sum_{k=1}^n\psi(b_k(x))=\alpha\right\}\right)=1
    \]
    for all $\alpha\in [\psi(2),\psi(+\infty)]\cap \mathbb{R}$ if and only if $\lim_{n\to\infty}\psi(n)/\log n=\infty$ or $\psi(2)=\psi(+\infty)$.
\end{thm}

\subsection{Outline of the paper}
The structure of the paper is as follows. In Section \ref{sec preliminary}, we introduce the tools that will be used in Sections \ref{sec thermodynamic} and \ref{sec multifractal}. Section \ref{sec thermodynamic} is devoted to developing the thermodynamic formalism for a non-uniformly expanding R\'{e}nyi interval map with countably many branches. In Section \ref{sec multifractal}, we perform the multifractal analysis and prove Theorem \ref{thm main} and Theorem \ref{thm main second}. 

\textbf{Notations.} Throughout we shall use the following notation:
For an index set $\mathcal{Q}$ and $\{a_{q}\}_{q\in\mathcal{Q}},\{b_{q}\}_{q\in\mathcal{Q}}\subset[0,\infty]$ we write $a_q\ll b_q$ if there exists a constant $C\geq 1$ such that for all $q\in \mathcal{Q}$ we have $a_q\leq Cb_q$. If we have $a_q\ll b_q$ and $b_q\ll a_q$ then we write $a_q\asymp b_q$.
For a probability space $(X,\mathcal{B})$, a probability measure $\mu$ on $(X,\mathcal{B})$ and a measurable function $\psi:X\rightarrow \mathbb{R}$ we set
$\mu(\psi):=\int \psi d\mu$.

\section{Preliminary}\label{sec preliminary}
In this section, we first describe the thermodynamic formalism on a general countable Markov shift. Let $E$ be a countable set and let $A:E\times E\rightarrow\{0,1\}$ be an incidence matrix. We define 
\[
\Sigma_A:=\{\omega\in E^{\mathbb{N}}:A_{\omega_i,\omega_{i+1}}=1,\ i\in\mathbb{N}\}
\]
and the left-shift map $\sigma:\Sigma_A\rightarrow \Sigma_A$ by $\sigma(\omega_1\omega_2\cdots)=\omega_2\cdots$. 
We denote by $\Sigma_A^{n}$ $(n\in\mathbb{N})$ the set of all admissible
words of length $n$ with respect to $A$ and
by $\Sigma_A^{*}$ the set of all admissible words which have a finite length
(i.e. $\Sigma_A^*=\cup_{n\in\mathbb{N}}\Sigma_A^{n}$).
For convenience, we set $\Sigma_A^0 := \{\varnothing\}$, where $\varnothing$ denotes the empty word. 
For $\omega\in E^{n}$ ($n\in\mathbb{N}$) we define the cylinder set of $\omega$ by
$[\omega]:=\{\tau\in\Sigma_A:\tau_{i}=\omega_{i},1\leq i\leq n\}$.
We endow $\Sigma_A$
with the metric $d$ defined by $d(\omega,\omega')={e^{-k}}$ $\text{if}\ \omega_{i}=\omega_{i}'\ \text{for\ all}\ i=1,\cdots,k-1\ \text{and}\ \omega_{k}\neq\omega_{k}'$
and $d(\omega,\omega')=0$ otherwise.
$\Sigma_A$ is said to be finitely primitive if there exist $n\in\mathbb{N}$ and a finite set $\Omega\subset \Sigma_A^n$ such that for all $\tau,\tau'\in E$ there is $\omega=\omega(\tau,\tau')\in \Omega$ for which $\tau\omega\tau'\in \Sigma_A^*$.

We will recall results from the thermodynamic formalism for ($\Sigma_A,\sigma$). For details, we refer the reader to \cite[Section 2]{mauldin2003graph}  and \cite[Section 17, 18, 20]{Urbanskinoninvertible}.
Let $\psi$ be a function on $\Sigma_A$.
For $Z\subset \Sigma_A$ we set 
\[
\psi(Z):=\sup_{\tau\in Z}\psi(\tau).
\]
For all $n\in\mathbb{N}$ we define 
$S_n(\psi):=\sum_{k=0}^{n-1} \psi\circ\sigma^k$.
 For a continuous function $\psi$ on $\Sigma_A$ and $F\subset E$ the topological pressure of $\psi$ introduced by Mauldin and Urba\'nski \cite{mauldin2003graph} is given as
\[
P_F(\psi):=\lim_{n\to\infty}\frac{1}{n}\log\sum_{\omega\in \Sigma_A^n\cap F^n}\exp\left(
S_n(\psi)([\omega]\cap F^{\mathbb{N}})
\right).
\]
If $F=E$, we simply write $P(\psi)$ for $P_F(\psi)$. 

A continuous function $\psi:\Sigma_A\rightarrow\mathbb{R}$ is called acceptable if it is uniformly continuous and $\sup_{e\in E}\{\sup(\psi|_{[e]})-\inf (\psi|_{[e]})\}<\infty$. Moreover, $\psi$ is said to be locally H\"older with exponent $\beta>0$ if  
\begin{align}\label{def locally holder}
    \sup_{n\in\mathbb{N}}\sup_{\omega\in \Sigma_A^n}\sup\{|\psi(\tau)-\psi(\tau')|(d(\tau,\tau'))^{-\beta}:\tau,\tau'\in [\omega],\ \tau\neq\tau'\}<\infty.
\end{align}
Note that for all $\beta>0$ if a function $\psi$ on $\Sigma_A$ is locally H\"older with exponent $\beta>0$ then $\psi$ is acceptable. 

\begin{thm}\label{thm compact approximation}{\cite[Theorem 2.1.5]{mauldin2003graph}}
    If $\psi:\Sigma_A\rightarrow \mathbb{R}$ is acceptable and $\Sigma_A$ is finitely primitive then we have $P(\psi)=\sup\{P_F(\psi):F\subset E,\ \#F<\infty\}$.
\end{thm}
For further discussion of this compact approximation property
of the topological pressure, we refer the reader to \cite{JKSinduced}.

We denote by $M(\sigma)$ the set of $\sigma$-invariant Borel probability measures on $\Sigma_A$. 
\begin{thm}{\cite[Theorem 2.1.8]{mauldin2003graph}}\label{thm variational principle induce} Suppose that $\Sigma_A$ is finitely primitive. If $\psi:\Sigma_A\rightarrow \mathbb{R}$ is acceptable then  we have the variational principle, that is,
$
P(\psi)=\sup_{\mu}\left\{h(\mu)+
\mu\left(\psi\right) \right\},$
where the supremum is taken over the set of measures $\mu\in M(\sigma)$ satisfying $\mu(\psi)>-\infty$. 
\end{thm}
\begin{prop}{\cite[Proposition 2.1.9]{mauldin2003graph}}\label{prop finite pressure}
 If $\psi:\Sigma_A\rightarrow \mathbb{R}$ is acceptable and $\Sigma_A$ is finitely primitive then $P(\psi)<\infty$
  if and only if 
$ \sum_{e\in E}\exp\left(\psi([e])\right)<\infty.$
\end{prop}
For $\psi:\Sigma_A\rightarrow \mathbb{R}$ with $P(\psi)<\infty$ a measure $\mu\in M(\sigma)$ is called a Gibbs measure for $\psi$ if there exists a constant $Q\geq 1$ such that for every $n\in\mathbb{N}$, $\omega\in \Sigma_A^n$ and $\tau\in[\omega]$ we have
\begin{align}\label{eq gibbs}
\frac{1}{Q}\leq \frac{\mu([\omega])}{\exp(S_n(\psi)(\tau)-P(\psi)n)}\leq Q.    
\end{align}
\begin{thm}\cite[Theorem 2.2.4 and Corollary 2.7.5]{mauldin2003graph}\label{thm Gibbs induce}
    Suppose that $\psi:\Sigma_A\rightarrow \mathbb{R}$ is locally H\"older with exponent $\beta>0$ and satisfies $P(\psi)<\infty$. If $\Sigma_A$ is finitely primitive then there exists a unique Gibbs measure $\mu\in M(\sigma)$ for $\psi$. Moreover, $\mu$ is ergodic.
\end{thm}
For $\psi:\Sigma_A\rightarrow \mathbb{R}$ with $P(\psi)<\infty$ we say that $\mu\in M(\sigma)$ is an equilibrium measure for $\psi$ if we have $\mu(\psi)>-\infty$ and $P(\psi)=h(\mu)+\mu(\psi)$.
\begin{thm}{\cite[Theorem 2.2.9]{mauldin2003graph}}\label{thm equilibrium state induce}
   Suppose that $\psi:\Sigma_A\rightarrow \mathbb{R}$ is locally H\"older with exponent $\beta>0$ and satisfies $P(\psi)<\infty$ and $\Sigma_A$ is finitely primitive. Furthermore, assume that $\mu(\psi)>-\infty$, where $\mu$ denotes the unique Gibbs measure for $\psi$ obtained in Theorem \ref{thm Gibbs induce}.
   Then, $\mu$ is the unique equilibrium measure for $\psi$.
\end{thm}

Let $f$ be a non-uniformly expanding R\'{e}nyi interval map with countably  many branches having the admissible induced map $\indumap$. 
Next, we describe a coding space of the induced map $\indumap$.
For $A\subset\Lambda$ we denote by $\text{Int}_{\Lambda}(A)$ the interior of $A$ with respect to the topology on $\Lambda$.
We define, $\text{for $n\geq2$}$,
\begin{align*}
\edgepart_{1}:=
\bigcup_{i\in \edge}\{i\h\}\cup\bigcup_{i\in \p}\bigcup_{j\in\edge_i}\{ji\edge_i\},\ \edgepart_n:=\bigcup_{i\in \p}\bigcup_{j\in \edge_{{i}}}\{ji^n\edge_i\}
\text{ and }
\edgesindu:=\bigcup_{n\in\mathbb{N}}\edgepart_n.
\end{align*}
For each $i\h\in \tilde E_1$ and $ji^n\edge_i\in \edgepart_n$ $(n\in\mathbb{N})$ we set
$    I_{iH}:=\bigcup_{j\in \h}I_{ij}$ and $I_{ji^n\edge_i}:=\bigcup_{k\in \edge_i}I_{ji^nk}$.
We notice that $\{\return<\infty\}=\bigcup_{\omega\in \edgesindu}I_\omega$. Moreover, for all $iH\in \edgepart_1$ and $ji^n\edge_i\in \edgepart_n$ ($n\in\mathbb{N}$) we have
\begin{align}\label{eq finitely primitive}
\indumap(\text{Int}_{\Lambda}(I_{iH}))=\text{Int}_{\Lambda}\left(\bigcup_{j\in H}I_{j}\right) \text{ and  }
\indumap(\text{Int}_{\Lambda}(I_{ji^n\edge_i}))=\text{Int}_{\Lambda}\left(\bigcup_{k\in\edge_i}I_{ik}\right).
\end{align}
Therefore, 
if $\indumap(I_\omega)\cap I_{\omega'}$ ($\omega, \omega'\in\edgesindu$) has non-empty interior then $\text{Int}_{\Lambda}(I_{\omega'})\subset \indumap(I_\omega)$.
This implies that
$\indumap:\{\return<\infty\}\rightarrow \inducingdomain$ is a Markov map with the countable Markov partition $\{I_\omega\}_{\omega\in \edgesindu}$.
We define the incidence matrix $B:\edgesindu\times \edgesindu\rightarrow\{0,1\}$
by $B_{\omega,\omega'}=1$ if $\text{Int}_{\Lambda}(I_{\omega'})\subset \indumap(I_{\omega})$ and
$B_{\omega,\omega'}=0$ otherwise. 
Define the countable Markov shift $(\inducoding,\shift)$ by
\[
\inducoding:=\lbrace\omega\in \edgesindu^{\mathbb{N}}:B_{\omega_{n},\omega_{n+1}}=1,\ n\in\mathbb{N}\rbrace,
\]
where $\shift:\inducoding\rightarrow\inducoding$ denotes the left shift map. 
For each $n\in\mathbb{N}$ and $\omega\in \inducoding^n$ we set 
$I_\omega:=\bigcap_{k=0}^{n-1}\overline{\indumap^{-k}(I_{\omega_{k+1}})}$. 
By \eqref{eq uniformly decay of sylinders}, 
for each $\omega\in \inducoding$ the set 
$\bigcap_{n\in\mathbb{N}}I_{\omega_1\cdots\omega_n}$
is a singleton. Thus, we can define the coding map $\codingmap:\inducoding\rightarrow \codingmap(\inducoding)$
by 
\[
\{\codingmap(\omega)\}=\bigcap_{n\in\mathbb{N}}I_{\omega_1\cdots\omega_n} \text{ and set }\indulimit:=\codingmap(\inducoding).
\]
Then, we have $\indumap(\indulimit)=\indulimit$. 
We denote by $M(\indumap)$ the set of $\indumap$-invariant Borel probability measures on $\indulimit$. For $A\subset\Lambda$ we denote by $\partial_{\Lambda}A$ the boundary of $A$ with respect to the topology on $\Lambda$.
\begin{rem}\label{rem measurable bijection}
    We notice that $\codingmap$ is continuous and one-to-one except on the preimage of the countable set $J_0:=\bigcup_{n=0}^\infty \indumap^{-n}(\bigcup_{\omega\in \edgesindu}\partial_{\Lambda} I_{\omega})$, where it is at most two-to-one. Furthermore, we have $\indumap\circ\codingmap=\codingmap\circ\shift$ on $\inducoding\setminus\codingmap^{-1}(J_0)$ and the restriction of $\codingmap$ to $\inducoding\setminus\codingmap^{-1}(J_0)$ has a continuous inverse.  Thus, $\codingmap$ induces a measurable bijection between $\inducoding\setminus\codingmap^{-1}(J_0)$ and $\indulimit\setminus J_0$. 
    Furthermore, by the same argument as in the proof of \cite[Lemma 3.5]{jaerisch2022multifractal}, for any $\measure\in M(\indumap)$ there exists $\codingmeasure\in M(\shift)$ such that $\measure=\codingmeasure\circ\codingmap^{-1}$ and $h(\measure)=h(\codingmeasure)$, and for $\codingmeasure\in M(\shift)$ we have $\codingmeasure\circ\codingmap^{-1}\in M(\indumap)$ and  $h(\codingmeasure\circ\codingmap^{-1})=h(\codingmeasure)$. 
\end{rem}
{
With a slight abuse of notation, we define the following functions:
\begin{align*}
&\return\circ\codingmap:\inducoding\rightarrow\mathbb{R},\ \return\circ\codingmap|_{\bigcup_{\omega\in \edgepart_n}[\omega]}:=n\ \text{ for all }n\in\mathbb{N},
\\&
\indupote\circ \codingmap:\inducoding\rightarrow\mathbb{R},\ \indupote\circ\codingmap(\omega):=\sum_{i=0}^{\return\circ \codingmap(\omega)-1}\phi(f^i(\codingmap(\omega))) \text{ and }
\\&
\log|\indumap'\circ\codingmap|:\inducoding\rightarrow\mathbb{R},\ 
\log|\indumap'\circ\codingmap|(\omega):=\sum_{i=0}^{\return\circ\codingmap(\omega)-1}
\log|f'(f^{i}(\codingmap(\omega)))|.
\end{align*}
Then, by definition, for all $\omega \in \inducoding\setminus \codingmap^{-1}(J_0)$ and $x=\codingmap(\omega)$ we have $\return(x)=\return\circ\codingmap(\omega)$, $\indupote(x)=\indupote\circ\codingmap(\omega)$ and $\log |\indumap'(x)|=\log |\indumap'\circ\codingmap(\omega)|$. In particular, for all $\codingmeasure\in M(\tilde \sigma)$ with $\codingmeasure(\codingmap^{-1}(J_0))=0$ we have $\codingmeasure\circ\codingmap^{-1}(\return)=\codingmeasure(\return\circ\codingmap)$, $\codingmeasure\circ\codingmap^{-1}(\indupote)=\codingmeasure(\indupote\circ \codingmap)$ and $\codingmeasure\circ\codingmap^{-1}(\log|\indumap'|)=\codingmeasure(\log|\indumap'\circ \codingmap|)$.
For $(b,q,s)\in\mathbb{R}^3$ we write
$(q\indupote+b\log|\indumap'|+s\return)\circ\codingmap=q \indupote \circ\codingmap+b\log|\indumap'\circ\codingmap|+s\return\circ\codingmap.$
}

Let $\phi:\Lambda\rightarrow\mathbb{R}$ be a Borel measurable function and let $\indupot$ be the induced potential defined by \eqref{eq def indu pot}. For some $\beta>0$ the induced potential $\indupot$ is said to be locally H\"older with exponent $\beta$ if  $\indupote\circ\codingmap$ is locally H\"older with exponent $\beta$.
By using the R\'{e}nyi condition (NERI4) for $f$, one can show that $\indumap$ satisfies the R\'{e}nyi condition, that is, 
\begin{align}\label{eq renyi induce}
\sup_{\omega\in \edgesindu}\sup_{x\in I_\omega}\frac{|\indumap''(x)|}{|\indumap'(x)|^2}<\infty.    
\end{align}
Therefore, since $\indumap$ is uniformly expanding, that is, there exists $c>1$ such that for all $x\in \indulimit$ we have $|\indumap'(x)|>c$, there exists $\beta>0$ such that $\log|\indumap'|$ is locally H\"older with exponent $\beta$.
In the following, for $\phi\in \mathcal{R}$ we assume that, if necessary by replacing the exponent $\beta$ with a smaller number, $\log|\indumap'|$ and $\indupote$ are locally H\"older with $\beta$.
Also, note that, since for each $\omega\in \edgesindu$, $\return\circ\codingmap$ is constant on $[\omega]$, $\return$ is  locally H\"older with exponent $\beta$. 
By \eqref{eq finitely primitive}, it is not difficult to verify that $\inducoding$ is finitely primitive. These conditions allow us to apply results from the thermodynamic formalism for  countable Markov shifts as introduced above.

    Let $\phi\in \mathcal{R}$. Define the pressure function $\indupressure:\mathbb{R}^3\rightarrow \mathbb{R}\cup\{\infty\}$ by 
\begin{align*}
\indupressure(b,q,s):=\pindushift(-q\indupote\circ\codingmap-b\log|\indumap'\circ\codingmap|-s\return\circ\codingmap),
\end{align*}
where $\pindushift$ denotes the topological pressure with respect to $(\inducoding,\shift)$.
Let
$\indufin:=\{(b,q,s)\in \mathbb{R}^3:\indupressure(b,q,s)<\infty\}.$    
In the following, for $(b,q,s)\in \indufin$ we denote by $\codingmeasure_{b,q,s}$ the unique $\shift$-invariant Gibbs measure obtained in Theorem \ref{thm Gibbs induce}.
Note that since $\codingmeasure_{b,q,s}$ is a Gibbs measure which yields $\codingmeasure_{b,q,s}(\codingmap ^{-1}(J_0))=0$, 
for all $(b,q,s)\in \indufin$ we have $\codingmeasure_{b,q,s}(\return\circ \codingmap)=\codingmeasure_{b,q,s}\circ\codingmap^{-1}(\return)$, $\codingmeasure_{b,q,s}(\indupote\circ \codingmap)=\codingmeasure_{b,q,s}\circ\codingmap^{-1}(\indupote)$ and $\codingmeasure_{b,q,s}(\log|\indumap'\circ \codingmap|)=\codingmeasure_{b,q,s}\circ\codingmap^{-1}(\log|\indumap'|)$.
For $\codingmeasure\in M(\shift)$ we define $\lambda(\codingmeasure):=\codingmeasure( \log|\indumap'\circ\codingmap|)$.
\begin{thm}{\cite[Theorem 2.6.12 and Proposition 2.6.13]{mauldin2003graph} (see also \cite[Theorem 20.1.12]{Urbanskinoninvertible})} \label{thm regularity of induced pressure}The function $(b,q,s)\mapsto\indupressure(b,q,s)$ is real-analytic on $\text{Int}(\indufin)$. Moreover, we have Ruelle's formula, that is, for $(b,q,s)\in \text{Int}(\indufin)$,
$
\frac{\partial}{\partial b}\indupressure(b,q,s)=-\lambda(\codingmeasure_{b,q,s}),
$
$
\frac{\partial}{\partial q}\indupressure(b,q,s)=-\codingmeasure_{b,q,s}(\indupote\circ\codingmap),
$
$
\frac{\partial}{\partial s}\indupressure(b,q,s)=-\codingmeasure_{b,q,s}(\return\circ\codingmap).
$
\end{thm}

Next, we describe the thermodynamic formalism on the dynamical system $(f,\Lambda)$.
Let $\phi:\Lambda\rightarrow \mathbb R$ be a Borel measurable  function.
We introduce the topological pressure of $\phi$ with respect to the dynamical system $(f,\Lambda)$ by 
\begin{align}\label{eq def non-induced pressure}
P_f(\phi):=\sup\left\{h(\mu)+\mu(\phi):\mu\in M(f),\ \mu(\phi)>-\infty\right\}.    
\end{align}
We say that $\mu\in M(f)$ is an equilibrium measure for $\phi$ if $\mu$ satisfies $\mu(\phi)>-\infty$ and
$
P_f(\phi)=h(\mu)+\mu(\phi).
$
We are interested in the pressure function $p:\mathbb{R}^2\rightarrow \mathbb{R}\cup\{\infty\}$ given by
\[
p(b,q):=P_f(-q\phi-b\log|f'|).
\]
Let 
\begin{align*}
&\fin:=\{(b,q)\in \mathbb{R}^2:p(b,q)<\infty\} \text{ and let }    \\&
\ 
\mathcal{N}:=\text{Int}\left(\left\{(b,q)\in \mathbb{R}^2:p(b,q)>\max_{i\in\p}\left\{-q\alpha_i\right\}\right\}\cap\fin\right).
\end{align*}
We define $\oricodingsp:=\edge^{\mathbb{N}}$ and denote by $\sigma$ the left shift map on $\oricodingsp$. 
Note that $\oricodingsp$ is a full-shift. In particular, $\oricodingsp$ is finitely primitive. 
Let $\oricodingmap:\oricodingsp\rightarrow \Lambda$ be the coding map defined as in the introduction.
A function $\psi:\Lambda\rightarrow \mathbb{R}$ is said to be acceptable if $\psi\circ\oricodingmap$ is acceptable. 
{
Define $\log|f'\circ \oricodingmap|:\oricodingsp\rightarrow \mathbb{R}$ by $\log|f'\circ \oricodingmap|(\omega):=\log |f'_{\omega_1}(\pi(\omega))|$, where $\omega=\omega_1\omega_2\cdots\in \oricodingsp$. For a function $\psi:\Lambda\rightarrow\mathbb{R}$ and $b\in\mathbb{R}$ we write
$(\psi+b\log|f'|)\circ \oricodingmap=\psi\circ\oricodingmap+b\log|f'\circ\oricodingmap|$.
}
By the R\'{e}nyi condition (NERI4) for $f$, we can apply \cite[Proposition 8.2.1]{mauldin2003graph} to obtain the following:
\begin{prop}{\cite[Proposition 8.2.1]{mauldin2003graph}}\label{prop acceptable}
Let $f$ be a non-uniformly expanding R\'{e}nyi interval map with countably  many branches having the admissible induced map $\indumap$. Then, the geometric potential $\log|f'\circ\oricodingmap|$ is acceptable.     
\end{prop}

Let $\phi\in \mathcal{R}$. Then, by the above proposition, for all $(b,q)\in \mathbb{R}^2$ the function $-q\phi\circ\oricodingmap-b\log|f'\circ\oricodingmap|$ is acceptable. 

\begin{rem}\label{rem measurable bijection oricoding}
By the same reasoning as in Remark \ref{rem measurable bijection}, for any $\mu\in M(f)$ there exists $\mu'\in M(\sigma)$ such that $\mu=\mu'\circ\oricodingmap^{-1}$ and $h(\mu)=h(\mu')$, and for $\mu'\in M(\sigma)$ we have $\mu'\circ\oricodingmap^{-1}\in M(f)$ and  $h(\mu'\circ\oricodingmap^{-1})=h(\mu')$.    
\end{rem}
Note that $f$ has full branches (NERI2). Hence,  
by the variational principle (Theorem \ref{thm variational principle induce}) and Remark \ref{rem measurable bijection oricoding}, for all $(b,q)\in \mathbb{R}^2$ we obtain
    $p(b,q)=P_\sigma(-q\phi\circ \oricodingmap-b\log|f'\circ\oricodingmap|),\nonumber$
where the supremum is taken over the set of measures $\mu'\in M(\sigma)$ satisfying $\mu'(-q\phi\circ\oricodingmap-b\log|f'\circ\oricodingmap|)>-\infty$ and $P_\sigma$ denotes the topological pressure with respect to the countable Markov shift $(\oricodingsp,\sigma)$. 
Thus, by Proposition \ref{prop finite pressure}, 
\begin{align}\label{eq finite pressure}
\text{$(b,q)\in\fin$ if and only if }
\sum_{i\in\edge}\exp\left((-q\phi-b\log|f'|)\circ\oricodingmap([i])\right)<\infty.    
\end{align}
For simplicity, we will denote $P_f$ and $P_\sigma$ both by $P$.

For $\tilde \nu\in M(\indumap)$ with $\tilde\nu(\return)<\infty$ we define
\begin{align}\label{eq def lift}
\nu:=\frac{1}{\tilde\nu( \return )}\sum_{n=0}^{\infty}\sum_{k=n+1}^\infty\tilde\nu|_{\{\return=k\}}\circ f^{-n}.    
\end{align}
Since $\indumap$ is a first return map of $f$, it is well-known that for $\tilde\nu\in M(\indumap)$ with $\tilde\nu( \return )<\infty$ we have $\nu\in M(f)$ (for example see \cite[Proposition 1.4.3]{viana}). Also, if $\tilde\nu\in M(\indumap)$ with $\tilde\nu( \return )<\infty$ is ergodic then we have Abramov-Kac's formula:
\begin{align}\label{eq Abramov-Kac's formula}
    \tilde\nu( \return ) h(\nu)=h(\tilde \nu) \text{ and }\tilde\nu( \return )\nu(\psi)=\tilde\nu(\tilde\psi)
\end{align}
for a Borel measurable function $\psi$ on $\Lambda$  with $\tilde\nu(|\tilde \psi|)<\infty$, where $\tilde \psi$ is the induced potential of $\psi$ defined by \eqref{eq def indu pot}.  Conversely, for an ergodic measure $\nu\in M(f)$ with $\nu(\indulimit)>0$ and $\tilde \nu:=\nu|_{\indulimit}/\nu(\indulimit)$  we have 
\begin{align}\label{eq classical Abramov-Kac's formula}
    \tilde\nu( \return ) h(\nu)=h(\tilde \nu) \text{ and }\tilde\nu( \return )\nu(\psi)=\tilde\nu( \tilde\psi)
\end{align}
for a Borel measurable function $\psi$ on $\Lambda$ with $\nu(|\psi|)<\infty$, where $\tilde \psi$ is the induced potential of $\psi$ defined by \eqref{eq def indu pot}. 
Define, for a finite set $\mathcal{L}$ and $\{\nu_{\ell}\}_{\ell\in \mathcal{L}}\subset M(f)$,
$
\text{Conv}(\{\nu_{\ell}\}_{\ell\in \mathcal{L}}):
=
\{\sum_{\ell\in\mathcal{L}}p_\ell\nu_{\ell}:\{p_\ell\}_{\ell\in\mathcal{L}}\subset[0,1],\ \sum_{\ell\in\mathcal{L}}p_\ell=1
\}.$
\begin{lemma}\label{lemma equivalent condition not liftable}
    Let $\nu\in M(f)$. Then $\nu(\indulimit)=0$ if and only if $\nu\in \conv$, where $\delta_{x_i}$ ($i\in\p$) denotes the Dirac measure at $x_i$. 
\end{lemma}
The proof of Lemma \ref{lemma equivalent condition not liftable} is straightforward and is therefore omitted.

\section{thermodynamic formalism}\label{sec thermodynamic}
In this section, we show that the pressure function $(b,q)\mapsto p(b,q)$ has nice properties on $\mathcal{N}$ as in the uniformly expanding case. To prove this, we use the 
induced pressure function $(b,q,s)\mapsto\indupressure(b,q,s)$.
A closely related study of the relationship between induced
and non-induced pressure functions can be found in \cite{JKSinduced}.

We denote by $\maps$ the set of non-uniformly expanding R\'{e}nyi interval maps with countably many branches having the admissible induced map $\indumap$.   
In this section, we assume throughout that $f\in \maps$ and $\phi\in \mathcal{R}$.
Recall that $\alpha_i:=\phi(x_i)$ ($i\in\p$).
For $q\in \mathbb{R}$ we set 
\[
\LB:= \max_{i\in\p}\{-q\alpha_i\}.
\]

\begin{lemma}\label{lem finite pressure on a special parameter}
    For all $(b,q)\in\fin$ and $s\in(\LB,\infty)$ we have $(b,q,s)\in \indufin$. 
   In particular, for all $(b,q)\in\mathcal{N}$ we have $(b,q,p(b,q))\in \indufin$. 
\end{lemma}

\begin{proof}
Let $(b,q)\in\fin$ and let $s\in( \LB,\infty)$.
We take a small $\epsilon>0$ with $ \LB+\epsilon<s$. 
Since $\phi$ is continuous  at each parabolic fixed point $x_i$ ($i\in\p$), there exists $N\geq 2$ such that for all $i\in\p$, $n\geq N$ and $x\in I_{i^n}$ we have $|q\phi(x)-q\alpha_i|<\epsilon$. Thus, by (F) and \eqref{eq finite pressure}, we obtain
\begin{align*}
&
\sum_{n=N}^{\infty}\sum_{i\in \p}\sum_{j\in\edge_i}
e^{\left(-q\indupote-b\log|\indumap'|- s \return\right)\circ\codingmap([ji^n\edge_i])}
\asymp
\sum_{n=N}^{\infty}\sum_{i\in \p}\sum_{j\in\edge_i}
{e^{(-q\phi-b\log |f'|)\circ\codingmap([ji^n\edge_i])}}
\\&\frac{\exp
\left(
\left(
\sum_{k=1}^{n-N+1}(-q\phi- s )\circ\codingmap\circ\shift^k
\right)
([ji^n\edge_i])
\right)}{
n^{b(1+\exponent)}}
e^{\left(\sum_{k=n-N+2}^{n-1}(-q\phi- s )\circ\codingmap\circ\shift^k\right)([ji^n\edge_i])}
\\&\ll
\sum_{n=N}^{\infty}\frac{
e^{( \LB+\epsilon- s )n
}}{n^{b(1+\exponent)}}
\sum_{j\in\edge}
e^{(-q\phi-b\log|f'|)\circ\oricodingmap([j])}
<\infty.
\end{align*}
Using a similar calculation, we also obtain 
$\sum_{\omega\in \bigcup_{n=1}^{N-1}E_n}e^{\left(-q\indupote-b\log|\indumap'|- s \return\right)\circ\codingmap([\omega])}<\infty$.
Therefore, by Proposition \ref{prop finite pressure}, the proof is complete.
\end{proof}

Define, for $(b,q)\in \mathbb{R}^2$,
\[
\indumultipote_{b,q}:=(-q\indupote-b\log|\indumap'|-p(b,q)\return)\circ\codingmap.
\]
For $(b,q)\in \fin$ with $p(b,q)>\LB$ we write $\codingmeasure_{b,q}:=\codingmeasure_{b,q,p(b,q)}$.

\begin{lemma}\label{lemma finite return time}
    For all $(b,q)\in\fin$ with $p(b,q)>\LB$ and $l\in\mathbb{N}$ we have $\codingmeasure_{b,q}(| \return\circ\codingmap|^l)<\infty$. 
    Moreover, if $(b,q)\in \mathcal{N}$ we have
    $\codingmeasure_{b,q}( |\indupote\circ\codingmap|^{l})<\infty$ and $\codingmeasure_{b,q}(|\log|\indumap'\circ\codingmap||^l)<\infty$. In particular, $\codingmeasure_{b,q}$ is the unique equilibrium measure for $\indumultipote_{b,q}$.
\end{lemma}
\begin{proof}
For $(b_0,q_0,s_0)\in \mathbb{R}^3$ we set
\[
S_{b_0,q_0,s_0}:= \sum_{n=1}^{\infty}
   \sum_{i\in \p}
   {e^{
   \left(
   \sum_{k=0}^{n-2}(-q_0\phi-s_0)\circ f^{k}
   \right)\circ\oricodingmap([i^n])}}{n^{-b_0(1+\exponent)}}.
\]
    Let $(b,q)\in\fin$ satisfy $p(b,q)>\LB$ and let $l\in\mathbb{N}$.
     We take a small $\epsilon>0$ such that $\LB+\epsilon<p(b,q)$. 
By using \eqref{eq gibbs}, (F), (P) and the limit $\lim_{x\to\infty}x^l e^{-\epsilon x}=0$, we obtain
\begin{align*}
&   \codingmeasure_{b,q}( |\indupote\circ\codingmap|^{l})
   \ll\sum_{\omega\in \edgesindu}
  |\indupote|^l\circ\codingmap([\omega])
  e^{\indumultipote_{b,q}([\omega])}
   \ll \sum_{n=2}^\infty 
   \sum_{\omega\in\edgepart_n}
   \left|
   \sum_{k=1}^{n-1}
   \phi\circ f^k
   +\phi
   \right|^l\circ\codingmap([\omega])
\\&   e^{(-q\phi-b\log|f'|)\circ\codingmap([\omega])}
   \frac{e^{
   \left(
   \sum_{k=1}^{n-1}(-q\phi-p(b,q))\circ f^{k}
   \right)\circ\codingmap([\omega])
   }}{n^{b(1+\exponent)}}
   \ll 
   \sum_{n=2}^\infty 
   \sum_{\omega\in\edgepart_n}
   |\phi|^l\circ\codingmap([\omega])
   \\&e^{(-q\phi-b\log|f'|)\circ\codingmap([\omega])}
   \left|
   \sum_{k=1}^{n-1}
   \phi\circ f^k
   \right|^l
   \circ\codingmap([\omega])\frac{e^{
   \left(
   \sum_{k=1}^{n-1}(-q\phi-p(b,q))\circ f^{k}
   \right)\circ\codingmap([\omega])
   }}{n^{b(1+\exponent)}}
   \\&\ll
   S_{b,q-\epsilon,p(b,q)}
   \sum_{i\in \edge}
   e^{(-(q-\epsilon)\phi-b\log|f'|)\circ\oricodingmap([i])}.
\end{align*}
By using the same calculation, we obtain
\begin{align*}
    &\codingmeasure_{b,q}(|\log|\indumap'||^l)\ll
    S_{b-\epsilon,q,p(b,q)}
   \sum_{i\in \edge}
   e^{(-q\phi-(b-\epsilon)\log|f'|)\circ\oricodingmap([i])}
   \\&
   \codingmeasure_{b,q}(|\return\circ\codingmap|^l)\ll
   S_{b,q,p(b,q)-\epsilon}
   \sum_{i\in \edge}
   e^{(-q\phi-b\log|f'|)\circ\oricodingmap([i])}.
\end{align*}
By the same argument as in the proof of Lemma \ref{lem finite pressure on a special parameter}, one can show that 
$S_{b-\epsilon,q,p(b,q)}<\infty$, $S_{b,q-\epsilon,p(b,q)}<\infty$ and $S_{b,q,p(b,q)-\epsilon}<\infty$. 
On the other hand, by \eqref{eq finite pressure}, we have $\sum_{i\in \edge}
   e^{(-q\phi-b\log|f'|)\circ\oricodingmap([i])}<\infty$.
Therefore,    
   $\codingmeasure_{b,q}(|\return\circ\codingmap|^l)<\infty$.
Furthermore, if $(b,q)\in \mathcal{N}$ then, by taking $\epsilon>0$ smaller if necessary, we may assume that $B((b,q),\epsilon)\subset \mathcal{N}$, where $B((b,q),\epsilon)$ denotes the open ball centered at $(b,q)$ with radius $\epsilon$.
Hence, \eqref{eq finite pressure} yields that $\sum_{i\in \edge}
   e^{(-(q-\epsilon)\phi-(b-\epsilon)\log|f'|)\circ\oricodingmap([i])}<\infty$ and thus,     $\codingmeasure_{b,q}( |\indupote\circ\codingmap|^{l})<\infty$ and $\codingmeasure_{b,q}(|\log|\indumap'||^l)<\infty$ for $(b,q)\in \mathcal{N}$.
\end{proof}

For $(b,q)\in \fin$ with $p(b,q)>\LB$ we define the measures $\measure_{b,q}:=\codingmeasure_{b,q}\circ\codingmap^{-1}$ on $\indulimit$ and 
$
\mu_{b,q}:=(\measure_{b,q}(\return ))^{-1}\sum_{n=0}^{\infty}\sum_{k=n+1}^\infty\measure_{b,q}|_{\{\return=k\}}\circ f^{-n}$ on $\Lambda$.    
Then, by Lemma \ref{lemma finite return time}, Remark \ref{rem measurable bijection} and \eqref{eq Abramov-Kac's formula}, for all $(b,q)\in\mathcal{N}$ we obtain
\begin{align}\label{eq induced pressure is less than zero}
&\indupressure(b,q,p(b,q))=h(\codingmeasure_{b,q})+\codingmeasure_{b,q}(\indumultipote_{b,q} )
    \\&
    \nonumber
    =\measure_{b,q}(\return)
    \left(
    h(\mu_{b,q})+\mu_{b,q}(-q\phi-b\log|f'|)-p(b,q)
    \right) \leq 0.
\end{align}
The proofs of the following two theorems follow from similar arguments as in the finitely many branches case (see the proofs of \cite[Theorem 3.3 and Theorem 3.4]{arimanonuniformly}). For the convenience of the reader we include  proofs for these theorems in Appendix.

\begin{thm}\label{thm uniquness and existence of the equilibrium state}
     For $(b,q)\in \mathcal{N}$ we have that $\indupressure(b,q,p(b,q))=0$. Furthermore,  for $(b,q)\in \mathcal{N}$,  $\mu_{b,q}$ is the unique equilibrium measure for $-q\phi-b\log |f'|$.
\end{thm}

\begin{thm}\label{thm regularity of non induced pressure}
The pressure function $(b,q)\mapsto p(b,q)$ is real-analytic on $\mathcal{N}$ and for $(b,q)\in \mathcal{N}$ we have
    \begin{align}\label{eq ruell's formula noninduced}
    \frac{\partial}{\partial b}p(b,q)=-\lambda(\mu_{b,q}) \text{ and } \frac{\partial}{\partial q}p(b,q)=-\mu_{b,q}( \phi ). 
    \end{align} 
    Moreover, for $(b,q)\in \mathcal{N}$ we have $\frac{\partial^2}{\partial q^2}p(b,q)=0$ if and only if $\alpha_{\inf}=\alpha_{\sup}$.  
\end{thm}

\subsection{Convergence of equilibrium measures}\label{sec convergence}
In this section, we discuss when the sequence of equilibrium measures obtained in the previous section is tight. We also provide sufficient conditions under which a weak* limit of this sequence is again an equilibrium measure.

For $(b,q)\in \mathbb{R}^2$ we define $\multipote_{b,q}:=(-q\phi-b\log|f'|)\circ\oricodingmap$ and 
\[
C_{b,q}:=
\sum_{\tilde\omega\in \edgesindu}
   e^{\indumultipote_{b,q}([\tilde \omega])}
\]
Since $\inducoding$ is finitely primitive, for every bounded set $C\subset\mathbb{R}^2$ we have, for all $(b,q)\in C$,
\begin{align}\label{eq comparablity pressure and sum}
    e^{\indupressure(b,q,p(b,q))}\asymp C_{b,q}.
\end{align}

\begin{rem}\label{rem culuculation gibbs}
    Let $(b,q)$ be in $\fin$ with $p(b,q)>\LB$. 
    Then, since $\codingmeasure_{b,q}$ is a Gibbs measure, for every countable set $K\subset \Lambda$ and $\tilde K\subset\indulimit$ we have  $\mu_{b,q}(K)=0$ and $\measure_{b,q}(\tilde K)=0$. 
In particular, for measurable sets $A,B\subset \Lambda$ such that $A\triangle B:=A\setminus B\cup B\setminus A$ is countable,  
we have $\mu_{b,q}(A)=\mu_{b,q}(B)$. Also, since $\codingmap$ is one-to-one except on the countable set $J_0$, for all $n\in \mathbb{N}$ and $Z\subset  \inducoding^n$ we have $\measure_{b,q}(\bigcup_{\tilde\omega\in Z}\codingmap([\tilde \omega]))=\sum_{\tilde \omega\in Z}\measure_{b,q}(\codingmap([\tilde \omega]))$. On the other hand, by the definition of $\mu_{b,q}$, for every measurable set $\tilde B\subset \indulimit$ we have $\mu_{b,q}(\tilde B)=\measure_{b,q}(\return)^{-1}\measure_{b,q}(\tilde B)$ (see \cite[Corollary 1.4.4]{viana}). 
\end{rem}

Note that the coding map $\pi$ is continuous and one-to-one except on the preimage of the countable set $\bigcup_{n=0}^\infty f^{-n}(\bigcup_{e\in \edge}\partial \Delta_e)$, where it is at most two-to-one. Furthermore, the restriction of $\pi$ to $\Sigma\setminus\pi^{-1}(\bigcup_{n=0}^\infty f^{-n}(\bigcup_{e\in \edge}\partial \Delta_e))$ has a continuous inverse. Thus, by Remark \ref{rem culuculation gibbs}, for each $(b,q)\in\fin$ with $p(b,q)>\LB$ we can define the Borel probability measure $\mu_{b,q}'$ on $\Sigma$ by setting 
\[
\mu_{b,q}'(B):=\mu_{b,q}(\pi(B)) \text{ for every Borel measurable set }B\subset \Sigma.
\]
Then since for each $(b,q)\in\fin$ with $p(b,q)>\LB$ we have 
\begin{align}\label{eq gibbs boundary 0}
\mu_{b,q}\left(\bigcup_{n=0}^\infty f^{-n}\left(\bigcup_{e\in \edge}\partial \Delta_e\right)\right)=0,
\end{align}
for each $(b,q)\in\fin$ with $p(b,q)>\LB$ we obtain 
\begin{align}
\label{eq special lift}
    \mu_{b,q}=\mu_{b,q}'\circ\pi^{-1} \text{ and } h(\mu_{b,q})=h(\mu_{b,q}').
\end{align}

Define 
$
O_{\p}:=\bigcup_{k=0}^{\infty}\bigcup_{i\in\p} f^{-k}(x_i).
$ 
 Note that $O_{\p}$ is a countable set.
\begin{lemma}\label{lemma tight}
Let $\{(b_n,q_n)\}_{n\in\mathbb{N}}\subset\mathbb{R}^2$ be a bounded sequence satisfying the following conditions:
\begin{itemize}
    \item[(T1)] For all $\epsilon >0$ there exists $F\subset \h$ such that $\edge\setminus F$ is finite and for all $n\in\mathbb{N}$ we have 
    \[\sum_{m\in F}e^{\multipote_{b_n,q_n}([m])}<\epsilon.\]
\item[(T2)] We have $\sup_{n\in\mathbb{N}}|\indupressure(b_n,q_n,p(b_n,q_n))|<\infty$.
    \item[(T3)] For all $n\in\mathbb{N}$ we have $p(b_n,q_n)>\text{LB}(q_n)$.
\end{itemize}
Then, $\{\mu_{b_n,q_n}'\}_{n\in\mathbb{N}}$ is tight.
     
\end{lemma}
\begin{proof}
Let $\{(b_n,q_n)\}_{n\in\mathbb{N}}\subset\mathbb{R}^2$ be a bounded sequence satisfying the conditions (T1), (T2) and (T3).
For each $n\in\mathbb{N}$ we set 
\begin{align*}
    &C_{1,n}:=e^{-\indupressure(b_n,q_n,p(b_n,q_n))}C_{b_n,q_n},
    \ 
    C_{2,n}:=e^{-2\indupressure(b_n,q_n,p(b_n,q_n))}C_{b_n,q_n}^2
    \text{ and }
    \\& C_{l,n}:=(l-2)\sum_{L=1}^{l-1}e^{-(L+1)\indupressure(b_n,q_n,p(b_n,q_n))}C_{b_n,q_n}^{L+1} \text{ for }l\geq 3.
\end{align*}
By (T2) and \eqref{eq comparablity pressure and sum}, for all $l\in\mathbb{N}$ we have $C_l:=\sup_{n\in\mathbb{N}}C_{l,n}<\infty$.

Let $m\in \h$. Notice that 
\begin{align*}
&\oricodingmap([m])\subset
\codingmap([m\h])\cup\bigcup_{i\in\p}\bigcup_{k=1}^{\infty}\codingmap\left([mi^k\edge_i]\right)
\cup O_{\p}
\\&\bigcup_{\omega\in \edge}\oricodingmap([\omega m])\subset
\bigcup_{\omega\in \edge}\codingmap([\omega\h m\h])
\cup
\bigcup_{\omega\in \edge}\bigcup_{i\in\p}\bigcup_{k\in\mathbb{N}}\codingmap([\omega\h mi^k\edge_i])\cup O_{\p}. 
\end{align*}
Moreover, for all $q\in \mathbb{N}$ and $F\subset\edge$ we have 
\[
\bigcup_{j\in F}
\bigcup_{\tau\in \oricodingsp^{q}_j}\oricodingmap([j\tau m])
\subset 
\bigcup_{L=1}^{q+1}\bigcup_{\tilde \omega\in \inducoding^{L}}\left(\codingmap([\tilde\omega m\h])\cup\bigcup_{i\in\p}\bigcup_{k\in\mathbb{N}}\codingmap([\tilde \omega mi^k\edge_i])\right)\cup O_{\p},
\]
where $\oricodingsp^q_j:=\oricodingsp^q$ if $j\in\h$ and $\oricodingsp^q_j:=\{\omega\in \oricodingsp^q:\omega_1\neq j\}$ otherwise. Thus, by \eqref{eq gibbs} and Remark \ref{rem culuculation gibbs}, for all $n\in\mathbb{N}$ we obtain
\begin{align}\label{eq proof tight 1}
    \mu_{b_n,q_n}(\oricodingmap([m]))\leq \measure_{b_n,q_n}(\oricodingmap([m]))\ll C_{1} e^{\multipote_{b_n,q_n}([m])},
\end{align}
\begin{align}\label{eq proof tight 2}
   &\mu_{b_n,q_n}(\oricodingmap(\{\omega\in \oricodingsp:\omega_2=m\}))
   \ll C_{2}
   e^{\multipote_{b_n,q_n}([m])}.
\end{align}
and, for all $q\in\mathbb{N}$  and $F\subset\edge$,
\begin{align}\label{eq proof tight basic}
    \sum_{j\in F}\sum_{\tau\in \oricodingsp_j^q} \measure_{b_n,q_n}(\oricodingmap(j\tau m))
    \ll
       e^{\multipote_{b_n,q_n}([m])}
    \sum_{L=1}^{q+1}e^{-(L+1)\indupressure(b_n,q_n,p(b_n,q_n))}
    C_{b_n,q_n}^{L+1}
\end{align}
Let $l\geq 3$. 
 For all $n\in\mathbb{N}$ we have
     \begin{align}\label{eq proof tight decomposition}
         &\mu_{b_n,q_n}(\oricodingmap(\{\omega\in \oricodingsp:\omega_l=m\}))=\sum_{\omega\in \oricodingsp^{l-1}}\mu_{b_n,q_n}(\oricodingmap([\omega m]))
         \\&=\sum_{i\in\h}
         \sum_{\tau\in \oricodingsp^{l-2}}\mu_{b_n,q_n}(\oricodingmap([i\tau m]))+
         \sum_{k=1}^{l-1}
         \sum_{i\in\p}
         \sum_{
         \tau\in \oricodingsp^{l-1-k}_i}
         \mu_{b_n,q_n}(\oricodingmap([i^k\tau m])) \nonumber.
     \end{align}
     We note here that, by the definition of $\mu_{b_n,q_n}$ ($n\in\mathbb{N}$), for all $n\in\mathbb{N}$ and $\omega\in \oricodingsp^*$ we have
     \[
     \mu_{b_n,q_n}(\oricodingmap([\omega]))=
     \frac{1}{\measure_{b_n,q_n}(\return)}
     \sum_{q=0}^{\infty}\sum_{p=q+1}^{\infty}
     \measure_{b_n,q_n}
     \left(
     \left(\bigcup_{a\in \oricodingsp^q}\oricodingmap([a \omega])
     \right)
     \cap
     \left(
     \bigcup_{\tilde a\in E_p}\codingmap([\tilde a])
     \right)
     \right).
     \]
Since $m\in \h$, for all $i\in\edge$, $\tau\in \oricodingsp^{l-2}_i$,  $p\geq l-1$ and $\tilde a\in E_p$ we have $\oricodingmap([i\tau m])\cap \codingmap([\tilde a])\subset O_{\p}$, 
and for all $q\in\mathbb{N}$, $p\geq q+1$, $a\in \oricodingsp^q$ and $\tilde a\in E_p$ we have $\oricodingmap([ai\tau m])\cap \codingmap([\tilde a])\subset O_{\p}$. Therefore, by the above formula for $\mu_{b_n,q_n}$, for all $n\in\mathbb{N}$, $i\in\edge$ and $\tau\in \oricodingsp^{l-2}_i$ we obtain
   $
   \mu_{b_n,q_n}(\oricodingmap([i\tau m]))
   \leq (l-2)\measure_{b_n,q_n}(\oricodingmap([i\tau m])).
   $
Combining this with \eqref{eq proof tight basic}, for all $n\in\mathbb{N}$ and $l\geq 3$ we obtain
\begin{align}\label{eq proof tight l-H}
    &\sum_{i\in\h}\sum_{\tau\in \oricodingsp^{l-2}}\mu_{b_n,q_n}(\oricodingmap([i\tau m]))
    \\&\ll (l-2)       e^{\multipote_{b_n,q_n}([m])}\sum_{L=1}^{l-1}e^{-(L+1)\indupressure(b_n,q_n,p(b_n,q_n))}
    C_{b_n,q_n}^{L+1} \nonumber \text{ and }
\end{align}
\begin{align}\label{eq proof tight l-I}
    &\sum_{i\in \p}\sum_{\tau\in \oricodingsp^{l-2}_i}\mu_{b_n,q_n}(\oricodingmap([i\tau m]))
    \\&\ll (l-2)       e^{\multipote_{b_n,q_n}([m])}\sum_{L=1}^{l-1}e^{-(L+1)\indupressure(b_n,q_n,p(b_n,q_n))}
    C_{b_n,q_n}^{L+1} \nonumber.
\end{align}
Also, since for all $2\leq k\leq l-1$, $i\in\p$, $\tau\in \oricodingsp_i^{l-1-k}$, $p\geq 1$ and $\tilde a\in E_p$ we have $\oricodingmap([i^k\tau m])\cap\codingmap([\tilde a])\subset O_{\p}$, for all $n\in\mathbb{N}$, $2\leq k\leq l-1$, $i\in\p$ and  $\tau\in \oricodingsp_i^{l-1-k}$ we obtain
 $   \mu_{b_n,q_n}(\pi([i^k\tau m]))
    \leq \sum_{q=1}^{\infty}\sum_{a\in \edge_i}\measure_{b_n,q_n}(\oricodingmap([ai^{q-1+k}\tau m])).
$
Therefore, since for $2\leq k\leq l-1$ we have
\begin{align*}
    &\bigcup_{i\in\p}\bigcup_{\tau\in \oricodingsp^{l-1-k}_i}\bigcup_{q=1}^{\infty}\bigcup_{a\in \edge_i}\oricodingmap([ai^{q-1+k}\tau m])
    \\&\subset \bigcup_{L=1}^{l-1-k+2}\bigcup_{\tilde \omega\in \inducoding^L}\left(\codingmap([\tilde\omega m\h])\cup\bigcup_{e\in\p}\bigcup_{\ell=1}^\infty\codingmap([\tilde \omega me^\ell\edge_e])\right)\cup O_{\p}
\end{align*}
and $\codingmeasure_{b_n,q_n}$ is the Gibbs measure for $\indumultipote_{b_n,q_n}$, for all $n\in\mathbb{N}$ and $l\geq 3$ we have 
\begin{align*}
         &\sum_{k=2}^{l-1}
         \sum_{i\in\p}
         \sum_{
         \tau\in \oricodingsp^{l-1-k}_i}
         \mu_{b_n,q_n}(\oricodingmap([i^k\tau m])) 
         \\&\nonumber\leq 
         \sum_{k=2}^{l-1}
         \sum_{i\in\p}
         \sum_{
         \tau\in \oricodingsp^{l-1-k}_i}
         \sum_{q=1}^{\infty}\sum_{a\in \edge_i}
         \measure_{b_n,q_n}(\oricodingmap([ai^{q-1+k}\tau m]))
         \\&\nonumber\leq 
         \sum_{k=2}^{l-1}
         \sum_{L=1}^{l-1-k+2}\sum_{\tilde \omega\in \inducoding^L}
         \left(
         \measure_{b_n,q_n}(\codingmap([\tilde \omega m\h]))+\sum_{e\in\p}\sum_{\ell=1}^{\infty}\measure_{b_n,q_n}(\codingmap([\tilde \omega me^{\ell}\edge_e]))
         \right)
         \\&\nonumber
         \ll
         e^{\multipote_{b_n,q_n}([m])}
         \sum_{k=2}^{l-1}
         \sum_{L=1}^{l-1-k+2}e^{-(L+1)\indupressure(b_n,q_n,p(b_n,q_n))}C_{b_n,q_n}^{L+1}.
\end{align*}
Combining this with \eqref{eq proof tight decomposition}, \eqref{eq proof tight l-H} and \eqref{eq proof tight l-I}, for all $n\in\mathbb{N}$ and $l\geq 3$ we obtain 
\begin{align}\label{eq proof tight l 3}
    \mu_{b_n,q_n}(\oricodingmap(\{\omega\in \oricodingsp:\omega_l=m\}))
    \ll C_{l}e^{\multipote_{b_n,q_n}([m])}.
\end{align}

Let $\epsilon>0$. By (T1), \eqref{eq proof tight 1}, \eqref{eq proof tight 2} and \eqref{eq proof tight l 3}, for all $l\in\mathbb{N}$ there exists $F_l\subset\h$ such that $\edge\setminus F_l$ is finite and for all $n\in\mathbb{N}$ we have 
$
\sum_{m\in F_l}\mu_{b_n,q_n}(\oricodingmap(\{\omega\in \oricodingsp:\omega_l=m\}))\leq {\epsilon}/{2^l}.
$
Therefore, for all $n\in\mathbb{N}$ we obtain 
\[
\mu_{b_n,q_n}\left(\oricodingmap\left(\oricodingsp\cap\prod_{l=1}^{\infty}\edge\setminus F_l\right)\right)\geq 1-\sum_{l=1}^\infty\sum_{m\in F_l}\mu_{b_n,q_n}(\oricodingmap(\{\omega\in \oricodingsp:\omega_l=m\}))\geq 1-\epsilon.
\]
This implies that $\{\mu_{b_n,q_n}'\}_{n\in\mathbb{N}}$ is tight.
\end{proof}

We recall here the definition of the measure-theoretic entropy $h(\mu')$ for $\mu'\in M(\sigma)$ (see \cite[Chapter 4]{walters2000introduction} for details) and the result in \cite{takahasi}. Let $\mathscr{C}=\{C_1,C_2,\cdots\}$ be a countable partition of $\oricodingsp$ into Borel sets and let $\mu'\in M(\sigma)$. The entropy of $\mathscr{C}$ with respect to $\mu'$ is defined by 
\[
H(\mu',\mathscr{C}):=-\sum_{k}\mu'(C_k)\log \mu'(C_k)
\]
with the convention $0\log 0=0$. If $H(\mu',\mathscr{C})<\infty$ then we define 
\[
h(\mu',\mathscr{C})=\lim_{n\to\infty}\frac{1}{n} H\left(\mu',\bigvee_{i=0}^{n-1}\sigma^{-i}\mathscr{C}\right),
\]
where $\bigvee$ denotes the join of the partitions $\sigma^{-i}\mathscr{C}$. The entropy with respect to $\mu'$ is defined by $h(\mu'):=\sup_{\mathscr{C}} h(\mu',\mathscr{C})$, where the supremum is taken over all countable partitions $\mathscr{C}$ with $H(\mu',\mathscr{C})<\infty$. For each $\ell\in \mathbb{N}$ we define two partitions of $\oricodingsp$: 
\[
\mathscr{A}_\ell=\left\{[1],\cdots,[\ell], \bigcup_{k=\ell+1}^\infty[k]\right\} \text{ and }
\mathscr{B}_\ell:=\left\{\bigcup_{k=1}^{\ell}[k],[\ell+1],[\ell+2],\cdots\right\}.
\]
By \cite[Lemma 2.1]{takahasi}, $h(\mu')=\infty$ if and only if $\lim_{\ell\to\infty}h(\mu',\mathscr{A}_\ell)=\infty$.
Since $\mathscr{A}_\ell\vee\mathscr{B}_\ell$ ($\ell\in\mathbb{N}$) is a generator, if $h(\mu')<\infty$ then for all $\ell\in\mathbb{N}$ we have  
\begin{align}\label{eq generator}
    h(\mu')=h(\mu',\mathscr{A}_\ell\vee \mathscr{B}_\ell).
\end{align}
We have the following:
\begin{lemma}{\cite[Lemma 2.2]{takahasi}}\label{lemma entropy}
    For all $\mu'\in M(\sigma)$ with $h(\mu')<\infty$ we have $\lim_{\ell\to\infty}h(\mu',\mathscr{A}_\ell)=h(\mu')$.
\end{lemma}

For a bounded sequence $\{(b_n,q_n)\}_{n\in\mathbb{N}}\subset\mathbb{R}^2$ we consider the following conditions:
\begin{itemize}
    \item[(T1.1)] For all $\epsilon>0$  there exists $F\subset \h$ such that $\edge\setminus F$ is finite and for all $n\in\mathbb{N}$, 
    \begin{align*}
    \sum_{m\in F}e^{\multipote_{b_n,q_n}([m])}|\multipote_{b_n,q_n}|([m])<\epsilon.
    \end{align*}
    \item[(T1.2)] For all $\epsilon>0$  there exists $F\subset \h$ such that $\edge\setminus F$ is finite and for all $n\in\mathbb{N}$,
    \begin{align*}
    \sum_{m\in F}e^{\multipote_{b_n,q_n}([m])}\phi\circ\pi([m])<\epsilon.
    \end{align*}
    \item[(T1.3)]For all $\epsilon>0$  there exists $F\subset \h$ such that $\edge\setminus F$ is finite and for all $n\in\mathbb{N}$,
    \begin{align*}
    \sum_{m\in F}e^{\multipote_{b_n,q_n}([m])}\log|f'|\circ\pi([m])<\epsilon.
    \end{align*}
\end{itemize}
Note that since $\lim_{m\to\infty}\log|f'|\circ\oricodingmap([m])=\infty$ if a bounded sequence $\{(b_n,q_n)\}_{n\in\mathbb{N}}$ satisfies (T1.3) then $\{(b_n,q_n)\}_{n\in\mathbb{N}}$ also satisfies (T1) and if $\{(b_n,q_n)\}_{n\in\mathbb{N}}$ satisfies (T1.2) and (T1.3) then $\{(b_n,q_n)\}_{n\in\mathbb{N}}$ satisfies (T1) and (T1.1).

Let $(b_{\infty},q_{\infty})\in \mathbb{R}^2$ and let $\{(b_n,q_n)\}_{n\in\mathbb{N}}\subset \mathbb{R}^2$ be a bounded sequence satisfying (T1), (T2) and (T3) with $\lim_{n\to\infty}(b_n,q_n)=(b_\infty,q_\infty)$.
Then,
by Prohorov's theorem and Lemma \ref{lemma tight},
there exist a subsequence $\{(b_{n_k},q_{n_k})\}_{k\in\mathbb{N}}$ of $\{(b_n,q_n)\}_{n\in\mathbb{N}}$ and 
$(\mu'_{b_\infty,q_\infty})^*\in M(\sigma)$ such that 
\[
\lim_{k\to\infty}\mu_{b_{n_k},q_{n_k}}'=(\mu_{b_\infty,q_\infty}')^*\] in the weak* topology. 
We set 
\[
\mu_{b_\infty,q_{\infty}}^*:=(\mu_{b_\infty,q_\infty}')^*\circ \pi^{-1}.
\]
Then by \eqref{eq special lift} and the continuity of the coding map $\pi$, we obtain 
\[
\lim_{k\to\infty}\mu_{b_{n_k},q_{n_k}}=\mu_{b_\infty,q_\infty}^*\]
in the weak* topology. Moreover, we have 
\begin{align}\label{eq lift limit measure invariant}
    \mu_{b_\infty,q_\infty}^*\in M(f) \text{ and }h(\mu_{b_\infty,q_\infty}^*)=h((\mu_{b_\infty,q_\infty}')^*)
\end{align}
The proof of this claim proceeds as follows: 
Let \[
\mathcal{M}_{\infty}:=(\mu'_{b_\infty,q_\infty})^*\left(B_0\right), \text{ where }B_0:=\pi^{-1}\left(\bigcup_{n=0}^\infty f^{-n}\left(\bigcup_{e\in \edge}\partial \Delta_e\right)\right)
\]
If $\mathcal{M}_{\infty}=0$
then \eqref{eq lift limit measure invariant} is clear. Thus, we assume that $B_0\neq \emptyset$.
Recall that by (NERI1) and (NERI2), 
\begin{align}\label{eq full branch and 1}
&\text{
$\overline{f(\Delta_e)}=[0,1]$ for each $e\in \edge$ and $1$ is the unique accumulation point}
\\&\text{of the set of endpoints of $\{\Delta_e\}_{e\in\edge}$.
}\nonumber
\end{align}
Thus, by (NERI5), if $B_0$ contains a periodic point $\tau_p$ of $\sigma$ then $0=\pi(\tau_p)$, $f(0)=0$ and there exists $\hat e\in \edge$ such that $0\in\overline{\Delta_{\hat e}}$ and $\pi^{-1}(0)=\{\tau_p\}=\{\hat e^{\infty}\},$
where 
$
\hat e^\infty:=\hat e\hat e\hat e\cdots$. In particular, if $\mu'\in M(\sigma)$ satisfies 
$
\mu'\left(B_0\right)=1
$
then 
\begin{align}\label{eq boundary measure}
    \mu'=\delta_{\hat e^\infty} \text{ and } \mu'\circ \pi^{-1}=\delta_0.
\end{align}
Moreover, we have 
\begin{align}\label{eq invariance boundary set}
    \sigma^{-1}\left(B_0\right)=B_0.
\end{align}
Indeed, if $\tau\in B_0$ then there exist $n\in \mathbb{N}\cup\{0\}$ and $e\in \edge$ such that $f^n(\pi(\tau))\in \partial \Delta_e$. If $\pi(\tau)\notin\bigcup_{e\in \edge} \partial \Delta_e$ then $n>0$ and $f^{n-1}(\pi(\sigma(\tau)))=f^{n}(\pi (\tau))\in \partial \Delta_e$. If $\pi(\tau)\in\bigcup_{e\in \edge} \partial \Delta_e$ then by \eqref{eq full branch and 1}, we have $\pi(\sigma(\tau))=0\in \partial \Delta_{\hat e}$. Hence, we obtain $B_0\subset \sigma^{-1}(B_0)$. Conversely, if $\tau\in \sigma^{-1}(B_0)$ then there exist $n\in \mathbb{N}\cup\{0\}$ and $e\in \edge$ such that $f^n(\pi(\sigma(\tau)))\in \partial \Delta_e$. If $\pi(\tau)\notin\bigcup_{e\in \edge} \partial \Delta_e$ then $f^{n+1}(\pi(\tau))=f^{n}(\pi(\sigma(\tau)))\in \partial \Delta_e$. If $\pi(\tau)\in\bigcup_{e\in \edge} \partial \Delta_e$ then $\tau\in B_0$. Hence, we obtain $\sigma^{-1}(B_0)\subset B_0$. This  completes the proof of \eqref{eq invariance boundary set}.
 We obtain the following: 
\begin{itemize}
    \item[(1)] 
Suppose that $0<\mathcal{M}_{\infty}<1$. We set 
\begin{align*}
&\mu_1:={(\mu'_{b_\infty,q_\infty})^*|_{B_0}}/{\mathcal{M}_\infty}
\text{ and }
\mu_2:={(\mu'_{b_\infty,q_\infty})^*|_{\Sigma\setminus B_0}}/({1-\mathcal{M}_\infty})
\end{align*}
By \eqref{eq invariance boundary set}, we obtain $\mu_1\in M(\sigma)$ and $\mu_2\in M(\sigma)$.
Since $\mu_1(B_0)=1$, \eqref{eq boundary measure} implies that  $\mu_1\circ \pi^{-1}=\delta_0\in M(f)$ and $h(\mu_1\circ \pi^{-1})=0=h(\mu_1)$.
Since $\mu_2(B_0)=0$, we obtain $\mu_2\circ \pi^{-1}\in M(f)$ and $h(\mu_2\circ \pi^{-1})=h(\mu_2)$. Therefore, $\mu_{b_\infty,q_\infty}^*\in M(f)$ and 
\begin{align*}
h((\mu_{b_\infty,q_\infty}')^*)=(1-\mathcal{M}_\infty){h(\mu_2)}=(1-\mathcal{M}_\infty){h(\mu_2\circ \pi^{-1})}=h(\mu_{b_\infty,q_\infty}^*).
\end{align*}
\item[(2)] Suppose that  $\mathcal{M}_{\infty}=1$. Then we have $(\mu'_{b_\infty,q_\infty})^*=\delta_{\hat e^\infty}$ and $\mu_{b_\infty,q_\infty}^*=\delta_0$. In particular, we obtain $\mu_{b_\infty,q_\infty}^*\in M(f)$ and $h((\mu_{b_\infty,q_\infty}')^*)=0=h(\mu_{b_\infty,q_\infty}^*).$
\end{itemize}
This completes the proof of \eqref{eq lift limit measure invariant}.

For simplicity of notation, for $k\in\mathbb{N}$ we write $\mu_k:=\mu_{b_{n_k},q_{n_k}}$, $\mu_k':=\mu_{b_{n_k},q_{n_k}}'$, 
$\multipote_k:=\multipote_{b_{n_k},q_{n_k}}$
$\indupressure_k:=\indupressure(b_{n_k},q_{n_k},p(b_{n_k},q_{n_k}))$, 
$\indumultipote_k:=\indumultipote_{b_{n_k},q_{n_k}}$, 
$\measure_k:=\measure_{b_{n_k},q_{n_k}}$, $\mu_\infty=\mu_{b_\infty,q_\infty}^*$,
and $\mu_\infty':=(\mu_{b_\infty,q_\infty}')^*$.
In the following two lemmas, we keep the notations introduced in this paragraph. Recall that $\p:=\{1,\cdots,\#\p\}$ and $\edge:=\mathbb{N}$ (see \eqref{eq def edge}).
\begin{lemma}\label{lemma}
We assume that a bounded sequence $\{(b_n,q_n)\}_{n\in\mathbb{N}}$ satisfies (T1), (T2), (T3), (T1.1).
Then, we have    $\limsup_{k\to\infty}h(\mu_k)\leq h(\mu_{\infty})$.
\end{lemma}
\begin{proof}
If $h(\mu_{\infty})=\infty$ then  there is nothing to prove. Thus, we assume that $h(\mu_{\infty})<\infty$.
Let $\epsilon>0$.
Then, by Lemma \ref{lemma entropy} and \eqref{eq lift limit measure invariant}, there exists $L\in \mathbb{N}$ such that for all $\ell\geq L$ we have 
    \begin{align}\label{eq entropy appro}
h(\mu'_\infty,\mathscr{A}_\ell)<h(\mu_\infty)+\epsilon.        
    \end{align}
 Notice that for all $\ell\in \mathbb{N}$, $\mathscr{A}_\ell$ is a finite partition. 
    Thus, by the same argument as in the proof of \cite[Theorem 8.2]{walters2000introduction}, for all $\ell \in\mathbb{N}$ the bounded map $\nu'\in M(\sigma) \mapsto h(\nu',\mathscr{A}_\ell)$ is upper semi-continuous (see also the proof of \cite[Lemma 2.6]{takahasi}). Hence, by \eqref{eq special lift}, \eqref{eq generator}, \cite[Theorem 4.12]{walters2000introduction} and \eqref{eq entropy appro}, for all $\ell\geq L$ we obtain 
    \begin{align*}   
    &\limsup_{k\to\infty}h(\mu_k)
    =
\limsup_{k\to\infty}h(\mu'_k,\mathscr{A}_\ell\vee\mathscr{B}_\ell)
\leq
\limsup_{k\to\infty}h(\mu'_k,\mathscr{A}_\ell)+\limsup_{k\to\infty}h(\mu'_k,\mathscr{B}_\ell)
\\&
\leq h(\mu'_\infty,\mathscr{A}_\ell)+\limsup_{k\to\infty}h(\mu'_k,\mathscr{B}_\ell)
\leq h(\mu_\infty)+\epsilon +\limsup_{k\to\infty}h(\mu'_k,\mathscr{B}_\ell).
    \end{align*}
Therefore, if we can show that there exists $\tilde L> \#\p$ such that for all $k\in\mathbb{N}$ we have 
\begin{align}\label{eq entropy uniform}
    h(\mu_k',\mathscr{B}_{\tilde{L}})<\epsilon
\end{align}
then, letting $\epsilon\to0$, the proof is complete. Thus, we shall show that \eqref{eq entropy uniform}. By \cite[Theorem 4.12]{walters2000introduction} and Remark \ref{rem culuculation gibbs}, for all $\ell\in\mathbb{N}$ with $\ell> \#\p$ and $k\in\mathbb{N}$ we have 
\begin{align}\label{eq entropy first step}
    &h(\mu'_k,\mathscr{B}_\ell)\leq H(\mu'_k,\mathscr{B}_\ell)
=-\mu_k\left(\bigcup_{j=1}^{\ell}\oricodingmap([j])\right)\log\mu_k\left(\bigcup_{j=1}^{\ell}\oricodingmap([j])\right)
    \\&-\measure_k(\return)^{-1}\sum_{j=\ell+1}^\infty \measure_
    k(\oricodingmap([j]))\log\measure_k(\oricodingmap([j]))
    +
    \frac{\log \measure_k(\return)}{\measure_k(\return)}
    \sum_{j=\ell+1}^\infty \measure_k(\oricodingmap([j])).\nonumber
\end{align}
{Note that for each $j\in\mathbb{N}$ with $j> \#\p$ we have 
\begin{align}\label{eq entropy refiment}
\pi([j])\triangle\left(\codingmap([j\h])\cup\bigcup_{i\in \p}\bigcup_{q\in\mathbb{N}}\codingmap([ji^q\edge_{i}])\right)\subset O_{\p}.
\end{align}
Hence, by Remark \ref{rem culuculation gibbs} and the calculation in the proof of \cite[Theorem 4.3]{walters2000introduction}, for each $k\in\mathbb{N}$ and $j\in\mathbb{N}$ with $j> \#\p$ we obtain
\begin{align}\label{eq entropy refiment increase}
 -   \measure_
    k(\oricodingmap([j]))\log\measure_k(\oricodingmap([j]))
   & \leq 
-\measure_k(\codingmap([j\h]))\log\measure_k(\codingmap([j\h]))
\\&-\nonumber
    \sum_{i\in\p}\sum_{s\in\mathbb{N}}\measure_
k(\codingmap([ji^s\edge_i]))\log\measure_k(\codingmap([ji^s\edge_i])).
\end{align}
}
Since $\codingmeasure_{b_{n_k},q_{n_k}}$ ($k\in\mathbb{N}$) is the Gibbs measure for $\indumultipote_{k}$, there exists $C\geq 1$ such that for all $k\in\mathbb{N}$ and $\tilde\tau\in\edgesindu$ we have 
$
C^{-1}\leq {\measure_k(\codingmap([\tilde\tau]))}/{\exp(\indumultipote_{k}([\tilde \tau])-\indupressure_k)}\leq C.
$
Hence, for all $k\in\mathbb{N}$ and  $\ell\in\mathbb{N}$ with $\ell>\#\p$ we obtain 
\begin{align*}
   & -\sum_{j=\ell+1}^{\infty}\left(\measure_k(\codingmap([j\h]))\log\measure_k(\codingmap([j\h]))
+
    \sum_{i\in\p}\sum_{s\in\mathbb{N}}\measure_
k(\codingmap([ji^s\edge_i]))\log\measure_k(\codingmap([ji^s\edge_i]))\right)
\\&\leq (\log C+|\indupressure_k|)
\sum_{j=\ell+1}^{\infty}\measure_k(\oricodingmap([j]))
\\&+Ce^{-\indupressure_k}\sum_{j=\ell+1}^{\infty}\left(e^{\indumultipote_k([j\h])}|\indumultipote_k|([j\h])+
\sum_{i\in\p}\sum_{s\in\mathbb{N}}e^{\indumultipote_k([ji^s\edge_i])}|\indumultipote_k|([ji^s\edge_i])
\right).
\end{align*}
Combining this with \eqref{eq entropy refiment increase} and (T2), for all  $k\in\mathbb{N}$ and  $\ell\in\mathbb{N}$ with $\ell>\#\p$ we obtain
\begin{align}\label{eq entropy secound}
   &-\measure_k(\return)^{-1}\sum_{j=\ell+1}^\infty \measure_
    k(\oricodingmap([j]))\log\measure_k(\oricodingmap([j]))
    \ll \sum_{j=\ell+1}^{\infty}\measure_k(\oricodingmap([j]))
   \\&\nonumber + \sum_{j=\ell+1}^{\infty}e^{\multipote_k([j])}|\multipote_k|([j])+\frac{1}{\measure_k(\return)}
   \sum_{j=\ell+1}^\infty
\sum_{i\in\p}\sum_{s\in\mathbb{N}}{e^{\indumultipote_k([ji^s\edge_i])}}
|\indumultipote_k|([ji^s\edge_i]).
\end{align}
Since $\return\geq1$, we have $\sup_{k\in\mathbb{N}}\log \measure_k(\return)/\measure_k(\return)<\infty$. Thus, by \eqref{eq proof tight 1} and (T1.1), there exists $L'> \#\p$ such that for all $k\in\mathbb{N}$ and $\ell\geq L'$ we have 
\begin{align*}
&\mu_k\left(\bigcup_{j=1}^{\ell}\oricodingmap([j])\right)\log\mu_k\left(\bigcup_{j=1}^{\ell}\oricodingmap([j])\right)<\epsilon,
\ 
\sum_{j=\ell+1}^{\infty}e^{\multipote_k([j])}|\multipote_k|([j])<\epsilon 
\\& \text{ and }\max\left\{1,\sup_{k\in\mathbb{N}}\frac{\log \measure_k(\return)}{\measure_k(\return)}\right\}
\sum_{j=\ell+1}^{\infty}\measure_k(\oricodingmap([j]))<\epsilon.
\end{align*}
Combining this with \eqref{eq entropy first step} and \eqref{eq entropy secound}, for all $k\in\mathbb{N}$ and $\ell\geq L'$  we obtain 
\begin{align*}
  h(\mu_k',\mathscr{B}_{\ell})\ll 4\epsilon+\frac{1}{\measure_k(\return)}
   \sum_{j=\ell+1}^\infty
\sum_{i\in\p}\sum_{s\in\mathbb{N}}{e^{\indumultipote_k([ji^s\edge_i])}}|\indumultipote_k|([ji^s\edge_i])).  
\end{align*}
Hence, if there exists $\tilde L\geq L'$ such that for all $k\in\mathbb{N}$ we have 
\begin{align}\label{eq entropy important}
\frac{1}{\measure_k(\return)}
\sum_{j=\tilde L+1}^\infty
\sum_{i\in\p}\sum_{s\in\mathbb{N}}{e^{\indumultipote_k([ji^s\edge_i])}}|\indumultipote_k|([ji^s\edge_i])<\epsilon    
\end{align}
then we obtain \eqref{eq entropy uniform} and the proof is complete. Notice that for all $j\in\h$, $i\in\p$ and $s\in\mathbb{N}$ we have 
\begin{align}\label{eq entropy shift}
f(\codingmap([ji^s\edge_i]))=\bigcup_{\tau\in\edge_i}\oricodingmap([i^s \tau])\setminus O_{\p}.   
\end{align}
Therefore, since $\sup_{k\in\mathbb{N}}|p(b_{n_k},q_{n_k})|<\infty$ by (T1), for all $k\in\mathbb{N}$, $j\in\h$, $i\in\p$ and $s\geq 2$ we have
\begin{align}\label{eq entropy C}
    &{e^{\indumultipote_k([ji^s\edge_i])}}
    \ll e^{C(k,i,s)}e^{\multipote_k([j])},
\end{align}
where 
\[
C(k,i,s):=
\sup_{x\in \bigcup_{\tau\in\edge_i}\oricodingmap([i^s \tau])
}\left\{\sum_{p=0}^{s-2}(-q_{n_k}\phi-b_{n_k}\log|f'|-p(b_{n_k},q_{n_k}))\circ f^p(x)\right\}.
\]
Moreover, since $\sup_{k\in\mathbb{N}}|p(b_{n_k},q_{n_k})|<\infty$ by (T1), $\{(b_n,q_n)\}_{n\in\mathbb{N}}$ is bounded and $\max_{i\in\p}\sup_{x\in \oricodingmap([i])}(\phi+\log|f'|)(x)<\infty$, for all $k\in \mathbb{N}$, $i\in\p$ and $s\geq 2$ we have $C(k,i,s)\ll s$
which yields that for all $k\in\mathbb{N}$, $j\in\h$, $i\in\p$ and $s\geq 1$, 
\begin{align}\label{eq entropy s}
    &|\indumultipote_k|([ji^s\edge_i]) \ll s+|\multipote_k|([j])
\end{align}
By \eqref{eq entropy C} and \eqref{eq entropy s}, for all $\ell\geq L'$ and $k\in\mathbb{N}$ we obtain 
\begin{align}\label{eq entropy total}
&\sum_{j=\ell+1}^\infty
\sum_{i\in\p}\sum_{s\in\mathbb{N}}{e^{\indumultipote_k([ji^s\edge_i])}}|\indumultipote_k|([ji^s\edge_i])
\\&\nonumber\ll\sum_{j=\ell+1}^{\infty}e^{\multipote_k([j])}\sum_{i\in\p}\sum_{s\in\mathbb{N}} se^{C(k,i,s)}+\sum_{j=\ell+1}^{\infty}e^{\multipote_k([j])}|\multipote_k|([j])\sum_{i\in\p}\sum_{s\in\mathbb{N}} e^{C(k,i,s)}.    
\end{align}
On the other hand, by (T2), \eqref{eq entropy shift} and \eqref{eq gibbs}, for all $k\in\mathbb{N}$ we obtain
\begin{align}\label{eq entropy return}
    &\measure_k(\return)\geq \sum_{j\in \h}\sum_{i\in\p}\sum_{s\in\mathbb{N}}s\measure_k(\codingmap([ji^s\edge_i]))\gg \sum_{j\in \h} 
    \sum_{i\in\p}\sum_{s\in\mathbb{N}} se^{\indumultipote_k([ji^s\edge_i])}
    \\&\nonumber\ge \sum_{j\in \h}e^{\inf_{\tau\in[j]}\multipote_k(\tau)}
    \sum_{i\in\p}\sum_{s\in\mathbb{N}} se^{C(k,i,s)}.
\end{align}
By \eqref{eq entropy total}, this implies that for all $\ell\geq L'$ and $k\in\mathbb{N}$ we have
\begin{align*}
   \frac{1}{\measure_k(\return)} \sum_{j=\ell+1}^\infty
\sum_{i\in\p}\sum_{s\in\mathbb{N}}{e^{\indumultipote_k([ji^s\edge_i])}}|\indumultipote_k|([ji^s\edge_i])\ll\sum_{j=\ell+1}^{\infty}e^{\multipote_k([j])}+\sum_{j=\ell+1}^{\infty}e^{\multipote_k([j])}|\multipote_k|([j]).
\end{align*}
By (T1) and (T1.1), we obtain \eqref{eq entropy important} and the proof is complete.
\end{proof}

\begin{lemma}\label{lemma uniformly integrable}
Let $\{(b_n,q_n)\}_{n\in\mathbb{N}}$ be a bounded sequence satisfying (T1), (T2), (T3).
If the sequence $\{(b_n,q_n)\}_{n\in\mathbb{N}}$ satisfies (T1.2)
then $\lim_{k\to\infty}\mu_{k}(\phi)=\mu_{\infty}(\phi)<\infty$, and if the sequence $\{(b_n,q_n)\}_{n\in\mathbb{N}}$ satisfies (T1.3)
then $\lim_{k\to\infty}\lambda(\mu_k)=\lambda(\mu_{\infty})<\infty$.
\end{lemma}
\begin{proof}
We first show the first half. Assume that the sequence $\{(b_n,q_n)\}_{n\in\mathbb{N}}$ satisfies (T1.2). 
    By (T1), (T2), \eqref{eq entropy refiment} and \eqref{eq gibbs}, for all $\ell\in\mathbb{N}$ with $\ell>\#\p$ and $k\in\mathbb{N}$ we have 
  \begin{align}\label{eq integral converge}
    &\sum_{j=\ell}^\infty\int _{\oricodingmap([j])} \phi d\mu_k
    \leq 
    \sum_{j=\ell}^{\infty}\phi \circ\oricodingmap([j])
    \left(\frac{\measure_k([j\h])}{\measure_k(\return)}
    +
    \sum_{i\in \p}\sum_{s\in\mathbb{N}}\frac{\measure_k([ji^s\edge_i])}{\measure_k(\return)}
    \right)
    \\&\nonumber\ll\sum_{j=\ell}^{\infty}\phi \circ\oricodingmap([j])
    \left(e^{\multipote_k([j])}
    +
    \sum_{i\in \p}\sum_{s\in\mathbb{N}}\frac{e^{\indumultipote_k([ji^s\edge_i])}}{\measure_k(\return)}
    \right)
  \end{align}
  By \eqref{eq entropy C} and \eqref{eq entropy return}, for all $\ell\in\mathbb{N}$ with $\ell>\#\p$, $j\geq \ell$ and $k\in\mathbb{N}$ we have that
  $
  \sum_{i\in \p}\sum_{s\in\mathbb{N}}{e^{\indumultipote_k([ji^s\edge_i])}}/{\measure_k(\return)}\ll e^{\multipote_k([j])}. 
  $
  Therefore, by \eqref{eq integral converge} and (T1.2), there exists $L>\#\p$ such that for all $\ell \geq L$ and $k\in\mathbb{N}$ we have
  \begin{align}\label{eq convergence uniform}
      \sum_{j=\ell}^{\infty}\int_{\oricodingmap([j])}\phi d\mu_k<\epsilon.
  \end{align}
  For $C\subset \Lambda$ and $C'\subset \Sigma$ we denote by $1_C$ the characteristic function with respect to $C$ and by $1_{C'}$ the characteristic function with respect to $C'$. Note that $\lim_{k\to\infty} \mu'_k=\mu_\infty'$ and for all $\ell\in\mathbb{N}$ the function $\phi\circ\pi\cdot1_{\bigcup_{j=1}^\ell[j]}$ is bounded and continuous. Thus, by \eqref{eq gibbs boundary 0}, \eqref{eq special lift} and the argument as in the proof of \eqref{eq lift limit measure invariant}, 
   for all $\ell\in\mathbb{N}$ we obtain
  \begin{align}\label{eq convergence integral weak}
&\lim_{k\to\infty}\mu_k\left(\phi\cdot1_{\bigcup_{j=1}^\ell\oricodingmap([j])}\right)
=\lim_{k\to\infty}\mu_k'\left(\phi\circ \pi\cdot1_{\bigcup_{j=1}^\ell[j]}\right)
\\&\nonumber=\mu_\infty'\left(\phi\circ \pi\cdot1_{\bigcup_{j=1}^\ell[j]}
\right)
=
\mu_\infty\left(\phi\cdot1_{\bigcup_{j=1}^\ell\oricodingmap([j])}
\right).
  \end{align}
  Combining this with \eqref{eq convergence uniform} we obtain 
  $\liminf_{k\to\infty}\mu_k(\phi)<\infty$. 
  Moreover, since $\phi\circ \pi$ is bounded from below, the Portmanteau theorem and \eqref{eq special lift} yield  $\mu_\infty(\phi)=\mu_{\infty}'(\phi\circ\pi)\leq \liminf_{k\to\infty}\mu_k'(\phi\circ \pi)\leq \liminf_{k\to\infty}\mu_k(\phi)<\infty$
  and thus, there exists $L'\geq L$ such that for all $\ell \geq L'$ we have
      $\sum_{j=\ell}^{\infty}\int_{\oricodingmap([j])}\phi d\mu_\infty<\epsilon.$
  Hence, 
  by \eqref{eq convergence uniform} and \eqref{eq convergence integral weak}, the proof of the first part is complete.
A similar argument shows the second part.
  \end{proof}

\begin{thm}\label{thm convergence of the equilibrium state}
    Let $(b_\infty,q_\infty)\in \mathbb{R}^2$. Assume that there exists a bounded sequence $\{(b_n,q_n)\}_{n\in\mathbb{N}}\subset \mathcal{N}$ satisfying (T1.2), (T1.3) and (T2) with $\lim_{n\to\infty}(b_n,q_n)=(b_\infty,q_\infty)$. Then, the limit measure $\mu_{\infty}$ obtained as above is an equilibrium measure for $-q_\infty\phi-b_\infty\log|f'|$.
\end{thm}
\begin{proof}
    Note that (T1.2) and (T1.3) imply (T1) and (T1.1) and the assumption $\{(b_n,q_n)\}_{n\in\mathbb{N}}\subset \mathcal{N}$ implies (T3).
    We first show $p(b_\infty,q_\infty)\leq\liminf_{n\to\infty} p(b_n,q_n)$.
    Let  $\epsilon>0$. 
    Combining (T1) and \eqref{eq finite pressure}, we obtain $(b_\infty,q_\infty)\in\fin$. Thus, by Theorem \ref{thm compact approximation}, there exists a finite set $F\subset \edge$ such that $|p_F(b_\infty,q_\infty)-p(b_\infty,q_\infty)|<\epsilon/2$. Here, for $n\in\mathbb{N}\cup\{\infty\}$ we set $p_F(b_n,q_n):=(P_\sigma)_F(-q_n\phi\circ \pi-b_n\log |f'\circ\pi|)$. Note that since $F$ is finite, we have 
    \[
    \lim_{n\to\infty}p_F(b_n,q_n)=p_F(b_\infty,q_\infty) \text{ and }p_F(b_n,q_n)\leq p(b_n,q_n).
    \] 
    In particular, there exists $N\geq 1$ such that for all $n\geq N$ we have $|p_F(b_n,q_n)-p_F(b_\infty,q_\infty)|<\epsilon/2$. Therefore, for all $n\geq N$ we obtain 
    \begin{align*}
        p(b_\infty,q_\infty)-p(b_n,q_n)\leq \epsilon/2+\epsilon/2+p_F(b_n,q_n)-p(b_n,q_n)\leq \epsilon.
    \end{align*}
    Hence, we obtain $p(b_\infty,q_\infty)\leq \liminf_{n\to\infty}p(b_n,q_n)+\epsilon$. Letting $\epsilon\to 0$, we obtain $p(b_\infty,q_\infty)\leq \liminf_{n\to\infty}p(b_n,q_n)$.
    Therefore, by Lemma \ref{lemma}, Lemma \ref{lemma uniformly integrable}  and Theorem \ref{thm uniquness and existence of the equilibrium state}, we obtain $\mu_\infty(-q_\infty\phi-b_\infty\log|f'|)>-\infty$  and 
        $p(b_\infty,q_\infty)
        \leq \limsup_{k\to\infty}p(b_{n_k},q_{n_k})
        = \limsup_{k\to\infty}(h(\mu_k)+\mu_k(-q_{n_k}\phi-b_{n_k}\log |f'|))
        \leq h(\mu_\infty)+\mu_\infty(-q_\infty\phi-b_\infty\log|f'|).
        $
\end{proof}

\section{Multifractal analysis}\label{sec multifractal}
In this section, we prove Theorem \ref{thm main}.
Throughout this section,  we assume that $f\in \maps$ satisfies the condition (G) and $\phi\in \mathcal{R}$. We also assume (R).

Recall that 
$\dimension:=\dim_H(\Lambda).$
Since each branch $f_i$ ($i\in\edge$) of $f$ is in $C^2$ and $f$ satisfies the R\'{e}nyi condition (NERI4), we can apply \cite[Theorem 4.6]{mauldin2000parabolictobeappdated} to obtain the following theorem:
\begin{thm}\label{thm bowen formula}{\cite[Theorem 4.6]{mauldin2000parabolictobeappdated}}
    We have
    \begin{align*}
    &\delta=\sup\left\{\frac{h(\nu)}{\lambda(\nu)}:\nu\in M(f),\ 0<\lambda(\nu),\ \text{$\nu$ is supported on a compact set}\right\}
    \\&=\inf\{t\in\mathbb{R}: P(-t\log|f'|)\leq 0\}.    
    \end{align*}
    \end{thm}
Note that, by (G), for all $b\in \mathbb{R}$ we have
\begin{align}\label{eq finite pressure geometric pote}
    \sum_{i\in \edge} e^{(-b\log|f'\circ\oricodingmap|)([i])}\asymp\sum_{i=1}^{\infty}{i^{-b\cdot\growth}}.
\end{align}
Combining this with Theorem \ref{thm bowen formula}, we obtain
\begin{align}\label{eq bowen}
    P(-\delta\log|f'|)=0 
\end{align}
The following lemma follows from the same argument as in the proof of \cite[Lemma 4.3]{arimanonuniformly}. 
\begin{lemma}\label{lemma indu and nonindu limit}
   We have $\delta=\dim_H(\indulimit)$. 
\end{lemma}

The following theorem is Bowen's formula for the induced system. 
Let 
\begin{align*}
s_\infty:=\growth^{-1}.    
\end{align*}

\begin{thm}\label{thm maximal measure and liftable problem}
We have $\tilde P(-\delta\log |\indumap'\circ\codingmap|)=0$ and  
$\codingmeasure_{\delta,0}(\log|\indumap'\circ\codingmap|)<\infty$. In particular, $\codingmeasure_{\delta,0}$ is the unique equilibrium measure for the potential $-\dimension\log |\indumap'\circ\codingmap|$.
Moreover,
$\codingmeasure_{\delta,0}(\return\circ\codingmap)=\infty$.  
\end{thm}
\begin{proof}
    By (F) and (G), 
    for all $b\in \mathbb{R}$ we have 
\begin{align*}
&\sum_{\omega\in \edgesindu}e^{(-b\log|\indumap'\circ\codingmap|)([\omega])}
=
\sum_{n=1}^{\infty}\sum_{i\in \p}\sum_{j\in\edge_i}e^{(-b\log|\indumap'\circ\codingmap|)([ji^n\edge_i])}
+\sum_{j\in \edge}e^{(-b\log|\indumap'\circ\codingmap|)([jH])}
\\&
\asymp
\sum_{n=1}^{\infty}\sum_{i\in \p}\sum_{j\in\edge_{{i}}}\frac{|f'\circ\oricodingmap|^{{-b}}([j])}{n^{b(1+\exponent)}}+\sum_{j=1}^{\infty}\frac{1}{j^{b\cdot\growth}}
\asymp\sum_{n=1}^{\infty}\frac{1}{n^{b(1+\exponent)}}\sum_{j=1}^{\infty}\frac{1}{j^{b\cdot\growth}}.
\end{align*}
 Therefore, since $\inducoding$ is finitely primitive, we obtain
 $\lim_{b\to \tilde s_\infty}\tilde P(-b\log |\indumap'\circ\codingmap|)=\infty,$
 where $\tilde s_\infty:=\max\{(1+\exponent)^{-1},s_\infty\}.$
 By Bowen's formula \cite[Theorem 4.2.13]{mauldin2003graph} and Lemma \ref{lemma indu and nonindu limit}, we obtain $\tilde P(-\delta\log |\indumap'\circ\codingmap|)=0$. Therefore, since $\lim_{b\to \tilde s_\infty}\tilde P(-b\log |\indumap'\circ\codingmap|)=\infty,$ we obtain $\codingmeasure_{\delta,0}(\log|\indumap'\circ\codingmap|)<\infty$.
On the other hand, by (F), (G) and  \eqref{eq gibbs}, we have 
\begin{align*}
    \codingmeasure_{\delta,0}(\return\circ\codingmap)=\sum_{n\in\mathbb{N}}\sum_{\omega\in \edgepart_n}n\codingmeasure_{\delta,0}([\omega])\asymp
    \sum_{n\in\mathbb{N}}\sum_{\omega\in \edgepart_n}ne^{(-\delta\log|\indumap'\circ\codingmap|)([\omega])}
    \asymp\sum_{n=1}^{\infty}\frac{n}{n^{\delta(1+\exponent)}}.
\end{align*}
Since $\delta(1+\exponent)-1\leq \exponent\leq 1$, we obtain the last statement and the proof is complete.
\end{proof}


For $(b,q)\in\mathbb{R}^2$ we define the set of equilibrium measures for $-q\phi-b\log |f'|$ by 
\[
M_{b,q}:=\{\nu\in M(f):\nu(-q\phi-b\log |f'|)>-\infty,\ p(b,q)=h(\nu)+\nu(-q\phi-b\log |f'|)\}.
\]
If $(\delta,0)=(b,q)$ then we simply write $M_{\delta}:=M_{\delta,0}.$

\begin{prop}\label{prop equilibrium state Lambda}
    We have $M_{\delta}=\conv$. 
\end{prop}

\begin{proof}
By \eqref{eq bowen}, we have $\bigcup_{i\in\p}\{\delta_{x_i}\}\subset M_{\delta}$.
Let $\nu\in M_{\delta}$ be an ergodic measure such that $\nu\notin \conv$. Then, by Lemma \ref{lemma equivalent condition not liftable}, we have $\nu(\indulimit)>0$. 
Let $\tilde\nu:=\nu|_{\indulimit}/\nu(\indulimit)$.
By Remark \ref{rem measurable bijection}, there exists $\tilde \nu'\in M(\shift)$ such that $\tilde \nu=\tilde \nu'\circ\codingmap^{-1}$ and $h(\tilde \nu)=h(\tilde \nu')$.
Thus, by Theorem \ref{thm variational principle induce}, Theorem \ref{thm maximal measure and liftable problem}, \eqref{eq classical Abramov-Kac's formula} and \eqref{eq bowen}, we obtain
$
\lambda(\tilde\nu')=\tilde \nu(\return)\lambda(\nu)<\infty\text{ and }
0=\tilde P(-\delta\log|\indumap'\circ\codingmap|)\geq {\tilde \nu(\return)}(h(\nu)-\delta\lambda(\nu))=0,
$
Therefore, $\tilde P(-\delta\log|\indumap'\circ\codingmap|)=h(\tilde \nu')-\delta\lambda(\tilde \nu')$ and $\tilde\nu'$ is an equilibrium measure for $-\delta\log |\indumap'\circ\codingmap|$. By the uniqueness of the equilibrium measure for $-\delta\log |\indumap'\circ\codingmap|$ (see Theorem \ref{thm equilibrium state induce}), we obtain 
$\tilde \nu'=\codingmeasure_{\delta}.$
By Theorem \ref{thm maximal measure and liftable problem} and \eqref{eq classical Abramov-Kac's formula}, we have $\infty=\codingmeasure_{\delta}(\return\circ\codingmap)=\tilde\nu'(\return\circ\codingmap)=1/\nu(\indulimit)<\infty$. This is a contradiction. Therefore, the set of ergodic measures in $M_\delta$ is $\bigcup_{i\in\p}\{\delta_{x_i}\}$. By the ergodic decomposition theorem (see \cite[Theorem 5.1.3]{viana}), $M_\delta=\conv$.
\end{proof}

By Theorem \ref{thm bowen formula}, for all $b\in (s_\infty,\delta)$ we have $p(b,0)>0$.
Moreover, by \eqref{eq finite pressure geometric pote} and \eqref{eq finite pressure}, for all $b\in (s_\infty,\delta)$ we have $(b,0)\in \fin$.
For $b\in(s_\infty,\delta)$ we set $\codingmeasure_{b}:=\codingmeasure_{b,0}$, $\measure_b=\measure_{b,0}$ and $\mu_{b}:=\mu_{b,0}$. The following lemma is used later.

{
\begin{lemma}\label{lemma equilibrium measure near Delta}
We assume that $\ratio<\infty$. Then, there exists $\{b_n\}_{n\in\mathbb{N}}\subset (s_\infty,\delta)$ such that $\lim_{n\to\infty}b_n=\delta$ and $\lim_{n\to\infty}\mu_{b_n}(\phi)\in \frat$.    
\end{lemma}
}

\begin{proof}
Let $\{(b_n,0)\}_{n\in\mathbb{N}}\subset (s_\infty,\delta)\times\{0\}$ be a sequence such that for all $n\in\mathbb{N}$ we have $b_n\leq b_{n+1}$ and $\lim_{n\to\infty}b_n=\delta$. We first show that $\{(b_n,0)\}_{n\in\mathbb{N}}$ satisfies the conditions (T1.2), (T1.3), (T2) and (T3) in Section \ref{sec thermodynamic}. Since $s_\infty<b_n<\delta$, Theorem \ref{thm bowen formula} yields that for all $n\in\mathbb{N}$ we have $p(b_n,0)>0=\text{LB}(0)$. Moreover, by Theorem \ref{thm maximal measure and liftable problem} and \eqref{eq bowen}, we have $\lim_{n\to\infty}\indupressure(b_n,0,p(b_n,0))=\indupressure(\delta,0,0)=0$. Hence, $\{(b_n,0)\}_{n\in\mathbb{N}}$ satisfies (T2) and (T3). On the other hand, by using (G) and our assumption, that is, $\ratio<\infty$, for all $n\in\mathbb{N}$ we have  
\begin{align}\label{eq equilibrium measure near delta}
    &\sum_{i=1}^\infty e^{(-b_n\log|f'|)\circ\oricodingmap([i])}\log|f'|\circ\oricodingmap([i])
    \ll\sum_{i=1}^{\infty}\frac{\log i}{i^{b_1\growth}} \text{ and }
    \\&\nonumber\sum_{i=1}^{\infty}e^{(-b_n\log|f'|)\circ\oricodingmap([i])}\phi\circ\oricodingmap([i])\ll \sum_{i=1}^{\infty}\frac{\phi\circ\oricodingmap([i])}{i^{b_1\growth}}\ll\sum_{i=1}^{\infty}\frac{\log i}{i^{b_1\growth}}.
\end{align}
Therefore, $\{(b_n,0)\}_{n\in\mathbb{N}}$ satisfies (T1.2) and (T1.3). Hence, by Lemma \ref{lemma tight}, Lemma \ref{lemma uniformly integrable} and Theorem \ref{thm convergence of the equilibrium state}, there exist a subsequence $\{b_{n_k}\}_{k\in\mathbb{N}}$ of $\{b_n\}_{n\in\mathbb{N}}$ and $\mu_\infty^*\in M(f)$ such that $\mu_\infty^*$ is an equilibrium measure for $-\delta\log |f'|$ and $\lim_{k\to\infty}\mu_{b_{n_k}}(\phi)=\mu_{\infty}^*(\phi)$. By Proposition \ref{prop equilibrium state Lambda}, we are done.
\end{proof}

For a convex function $(x_1,\cdots,x_n)\in\mathbb{R}^n\mapsto V(x_1,\cdots,x_n)\in\mathbb{R}$ $(n\in\mathbb{N})$,  $\boldsymbol{\hat x}=(\hat x_1,\cdots,\hat x_n)\in\mathbb{R}^n$ and $1\leq k\leq n$ we denote by $V^+_{x_k}(\boldsymbol{\hat x})$ the right-hand derivative of $V$ with respect to the variable $x_k$ at $\boldsymbol{\hat x}$ and by $V^-_{x_k}(\boldsymbol{\hat x})$ the left-hand derivative of $V$ with respect to the variable $x_k$ at $\boldsymbol{\hat x}$. 

The following proposition states that, even if $(\delta,0)\notin \mathcal{N}$, one can still obtain a weak version of Ruelle’s formula under certain conditions. In the case of finitely many branches, this is shown without  conditions by using the compactness of $M(f)$ (see \cite[Proposition 2.3]{kessebohmer2004multifractaltobeappdated} and \cite[Proposition 3.5]{arimanonuniformly}).

\begin{prop}\label{prop derivatibe of pressure weak conditions}
We assume that $\ratio<\infty$.
If there exists $q_0<0$ such that $\ratio|q_0|<\delta-s_\infty$ and   for all $q\in [q_0,0)$ we have $p(\delta,q)>\LB$ then we have $p^-_q(\delta,0)=\inf_{\nu\in M_\delta}\{-\nu(\phi)\}$. Also, if there exists $q_0>0$ such that $\ratio|q_0|<\delta-s_\infty$ and  for all $q\in (0,q_0]$ we have $p(\delta,q)>\LB$ then we have $p^+_q(\delta,0)=\sup_{\nu\in M_\delta}\{-\nu(\phi)\}$. 
\end{prop}
\begin{proof}
    We first show the first half. We assume that there exists $q_0<0$ such that $\ratio|q_0|<\delta-s_\infty$ and for all $q\in [q_0,0)$ we have $p(\delta,q)>\LB$. 
    Let $\{q_n\}_{n\in\mathbb{N}}$ be a sequence in $(q_0,0)$ such that  $\lim_{n\to\infty} q_n= 0$.
    We will show that $\{(\delta,q_n)\}_{n\in\mathbb{N}}$ satisfies (T1.2), (T1.3), (T2) and (T3) in Section \ref{sec thermodynamic}. 
    We take a small $0<\epsilon<1$ with $|q_0|(\ratio+\epsilon)+\epsilon<\delta-s_\infty$. Then, there exists $N\geq 1$ such that for all $i\geq N$ and $x\in \oricodingmap([i])$ we have $\phi(x)<(\ratio+\epsilon)\log |f'(x)|$. 
    Thus, for all $n\in\mathbb{N}$ we obtain
    \begin{align*}
        \sum_{i=N}^\infty e^{\multipote_{\delta,q_n}([i])}\phi\circ\oricodingmap([i])\leq (\ratio+\epsilon)   \sum_{i=N}^\infty e^{\multipote_{\delta,q_n}([i])}\log|f'|\circ\oricodingmap([i]).
    \end{align*}
    Moreover, by (G), we obtain 
    \begin{align*}
        &\sum_{i=N}^\infty e^{\multipote_{\delta,q_n}([i])}\log|f'|\circ\oricodingmap([i])
        \ll \sum_{i=N}^{\infty}\frac{1}{i^{\growth(\delta-|q_0|(\ratio+\epsilon)-\epsilon)}}.
            \end{align*}
Since $\delta-|q_0|(\ratio+\epsilon)-\epsilon>s_\infty=\growth^{-1}$, these inequalities imply that $\{(\delta,q_n)\}_{n\in\mathbb{N}}$ satisfies (T1.2) and (T1.3). Moreover, by Theorem \ref{thm maximal measure and liftable problem} and \eqref{eq bowen}, we have that $\lim_{n\to\infty}\indupressure(\delta,q_n,p(\delta,q_n))=\indupressure(\delta,0,0)=0$. Combining this with our assumption, we can see that $\{(\delta,q_n)\}_{n\in\mathbb{N}}$ satisfies (T2) and (T3). Hence, by Lemma \ref{lemma tight}, Lemma \ref{lemma uniformly integrable} and Theorem \ref{thm convergence of the equilibrium state}, there exist a subsequence $\{q_{n_k}\}_{k\in\mathbb{N}}$ of $\{q_n\}_{n\in\mathbb{N}}$ and $\mu_\infty^*\in M(f)$ such that $\mu_\infty^*$ is an equilibrium measure for $-\delta\log |f'|$ and $\lim_{n\to\infty}\mu_{\delta,q_{n_k}}(\phi)=\mu_{\infty}^*(\phi)$. Hence, by Theorem \ref{thm regularity of non induced pressure} and the convexity of $q\mapsto p(\delta,q)$ in a small neighborhood of $0$, we obtain
\begin{align}\label{eq proof derivative weak}
    p^-_q(\delta,0)=\lim_{k\to\infty}\frac{\partial}{\partial q}p(\delta,q_{n_k})=-\lim_{k\to\infty}\mu_{\delta,q_{n_k}}(\phi)=-\mu_{\infty}^*(\phi)\geq \inf_{\nu\in M_\delta}\{-\nu(\phi)\}.
\end{align}
On the other hand, the inequality $p^-_q(\delta,0)\leq\inf_{\nu\in M_\delta}\{-\nu(\phi)\}$ follows from the variational principle for the topological pressure (see, for example, \cite[p.812]{Urbanskinoninvertible}). 
Therefore, we obtain $p^-_q(\delta,0)=\inf_{\nu\in M_\delta}\{-\nu(\phi)\}$. By a similar argument, one can show the second half.
\end{proof}

\subsection{Conditional variational principle}
In this section, we strengthen the conditional variational principle established by Jaerisch and Takahasi \cite{MixedJT}.

Since $\phi$ and $\log |f'|$ are acceptable, $\phi$ and $\log |f'|$ have mild distortion, that is, 
\[
\sup_{i\in\edge}\sup_{x,y\in \text{Int}_\Lambda(I_i)}\{\psi(x)-\psi(y)\}<\infty \text{ and }\sup_{\omega\in\edge^n}\sup_{x,y\in \text{Int}_\Lambda(I_\omega)}\{S_n\psi(x)-S_n\psi(y)\}=o(n),
\]
where $\psi\in \{\phi,\log|f'|\}$ and $S_n\psi:=\sum_{k=0}^{n-1}\psi\circ f^k$. This distortion property and Theorem \ref{thm bowen formula} enable us to apply the main theorem of \cite{MixedJT}. For $\nu\in M(f)$ we define  $\dim_H(\nu)=h(\nu)/\lambda(\nu)$ if $\lambda(\nu)>0$ and $\dim_H(\nu)=0$ if $\lambda(\nu)=0$.
\begin{thm}{\cite[Main Theorem]{MixedJT}}\label{thm coditional variational principle}
    For all $\alpha\in (\alpha_{\inf},\alpha_{\sup})$ we have 
    \[
    b(\alpha)=
    \lim_{\epsilon\to0}\sup\left\{\dim_H(\nu):\nu\in M(f),\ \lambda(\nu)<\infty,\  |\nu(\phi)-\alpha|<\epsilon\right\}.
    \]
\end{thm}
We define the function $\tilde b:(\alpha_{\inf},\alpha_{\sup})\rightarrow [0,1]$ by 
\[
\tilde b(\alpha)=\sup\left\{\dim_H(\nu):\nu\in M(f),\ \lambda(\nu)<\infty,\  \nu(\phi)=\alpha\right\}.
\]
We denote by $\underline{i}$ the index in $\p$ satisfying $\alpha_{\underline{i}}=\min_{i\in\p}\{\alpha_i\}$ and by $\overline{i}$ the index in $\p$ satisfying $\alpha_{\overline{i}}=\max_{i\in\p}\{\alpha_i\}$.

Let $F$ be a finite subset of $\edge$ and let $\Lambda_F:=\oricodingmap(F^\mathbb{N})$. 
Since $\Lambda_F$ is an $f$-invariant set, we can consider the dynamical system $f_F:=f|_{\Lambda_F}$. We denote by $M(f,F)$ the set of $f_F$-invariant Borel probability measures supported by $\Lambda_F$.
Define $\alpha_{\inf,F}:=\inf\{\mu(\phi):\mu\in M(f,F)\}$ and $\alpha_{\sup,F}:=\sup\{\mu(\phi):\mu\in M(f,F)\}$. Since $f$ has full-branches, Birkhoff's ergodic theorem yields that (see the beginning of the proof of \cite[Lemma 3.2]{transience}) 
\begin{align}\label{eq compact range}
&\alpha_{\inf}=\inf \{\alpha_{\inf,F}:F\subset \edge,\ \#F<\infty\}
\text{ and }
\\&\alpha_{\sup}=\sup \{\alpha_{\sup,F}:F\subset \edge,\ \#F<\infty\}\nonumber.
\end{align}
In particular, for any $\underline{\alpha}\in (\alpha_{\inf},\alpha_{\underline{i}})$ and $\overline{\alpha}\in (\alpha_{\overline{i}},\alpha_{\sup})$ 
there exist $\underline{\mu_1},\underline{\mu_2},\overline{\mu_1},\overline{\mu_2}\in M(f)$ such that $\lambda(\mu)<\infty$ for all $\mu\in \{\underline{\mu_1},\underline{\mu_2},\overline{\mu_1},\overline{\mu_2}\}$,
    $\alpha_{\inf}<\underline{\mu_1}(\phi)<\underline{\alpha}<\underline{\mu_2}(\phi)<\alpha_{\underline{i}}$ and $\alpha_{\overline{i}}<\overline{\mu_2}(\phi)<\overline{\alpha}<\overline{\mu_1}(\phi)<\alpha_{\sup}$.

\begin{lemma}\label{lemma conti tilde b}
    $\tilde b$ is continuous on $(\alpha_{\inf},\alpha_{\underline{i}})$. Moreover, if $\ratio<\infty$ then $\tilde b$ is also continuous on $(\alpha_{\overline{i}},\alpha_{\sup})$.
\end{lemma}
\begin{proof}
    Let $\alpha\in(\alpha_{\inf},\alpha_{\underline{i}})$. Then,  there exist $\mu_1,\mu_2\in M(f)$ such that $\lambda(\mu_1)<\infty$, $\lambda(\mu_2)<\infty$ and 
    $\alpha_{\inf}<\mu_1(\phi)<\alpha<\mu_2(\phi)<\alpha_{\underline{i}}$. We notice that $\mu_1,\mu_2\notin \conv$. Thus, by Lemma \ref{lemma equivalent condition not liftable}, we obtain $\lambda(\mu_1)>0$ and $\lambda(\mu_2)>0$. Let $\{\beta_n\}_{n\in\mathbb{N}}\subset (\mu_1(\phi),\mu_2(\phi))$ be a sequence such that $\lim_{n\to\infty}\beta_n=\alpha$. We first show that $\tilde b(\alpha)\leq \liminf_{n\to\infty}\tilde b(\beta_n)$. Let $\epsilon>0$. Then, there exists $\mu\in M(f)$ such that $\lambda(\mu)<\infty$, $\mu(\phi)=\alpha$  and $\tilde b(\alpha)<\dim_H(\mu)+\epsilon$. Note that for all $\nu\in \conv$ we have $\nu(\phi)\in [\alpha_{\underline{i}},\alpha_{\overline{i}}]$. Therefore, since $\alpha_{\underline{i}}>\alpha=\mu(\phi)$, $\mu\notin\conv$ which yields that $\lambda(\mu)>0$. 
    Since $\{\beta_n\}_{n\in\mathbb{N}}\subset (\mu_1(\phi),\mu_2(\phi))$ and $\lim_{n\to\infty}\beta_n=\alpha$ there exist $\{p_n\}_{n\in\mathbb{N}}\subset[0,1]$ and $\{s_n\}_{n\in\mathbb{N}}\subset\{1,2\}$ such that for all $n\in\mathbb{N}$ we have $\beta_n=p_n\mu(\phi)+(1-p_n)\mu_{s_n}(\phi)$ and $\lim_{n\to\infty}p_n=1$. Set $\nu_n:=p_n\mu+(1-p_n)\mu_{s_n}$ for each $n\in\mathbb{N}$. Then, since $\lambda(\mu)>0$, we obtain
    \begin{align*}
\lim_{n\to\infty}\dim_H(\nu_n)=\lim_{n\to\infty}\frac{p_nh(\mu)+(1-p_n)h(\mu_{s_n})}{p_n\lambda(\mu)+(1-p_n)\lambda(\mu_{s_n})}=\dim_H(\mu).    
    \end{align*}
Hence, noting that $\nu_n(\phi)=\beta_n$ ($n\in\mathbb{N}$), we obtain  $\tilde b(\alpha)<\dim_H(\mu)+\epsilon=\lim_{n\to\infty}\dim_H(\nu_n)+\epsilon\leq\liminf_{n\to\infty}\tilde b(\beta_n)+\epsilon$. Letting $\epsilon\to0$, we obtain $\tilde b(\alpha)\leq \liminf_{n\to\infty}\tilde b(\beta_n)$.

Next, we shall show that $\limsup_{n\to\infty}\tilde b(\beta_n)\leq \tilde b(\alpha)$. Let $\epsilon>0$. Then, for each $n\in\mathbb{N}$ there exists $\mu_n\in M(f)$ such that $\lambda(\mu_n)<\infty$, 
$\mu_n(\phi)=\beta_n$ and $\tilde b(\beta_n)<\dim_H(\mu_n)+\epsilon$. By using the inequality $\alpha<\alpha_{\underline{i}}$, we will show that 
\begin{align}\label{eq proof of cont tilde b}
    \liminf_{n\to\infty}\lambda(\mu_n)>0.
\end{align}
For a contradiction, we assume that  $\liminf_{n\to\infty}\lambda(\mu_n)=0$. Then, by taking a subsequence if necessary, we may assume that $\lim_{n\to\infty}\lambda(\mu_n)=0$. Since for all $x\in \Lambda\setminus\{x_i\}_{i\in\p}$ we have $\log|f'(x)|>0$, 
\begin{align}\label{eq proof cont tilde b outside measure}
\text{
    for each closed set $Z\subset \Lambda$ with $\{x_i\}_{i\in\p}\cap Z=\emptyset$ we have $\lim_{n\to\infty}\mu_n(Z)=0$.
    }
\end{align} 
Fix $\eta>0$ with $\alpha<\alpha_{\underline{i}}-\eta$. 
Since $\phi$ is continuous  at each parabolic fixed point $x_i$ ($i\in\p$), there exists an open set $O\subset \Lambda$ such that $\{x_i\}_{i\in\p}\subset O$ and for all $x\in O$ we have $\phi(x)>\alpha_{\underline{i}}-\eta$. Hence, by (P) and \eqref{eq proof cont tilde b outside measure}, we obtain
\begin{align*}
    \alpha=\lim_{n\to\infty}\beta_n
    =\lim_{n\to\infty}\int\phi d\mu_n
    \geq\liminf_{n\to\infty}\int _O\phi d\mu_n
    >(\alpha_{\underline{i}}-\eta)\liminf_{n\to\infty}\mu_n(O)>\alpha.
\end{align*}
This is a contradiction. Thus, we obtain \eqref{eq proof of cont tilde b}. 
Since $\{\beta_n\}_{n\in\mathbb{N}}\subset (\mu_1(\phi),\mu_2(\phi))$, $\lim_{n\to\infty}\beta_n=\alpha$ and $\mu_n(\phi)=\beta_n$ ($n\in\mathbb{N}$), there exist $\{p_n\}_{n\in\mathbb{N}}\subset[0,1]$ and $\{s_n\}_{n\in\mathbb{N}}\subset\{1,2\}$ such that for all $n\in\mathbb{N}$ we have $\alpha=p_n\mu_n(\phi)+(1-p_n)\mu_{s_n}(\phi)$ and $\lim_{n\to\infty}p_n=1$. For each $n\in\mathbb{N}$ we set $\nu_n:=p_n\mu_n+(1-p_n)\mu_{s_n}$. Then, for each $n\in\mathbb{N}$ we have $\nu_n(\phi)=\alpha$ and, by \eqref{eq proof of cont tilde b}, 
\[
\limsup_{n\to\infty}\dim_H(\nu_n)=\limsup_{n\to\infty}\frac{p_nh(\mu_n)+(1-p_n)h(\mu_{s_n})}{p_n\lambda(\mu_n)+(1-p_n)\lambda(\mu_{s_n})}=\limsup_{n\to\infty}\dim_H(\mu_n).
\]
This implies that 
\begin{align*}
        \limsup_{n\to\infty}\tilde b(\beta_n)\leq \limsup_{n\to\infty}\dim_H(\mu_n)+\epsilon=\limsup_{n\to\infty}\dim_H(\nu_n)+\epsilon \leq \tilde b(\alpha)+\epsilon.
\end{align*}
Letting $\epsilon\to 0$, we obtain $\limsup_{n\to\infty}\tilde b(\beta_n)\leq\tilde b(\alpha)$. Hence, we conclude that $\lim_{n\to\infty}\tilde b(\beta_n)=\tilde b(\alpha)$ and thus, $\tilde b$ is continuous at $\alpha\in (\alpha_{\inf},\alpha_{\underline{i}})$.

Next, we consider the case $\alpha\in (\alpha_{\overline{i}},\alpha_{\sup})$. Again, there exist $\mu_1,\mu_2\in M(f)$ such that $\lambda(\mu_1)<\infty$, $\lambda(\mu_2)<\infty$ and 
    $\alpha_{\overline{i}}<\mu_1(\phi)<\alpha<\mu_2(\phi)<\alpha_{\sup}$. Let $\{\beta_n\}_{n\in\mathbb{N}}\subset (\mu_1(\phi),\mu_2(\phi))$ be a sequence such that $\lim_{n\to\infty}\beta_n=\alpha$. By a similar argument used in the proof of $\tilde b(\alpha)\leq \liminf_{n\to\infty}\tilde b(\beta_n)$ for $\alpha\in (\alpha_{\inf},\alpha_{\underline{i}})$, we can show that $\tilde b(\alpha)\leq \liminf_{n\to\infty}\tilde b(\beta_n)$.

    We assume that $\ratio<\infty$. Let $\epsilon>0$. Then, for each $n\in\mathbb{N}$ there exists $\mu_n\in M(f)$ such that $\lambda(\mu_n)<\infty$, 
$\mu_n(\phi)=\beta_n$ and $\tilde b(\beta_n)<\dim_H(\mu_n)+\epsilon$. 
If we can show 
\begin{align}\label{eq positive lyapunov}
    \liminf_{n\to\infty} \lambda(\mu_n)>0
\end{align}
then by repeating the argument used in the proof of $\limsup_{n\to\infty}\tilde b(\beta_n)\leq \tilde b(\alpha)$ for $\alpha\in (\alpha_{\inf},\alpha_{\underline{i}})$, we obtain $\limsup_{n\to\infty}\tilde b(\beta_n)\leq \tilde b(\alpha)$. Hence, the proof of the continuity of $\tilde b$ at $\alpha\in (\alpha_{\overline{i}},\alpha_{\sup})$ will be complete. 

For a contradiction, we assume that $\liminf_{n\to\infty}\lambda(\mu_n)=0$. Then, by taking a subsequence if necessary, we may assume that $\lim_{n\to\infty}\lambda(\mu_n)=0$. Then, we obtain \eqref{eq proof cont tilde b outside measure}.
Fix $\eta>0$ with $\alpha_{\overline{i}}+\eta<\alpha$. 
Since $\phi$ is continuous  at each parabolic fixed point $x_i$ ($i\in\p$), there exists an open set $O\subset \Lambda$ such that $\{x_i\}_{i\in\p}\subset O$ and for all $x\in O$ we have $\phi(x)<\alpha_{\overline{i}}+\eta$.
On the other hand, since  $\ratio<\infty$, there exist $N\in \mathbb{N}$ and $C>0$ such that for all $n\geq N$ and $x\in I_n$ we have $\phi(x)\leq C\log |f'(x)|$. By using (P) and \eqref{eq proof cont tilde b outside measure} and noting that $D:=\sup_{1\leq i\leq N }\sup_{x\in I_i}\phi(x)<\infty$, we obtain
\begin{align*}
0\leq\lim_{n\to\infty}\int_{\Lambda\setminus O}\phi d\mu_n
\leq \lim_{n\to\infty}\left( D\mu_n\left(\bigcup_{i=1}^{N}I_i\setminus O\right)+C\lambda(\mu_n)\right)=0.
\end{align*}
This implies that 
\begin{align*}
    \alpha=\lim_{n\to\infty}\beta_n
    =\lim_{n\to\infty}\int\phi d\mu_n
    =\limsup_{n\to\infty}\int _O\phi d\mu_n
    <\alpha_{\overline{i}}+\eta<\alpha.
\end{align*}
This is a contradiction. Thus, we obtain \eqref{eq positive lyapunov} and the proof is complete.
\end{proof}

The following theorem follows easily from Lemma \ref{lemma conti tilde b} (see the proof of \cite[Proposition 3.4]{A.}):
\begin{thm}\label{thm stong conditional variational principle}
    For all $\alpha\in (\alpha_{\inf},\alpha_{\underline{i}})$ we have
    $b(\alpha)=\tilde b(\alpha)$. Moreover, if we have  $\ratio<\infty$ then for all $\alpha\in (\alpha_{\overline{i}},\alpha_{\sup})$ we have  $b(\alpha)=\tilde b(\alpha)$.
\end{thm}

\subsection{The flat part and lower bound of $b(\alpha)$}
In this section, we discuss various types of lower bounds for the function $\alpha\mapsto b(\alpha)$.
\begin{prop}\label{prop frat}
    For all $\alpha\in \frat$ we have $b(\alpha)=\delta$. 
\end{prop}
\begin{proof}
    We first show that for all $\alpha\in (\alpha_{\underline{i}},\alpha_{\overline{i}})$ we have $b(\alpha)=\delta$. Let $\epsilon>0$ and let $\alpha\in (\alpha_{\underline{i}},\alpha_{\overline{i}})$. By Theorem \ref{thm bowen formula}, there exists  $\nu\in M(f)$ such that $0<\lambda(\nu)<\infty$ and  $\delta<h(\nu)/\lambda(\nu)+\epsilon$. Then, there exists $i'\in\{\underline{i},\overline{i}\}$ and $p\in (0,1]$ such that we have $\alpha=p\nu(\phi)+(1-p)\alpha_{i'}$. We set $\mu=p\nu+(1-p)\delta_{x_{i'}}$. Then, since $p>0$, we have $\mu(\phi)=\alpha$, $\lambda(\mu)=p\lambda(\nu)>0$ and $\dim_H(\mu)=h(\mu)/\lambda(\mu)=h(\nu)/\lambda(\nu)>\delta-\epsilon$. Therefore, by Theorem \ref{thm coditional variational principle}, we obtain $b(\alpha)\geq\dim_H(\mu)>\delta-\epsilon$. Letting $\epsilon \to0$, we obtain $b(\alpha)=\delta$. Moreover, if $(\alpha_{\underline{i}},\alpha_{\overline{i}})\neq \emptyset$ then Theorem \ref{thm coditional variational principle} yields that $b(\alpha)=\delta$ for $\alpha\in \{\alpha_{\underline{i}},\alpha_{\overline{i}}\}$.
    In the case where $(\alpha_{\underline{i}},\alpha_{\overline{i}})= \emptyset$, by slightly modifying the above argument, one can show that $b(\alpha)=\delta$ for  $\alpha=\alpha_{\underline{i}}=\alpha_{\overline{i}}$. 
\end{proof}

\begin{prop}\label{prop frat under some assumption}
    Assume that $\ratio=\infty$. Then, for all $\alpha\in (\alpha_{\overline{i}},\infty)$ we have $b(\alpha)=\delta$.
\end{prop}
\begin{proof}
    We first show that for all $\epsilon>0$ there exists $\{\mu_n\}_{n\in\mathbb{N}}\subset M(f)$ such that for all $n\in\mathbb{N}$ we have $0<\lambda(\mu_n)<\infty$, 
    $\mu_n(\phi)<\infty$,
    $\lim_{n\to\infty}\dim_H(\mu_n)>\delta-\epsilon$ and $\lim_{n\to\infty}\mu_n(\phi)=\infty$.
    Let $\epsilon>0$.
    By Theorem \ref{thm bowen formula}, there exists $\nu$ such that $\dim_H(\nu)>\delta-\epsilon$, $0<\lambda(\nu)<\infty$ and $\nu(\phi)<\infty$.
    On the other hand, since we assume that $\ratio=\infty$, for all $n\in\mathbb{N}$ there exists $k_n\in\mathbb{N}$ such that $k_n\notin\p$ and 
    \begin{align}\label{eq frat}
    \frac{\phi(x_{k_n})}{\log |f'(x_{k_n})|}\geq n^2,    
    \end{align}
    where $x_{k_n}$ denotes the unique fixed point of $f_{k_n}$ in $\overline{\Delta_{k_n}}$. We set, for all $n\in\mathbb{N}$, 
    $p_n:=(n\log |f'(x_{k_n})|)^{-1} \text{ and }\mu_n:=(1-p_n)\nu+p_n\delta_{x_{k_n}}.
    $
Then, for all $n\in\mathbb{N}$ we have $\lambda(\mu_n)>0$ and 
$
\liminf_{n\to\infty}\mu_n(\phi)\geq\liminf_{n\to\infty}p_n\phi(x_{k_n})=\infty 
$
by (P) and \eqref{eq frat}. 
We also have, for all $n\in\mathbb{N}$, $\lambda(\mu_n)<\infty$ and $\mu_n(\phi)<\infty$. 
Moreover,  since $\lim_{n\to\infty} p_n\log|f'(x_{k_n})|=0$, we have $\lim_{n\to\infty}\dim_H(\mu_n)=\dim_H(\nu)>\delta-\epsilon$. Thus, $\{\mu_n\}_{n\in\mathbb{N}}$ is a sequence satisfying the desired conditions.

Let $\alpha\in(\alpha_{\overline{i}},\infty)$ and let $\epsilon>0$. Then, for all sufficiently large $n\in\mathbb{N}$ there exists $p_n'\in (0,1]$ such that $\alpha=p_n'\alpha_{\overline{i}}+(1-p_n')\mu_n(\phi)$. We set $\xi_n:=p_n'\delta_{x_{\overline{i}}}+(1-p_n')\mu_n$. Then, for all sufficiently large $n\in\mathbb{N}$ we have $\xi_n(\phi)=\alpha$ and $\lim_{n\to\infty}\dim_H(\xi_n)=\lim_{n\to\infty}\dim_H(\mu_n)>\delta-\epsilon$. By Theorem \ref{thm coditional variational principle}, this implies that $b(\alpha)>\delta-\epsilon$. Letting $\epsilon\to0$, we are done. 
    \end{proof}

As in \cite{IommiJordanBirkhoff}, we define 
\begin{align*}
    \delta^*:=\inf\{b\in [0,\delta]:p(b,q)<\infty \text{ for some } q\in\mathbb{R}\}.
\end{align*}
\begin{lemma}\label{lemma lower bound}
    For all $\alpha\in (\alpha_{\inf},\alpha_{\sup})$ we have $b(\alpha)\geq \delta^*$. 
\end{lemma}
\begin{proof}
If $\delta^*=0$ then there is nothing to prove. Hence, we assume that $\delta^*>0$.
    If $\alpha\in \frat$ then by Proposition \ref{prop frat}, we have $b(\alpha)=\delta\geq \delta^*$. Moreover, if $\alpha\in (\alpha_{\overline{i}},\infty)$ and $\ratio=\infty$ then by Proposition \ref{prop frat under some assumption}, we have $b(\alpha)=\delta\geq\delta^*$. 

    Let $\alpha\in (\alpha_{\inf},\alpha_{\sup})\setminus \frat$. If $\alpha\in (\alpha_{\overline{i}},\alpha_{\sup})$, we assume that $\ratio<\infty$. Then, $b(\alpha)\geq \delta^*$ follows from essentially the same argument as in \cite[Lemma 4.2]{IommiJordanBirkhoff}, which is based on the variational principle for the topological pressure and the conditional variational principle.
\end{proof}

\begin{lemma}\label{lemma weak lower bound}
    For all $\alpha\in (\alpha_{\inf},\alpha_{\sup})$ we have $b(\alpha)>0$.
\end{lemma}
\begin{proof}
If $\alpha\in \frat$ then by Proposition \ref{prop frat}, we have $b(\alpha)=\delta>0$. Let $\alpha\in (\alpha_{\inf},\min\frat)$. Then, there exists $\mu\in M(f)$ such that $\lambda(\mu)<\infty$ and 
    $\alpha_{\inf}<\mu(\phi)<\alpha$. By Proposition \ref{prop frat} and Theorem \ref{thm coditional variational principle}, there exists $\nu\in M(f)$ such that $\lambda(\nu)<\infty$, $\dim_H(\nu)>0$ and $\alpha<\nu(\phi)<\infty$. Notice that since $\dim_H(\nu)>0$, we have $\lambda(\nu)>0$ and $h(\nu)>0$. Since $\mu(\phi)<\alpha< \nu(\phi)$, there exists $p\in(0,1)$ such that $\alpha=p\mu(\phi)+(1-p)\nu(\phi)$. Set $\xi:=p\mu+(1-p)\nu$. Then, we obtain $\xi(\phi)=\alpha$ and thus, by Theorem \ref{thm coditional variational principle}, 
    $b(\alpha)\geq \dim_H(\xi)>0.$ 
    By a similar argument, we can show that for all $\alpha\in (\max\frat,\alpha_{\sup})$ we have $b(\alpha)>0$.
\end{proof}



    

\subsection{Properties of the function $(b,q)\mapsto p(b,q)+q\alpha$}

For $\alpha\in (\alpha_{\inf},\alpha_{\sup})$ we define the function $\palpha:\mathbb{R}^2\rightarrow\mathbb{R}\cup\{\infty\}$ by 
\[
\palpha(b,q):=p(b,q)+q\alpha=P(q(-\phi+\alpha)-b\log|f'|).
\]
Our main theorem is proved by relating the function $\alpha\mapsto b(\alpha)$ to the function $(b,q)\mapsto \palpha(b,q)$ for $\alpha\in (\alpha_{\inf},\alpha_{\sup})$. In this section, we first establish properties of the function $(b,q)\mapsto \palpha(b,q)$ for $\alpha\in (\alpha_{\inf},\alpha_{\sup})$. We then prove Theorem \ref{thm main} under the assumption that these two functions are related.

The following lemma provides a key observation for the analysis of the set $\{(b,q)\in \mathbb{R}^2:p(b,q)>\LB\}$ (see Remark \ref{rem cheaking}). 
\begin{lemma}\label{lemma positivity p alpha}
    For all $\alpha\in (\alpha_{\inf},\alpha_{\sup})$ with $b(\alpha)>\delta^*$ and $q\in \mathbb{R}$ we have $\palpha(b(\alpha),q)\geq 0$. Moreover, for all $b\in \mathbb{R}$ we have 
    \begin{align}\label{eq asymptotic behavior of alpha pressure}
        \lim_{|q|\to\infty}\palpha(b,q)=\infty.
    \end{align}
\end{lemma}
\begin{proof}
Let $\alpha\in (\alpha_{\inf},\alpha_{\sup})$ with $b(\alpha)>\delta^*$ and let $q\in \mathbb{R}$.
$\palpha(b(\alpha),q)\geq 0$ follows from the same argument as in \cite[Lemma 4.3]{IommiJordanBirkhoff}, which is based on the variational principle for the topological pressure and the conditional variational principle.
Next, we shall show \eqref{eq asymptotic behavior of alpha pressure}. Let $\alpha\in (\alpha_{\inf},\alpha_{\sup})$ and let $b\in\mathbb{R}$. Since $\alpha \in (\alpha_{\inf},\alpha_{\sup})$, there exists $\underline{\nu},\overline{\nu}\in M(f)$ such that $\underline{\nu}(\phi)<\alpha<\overline{\nu}(\phi)$, $\overline{\nu}(\phi+\log|f'|)<\infty$ and $\underline{\nu}(\phi+\log|f'|)<\infty$. Hence, we obtain
     $   \lim_{q\to\infty}\palpha(b,q)
        \geq 
         h(\underline{\nu})+\lim_{q\to\infty}q(\alpha-\underline{\nu}(\phi))-b\lambda(\underline{\nu})=\infty
        $\text{ and }
        $
        \lim_{q\to-\infty}\palpha(b,q)
        \geq 
         h(\overline{\nu})+\lim_{q\to-\infty}q(\alpha-\overline{\nu}
         (\phi))-b\lambda(\overline{\nu})=\infty.
         $
\end{proof}

\begin{rem}\label{rem cheaking}
Let $\alpha\in (\alpha_{\inf},\alpha_{\sup})$ satisfy $b(\alpha)>\delta^*$.
By Lemma \ref {lemma positivity p alpha}, we have  $p(b(\alpha),q)\geq -q\alpha>\LB$ if $\alpha\in (\alpha_{\inf},\min \frat)$ and $q\in (0,\infty)$, and $p(b(\alpha),q)\geq -q\alpha>\LB$ if $\alpha\in (\max\frat, \alpha_{\sup})$ and $q\in (-\infty,0)$.    
\end{rem}

For each $\alpha\in (\alpha_{\inf},\alpha_{\sup})$ 
we denote by $\mathcal{Q}$ the set of all $q\in \mathbb{R}$ such that $(b(\alpha),q)\in \mathcal{N}$, 
 \begin{align}\label{eq relation in the statement}
\palpha(b(\alpha),q)=0 \text{ and } \frac{\partial}{\partial q}\palpha(b(\alpha),q)=0.        
    \end{align}

\begin{lemma}\label{lemma sufficent condition for nice}
    Let $\alpha\in (\alpha_{\inf},\alpha_{\sup})$ satisfy $b(\alpha)>\delta^*$. If there exists $q_0\in \mathbb{R}$ such that $(b(\alpha),q_0)\in\mathcal{N}$ and $(\partial/\partial q)\palpha(b(\alpha),q_0)=0$ then we have $\palpha(b(\alpha),q_0)=0$
\end{lemma}
\begin{proof}
    Let $\alpha\in (\alpha_{\inf},\alpha_{\sup})$ satisfy $b(\alpha)>\delta^*$ and let $q_0\in \mathbb{R}$ satisfy $(b(\alpha),q_0)\in\mathcal{N}$ and $(\partial/\partial q)\palpha(b(\alpha),q_0)=0$.
    Note that by Theorem \ref{thm regularity of non induced pressure}, $(\partial/\partial q)\palpha(b(\alpha),q_0)=0$ implies that $\mu_{b(\alpha),q_0}(\phi)=\alpha$. Thus,
    by Lemma \ref{lemma positivity p alpha} and Theorem \ref{thm coditional variational principle}, we obtain
    $0\leq \palpha(b(\alpha),q_0)=h(\mu_{b(\alpha),q_0})-b(\alpha)\lambda(\mu_{b(\alpha),q_0})\leq 0$
\end{proof}

Let $\nice$ denote the set of  all $\alpha\in (\alpha_{\inf},\alpha_{\sup})$ for which there exists a unique number $q(\alpha)$ such that $\mathcal{Q}=\{q(\alpha)\}$. We define the function $\alpha\in \nice\mapsto q(\alpha)\in \mathbb{R}$.

\begin{prop}\label{prop equilibrium measure conditional}
   For all $\alpha\in \nice$, the equilibrium measure $\mu_{b(\alpha),q(\alpha)}$ for $-q(\alpha)\phi-b(\alpha)\log |f'|$ is a unique measure $\nu\in M(f)$ satisfying $\lambda(\nu)>0$, $\nu(\phi)=\alpha$ and $b(\alpha)=h(\nu)/\lambda(\nu)$. 
\end{prop}
\begin{proof}
Let $\alpha\in \nice$. Then, by \eqref{eq relation in the statement}, we obtain $\mu_{b(\alpha),q(\alpha)}(\phi)=\alpha$ and thus,
\begin{align*}
&0=\palpha(b(\alpha),q(\alpha))=h(\mu_{b(\alpha),q(\alpha)})+q(\alpha)(\alpha-\mu_{b(\alpha),q(\alpha)}(\phi)) -b(\alpha)\lambda(\mu_{b(\alpha),q(\alpha)})
\\&=h(\mu_{b(\alpha),q(\alpha)})-b(\alpha)\lambda(\mu_{b(\alpha),q(\alpha)})    
\end{align*}
and thus,  $b(\alpha)=h(\mu_{b(\alpha),q(\alpha)})/\lambda(\mu_{b(\alpha),q(\alpha)})$. 
Next, we shall prove uniqueness. Let $\nu$ be in $M(f)$ such that $\lambda(\nu)>0$, $\nu(\phi)=\alpha$ and $b(\alpha)=h(\nu)/\lambda(\nu)$. Then, we have
\[
h(\nu)+q(\alpha)(\alpha-\nu(\phi))-b(\alpha)\lambda(\nu)=0=p(b(\alpha),q(\alpha))+q(\alpha)\alpha.
\]
Thus, $\nu$ is an equilibrium measure for $-q(\alpha)\phi-b(\alpha)\log|f'|$. Therefore, since $(b(\alpha),q(\alpha))\in \mathcal{N}$, the uniqueness of an equilibrium measure for $-q(\alpha)\phi-b(\alpha)\log|f'|$ (Theorem \ref{thm uniquness and existence of the equilibrium state}) yields that $\nu=\mu_{b(\alpha),q(\alpha)}$.    
\end{proof}

\begin{prop}\label{prop real analytic}
    The functions $\alpha\mapsto b(\alpha)$ and $\alpha\mapsto q(\alpha)$ are real-analytic on $\text{Int}(\nice)$.
\end{prop}

\begin{proof}
If $\alpha_{\inf}=\alpha_{\sup}$, there is nothing to prove. Thus, we assume that $\alpha_{\inf}<\alpha_{\sup}$.
We proceed with this proof as in \cite[Lemma 9.2.4]{Barreirabook}.
We define the function $G:\mathbb{R}\times\mathcal{N}\rightarrow \mathbb{R}^2$ by $G(\alpha,b,q):= (\palpha(b,q), \frac{\partial}{\partial q}\palpha(b,q))=(p(b,q)+q\alpha, \frac{\partial}{\partial q}p(b,q)+\alpha)$. By definition of $\nice$, for all $\alpha\in\nice$ we have $G(\alpha,b(\alpha),q(\alpha))=0$. 
We want to apply the implicit function theorem in order to show the regularity of the functions $b$ and $q$. To do this, it is sufficient to show that 
\begin{align*}
    \text{det}\begin{pmatrix}
   \frac{\partial}{\partial b}\palpha(b(\alpha),q(\alpha)) & \frac{\partial}{\partial q}\palpha(b(\alpha),q(\alpha)) \\
   \frac{\partial^2}{\partial b\partial q}\palpha(b(\alpha),q(\alpha)) & 
   \frac{\partial^2}{\partial q^2}\palpha(b(\alpha),q(\alpha))
\end{pmatrix}
\neq 0.
\end{align*}
Note that, by Theorem \ref{thm regularity of non induced pressure}, we have $\frac{\partial}{\partial q}\palpha(b(\alpha),q(\alpha))=0$ and $\frac{\partial}{\partial b}\palpha(b(\alpha),q(\alpha))=-\lambda(\mu_{b(\alpha),q(\alpha)})<0$.
By $\alpha_{\inf}<\alpha_{\sup}$, Theorem \ref{thm regularity of non induced pressure} implies that $\frac{\partial^2}{\partial q^2}\palpha(b(\alpha),q(\alpha))\neq 0$. Therefore, by the implicit function theorem and Theorem \ref{thm regularity of non induced pressure}, we are done. 
\end{proof}

\begin{prop}\label{prop monotone}
We assume that for all $\alpha\in (\alpha_{\inf},\alpha_{\sup})\setminus\frat$ we have $b(\alpha)<\delta$. Then, 
     $b$ is monotone increasing on $(\alpha_{\inf},\min \frat)$ and  monotone decreasing on $(\max \frat, \alpha_{\sup})$. Moreover, it is strictly increasing on $(\alpha_{\inf},\min \frat)\cap\nice$ and strictly decreasing on $(\max \frat, \alpha_{\sup})\cap\nice$.
\end{prop}
\begin{proof}
We assume that for all $\alpha\in (\alpha_{\inf},\alpha_{\sup})\setminus\frat$ we have $b(\alpha)<\delta$.
    Let $\alpha_1,\alpha_2\in (\alpha_{\inf},\min \frat)$ satisfy $\alpha_1<\alpha_2$.  
    We take a small $\epsilon>0$ such that $\alpha_1+\epsilon<\alpha_2<\min\frat-\epsilon$ and $b(\alpha_1)<\delta-\epsilon$. We first show that there exists $\mu\in M(f)$ such that  
    \begin{align}\label{eq proof of monotonicity saturated}
        \lambda(\mu)>0,\ \min\frat-\epsilon<\mu(\phi) \text{ and } \delta-\epsilon<\frac{h(\mu)}{\lambda(\mu)}.
    \end{align}
    By Theorem \ref{thm bowen formula}, there exists $\nu\in M(f)$ such that $0<\lambda(\nu)<\infty$, $\nu(\phi)<\infty$ and $\delta-\epsilon<h(\nu)/\lambda(\nu)$. Moreover, there exists $p\in [0,1)$ such that $\min \frat-\epsilon<p\alpha_{\underline{i}}+(1-p)\nu(\phi)$. We set $\mu:=p\delta_{x_{\underline{i}}}+(1-p)\nu$. Then, we obtain $h(\mu)/\lambda(\mu)=h(\nu)/\lambda(\nu)$ and thus, $\mu$ satisfies \eqref{eq proof of monotonicity saturated}. Moreover, by Theorem \ref{thm coditional variational principle}, there exists $\xi\in M(f)$ such that $\lambda(\xi)<\infty$, $\xi(\phi)<\alpha_1+\epsilon$ and $\dim_H(\xi)>b(\alpha_1)-\epsilon$. 
    
    Since $\xi(\phi)<\alpha_2<\mu(\phi)$, there exists $\bar p\in (0,1)$ such that $\alpha_2=\bar p\xi(\phi)+(1-\bar p)\mu(\phi)$.  
We set $\bar\nu=\bar p\xi+(1-\bar p)\mu$. Then, we obtain $\bar\nu(\phi)=\alpha_2$ and 
\begin{align*}
&h(\bar \nu)=\bar ph(\xi)+(1-\bar p)h(\mu)
>(b(\alpha_1)-\epsilon)(\bar p\lambda(\xi)+(1-\bar p)\lambda(\mu))    
=(b(\alpha_1)-\epsilon)\lambda(\bar \nu).
\end{align*}
Hence, by Theorem \ref{thm coditional variational principle}, we obtain 
\begin{align}\label{eq monotone}
b(\alpha_2)>b(\alpha_1)-\epsilon.    
\end{align}
Letting $\epsilon\to0$, we obtain $b(\alpha_2)\geq b(\alpha_1)$ and thus, $b$ is monotone increasing on $(\alpha_{\inf},\min \frat)$. If $\alpha_1\in \nice$ then by Proposition \ref{prop equilibrium measure conditional}, we have $h(\mu_{b(\alpha_1),q(\alpha_1)})=b(\alpha_1)\lambda(\mu_{b(\alpha_1),q(\alpha_1)})$ and $\mu_{b(\alpha_1),q(\alpha_1)}(\phi)=\alpha_1$. Thus, by slightly modifying the above argument, we can remove $\epsilon$ in \eqref{eq monotone} namely, we obtain $b(\alpha_2)>b(\alpha_1)$. This implies that $b$ is strictly increasing on $(\alpha_{\inf},\min \frat)\cap\nice$.    
By a similar argument, one can show the second half.
\end{proof}

In the following three sections, we analyze the set $\nice$
 under various assumptions on $\ratio$.

\subsection{Analysis of the set $\nice$ under the condition $\ratio=0$}
In this section, we always assume that $\ratio=0$. 
\begin{lemma}\label{lemma 0 delta*} Assume that $\ratio=0$.
 Then, we have $(s_\infty,\infty)\times \mathbb{R}\subset\fin$.  Moreover, we have $\delta^*=s_\infty$. 
\end{lemma}
\begin{proof}
    Let $b>s_\infty$. By (P), if for all $q<0$ we have $p(b,q)<\infty$  then for all $q\in\mathbb{R}$ we have $p(b,q)<\infty$. Let $q<0$ and let $\epsilon$ be a strictly positive number such that $b-\epsilon>s_\infty$. Since $\ratio=0$, there exists $N\in\mathbb{N}$ such that for all $i\geq N$ and $x\in \oricodingmap([i])$ we have $-q\phi(x)<\epsilon \log|f'(x)|$. Thus, by using \eqref{eq finite pressure geometric pote} and the inequality $b-\epsilon>s_\infty$, we obtain 
    \begin{align*}
        \sum_{i=N}^\infty e^{\multipote_{b,q}([i])}\ll \sum_{i=N}^\infty e^{(-(b-\epsilon)\log |f'|)\circ\oricodingmap([i])}<\infty.
    \end{align*}
    By \eqref{eq finite pressure}, we obtain $p(b,q)<\infty$. This also implies $\delta^*\leq s_\infty$. 
    
    We shall show that $\delta^*\geq s_\infty$. Let $\delta^*<b$ and let $\epsilon >0$. Then, by (P), there exists $q>0$ such that we have $p(b,q)<\infty$ and thus, $\sum_{i\in \edge}e^{\multipote_{b,q}([i])}<\infty$. Since $\ratio=0$, there exists $N\in\mathbb{N}$ such that for all $i\geq N$ and $x\in \oricodingmap([i])$ we have $-\epsilon \log|f'|<-q\phi(x)$. Hence, we obtain
       \begin{align*}
        \sum_{i=N}^\infty e^{(-(b+\epsilon)\log |f'|)\circ\oricodingmap([i])}\leq 
        \sum_{i=N}^\infty e^{\multipote_{b,q}([i])}<\infty.
    \end{align*}
    By \eqref{eq finite pressure geometric pote}, this implies that $b+\epsilon>s_\infty$. Letting $\epsilon\to0$, we obtain $b\geq s_\infty$. Since $b$ is an arbitrary number with $\delta^*<b$, we obtain $\delta^*\geq s_\infty$.
\end{proof}

By Remark \ref{rem cheaking} and Lemma \ref{lemma 0 delta*}, we have
\begin{align}\label{eq good 0}
   &\{b(\alpha)\}\times(0,\infty)\subset \mathcal{N} \text{ if $\alpha\in(\alpha_{\inf},\min\frat)$ with $b(\alpha)>\delta^*$ and}
   \\&\{b(\alpha)\}\times(-\infty,0)\subset \mathcal{N} \text{ if $\alpha\in(\max\frat,\alpha_{\sup})$ with $b(\alpha)>\delta^*$.}\nonumber
\end{align}

\begin{prop}\label{prop only frat}
     For all $\alpha\in (\alpha_{\inf},\alpha_{\sup})\setminus \frat$ we have $b(\alpha)<\delta$.
\end{prop}

\begin{proof}
    We first consider the case $\alpha\in (\alpha_{\inf},\min \frat)$. For a contradiction, we assume that there exists  $\alpha\in (\alpha_{\inf},\min \frat)$ such that $b(\alpha)=\delta$. Then, by Lemma \ref{lemma 0 delta*}, we have $b(\alpha)=\delta>\delta^*$. Hence,
    by Theorem \ref{thm bowen formula} and Lemma \ref{lemma positivity p alpha}, we have $\palpha(\delta,0)=0$ and $\palpha(\delta,q)\geq 0$ for all $q>0$. By the convexity of the function $q\mapsto\palpha(\delta,q)$, we have $(\palpha)_q^+(\delta,0)\geq0$. 
    On the other hand, by Proposition \ref{prop derivatibe of pressure weak conditions} and Proposition \ref{prop equilibrium state Lambda}, we have
     $   (\palpha)_q^+(\delta,0)=\sup_{\nu\in M_\delta}\{-\nu(\phi)\}+\alpha= -\min \frat+\alpha<0.
    $
    This is a contradiction. Therefore, for all $\alpha\in (\alpha_{\inf},\min \frat)$ we have $b(\alpha)<\delta$. By a similar argument, one can show that  for all $\alpha\in (\max \frat,\alpha_{\sup})$ we have $b(\alpha)<\delta$. 
    \end{proof}

\begin{prop}\label{prop relationship 0}
    For each $\alpha\in (\alpha_{\inf},\alpha_{\sup})\setminus\frat$ with $b(\alpha)>s_\infty$ we have $\alpha\in \nice$. Moreover, if $\alpha\in (\alpha_{\inf},\min \frat)$ then $q(\alpha)\in (0,\infty)$ and if $\alpha\in (\max\frat,\alpha_{\sup})$ then $q(\alpha)\in (-\infty,0)$.
\end{prop}
\begin{proof}
If $\alpha_{\inf}=\alpha_{\sup}$, there is nothing to prove. Thus, we assume that $\alpha_{\inf}<\alpha_{\sup}$.
Let $\alpha\in(\alpha_{\inf},\min \frat)$ with $b(\alpha)>s_\infty$. 
By Theorem \ref{thm bowen formula}, for all $b\in (s_\infty,\delta)$ we have $\palpha(b,0)>0$ and thus, $(s_\infty,\delta)\times\{0\}\subset \mathcal{N}$. 
Hence, by \eqref{eq good 0} and Theorem \ref{thm regularity of non induced pressure}, the function $\palpha$ is real-analytic on an open set $\mathcal{O}$ containing $\{b(\alpha)\}\times [0,\infty)\cup (s_\infty,\delta)\times \{0\}$. We will show that 
\begin{align}\label{eq proof of relation}
\frac{\partial}{\partial q}\palpha(b(\alpha),0)<0.    
\end{align}
For a contradiction, we assume that 
\begin{align}\label{eq proof of relation contradiction}
\frac{\partial}{\partial q}\palpha(b(\alpha),0)\geq 0.    
\end{align}
We take a small number $\epsilon>0$ such that  $\alpha<\min \frat-\epsilon$. By Lemma \ref{lemma equilibrium measure near Delta}, there exists $b_0\in(b(\alpha),\delta)$ such that  $\min\frat-\epsilon\leq \mu_{b_0}(\phi)$. Then, by Theorem \ref{thm regularity of non induced pressure}, we have 
\begin{align*}
    \frac{\partial}{\partial q}\palpha(b_0,0)=-\mu_{b_0}(\phi)+\alpha<-\min \frat+\epsilon+\alpha<0.
\end{align*}
Combining this with \eqref{eq proof of relation contradiction} and using the continuity of the function $b\in (s_\infty,\delta)\mapsto ({\partial}/{\partial q})\palpha(b,0)$, there exists $b'\in [b(\alpha),b_0)$ such that $({\partial}/{\partial q})\palpha(b',0)=0$. By Lemma \ref{lemma sufficent condition for nice}, this yields that $\palpha(b',0)=0$.
Since $b'<b_0<\delta$, this contradicts $0<\palpha(b',0)$ and we obtain \eqref{eq proof of relation}. 
Since we assume that $\alpha_{\inf}<\alpha_{\sup}$, Theorem \ref{thm regularity of non induced pressure} yields that the function $q\mapsto\palpha(b(\alpha),q)$ is strictly convex on $[0,\infty)$. Thus, by \eqref{eq proof of relation} and Lemma \ref{lemma positivity p alpha},  there exists a unique number $q(\alpha)\in (0,\infty)$ such that $({\partial}/{\partial q})\palpha(b(\alpha),q(\alpha))=0$. Therefore, by Lemma \ref{lemma sufficent condition for nice}, we obtain $\alpha\in\nice$. By a similar argument, we can show that for all $\alpha\in (\max A, \alpha_{\sup})$ with $b(\alpha)>s_\infty$ we have $\alpha\in\nice$.
\end{proof}

\subsection{Analysis of the set $\nice$ under the condition $\ratio=\infty$}

\begin{lemma}\label{lemma infinite delta*}
    We assume that $\ratio=\infty$. Then, we have $ \mathbb{R}\times(0,\infty)\cup(s_\infty,\infty)\times\{0\}=\fin$.
    In particular, $\delta^*=0$.
\end{lemma}

\begin{proof}
    We assume that $\ratio=\infty$. By \eqref{eq finite pressure geometric pote}, we have $(s_\infty,\infty)\times\{0\}\subset\fin$ and for all $b\leq s_\infty$ we have $(b,0)\notin\fin$. Let $(b,q)\in \mathbb{R}\times (0,\infty)$. We take a large number $M>0$ with $qM+b>s_\infty$. Since $\ratio=\infty$, there exists $N\in\mathbb{N}$ such that for all $i\geq N$ and $x\in \oricodingmap([i])$ we have $\phi(x)\geq M\log|f'(x)|$. Thus, by (G) we obtain 
    \begin{align*}
        \sum_{i=N}^{\infty}e^{\multipote_{b,q}([i])}\leq \sum_{i=N}^{\infty}e^{(-qM\log|f'|)\circ\oricodingmap([i])}e^{(-b\log|f'|)\circ\oricodingmap([i])}\asymp\sum_{i=N}^\infty\frac{1}{i^{\growth(qM+b)}}.
    \end{align*}
    Since $qM+b>s_\infty=\growth^{-1}$, we obtain $p(b,q)<\infty$. By a similar argument, one can show that for all $(b,q)\in \mathbb{R}\times(-\infty,0)$ we have $(b,q)\notin \fin$.
\end{proof}

Combining Lemma \ref{lemma infinite delta*} with Lemma \ref{lemma weak lower bound}, we can see that if $\ratio=\infty$ then for all $\alpha\in (\alpha_{\inf},\alpha_{\sup})$ we have 
$b(\alpha)>\delta^*.$
 By Remark \ref{rem cheaking} and Lemma \ref{lemma infinite delta*}, if $\ratio=\infty$ and $\alpha\in(\alpha_{\inf},\min \frat)$ then we have
\begin{align}\label{eq nice infinity}
\{b(\alpha)\}\times(0,\infty)\subset\mathcal{N}.    
\end{align}

\begin{prop}\label{prop only frat infinity}
    Assume that $\ratio=\infty$. Then, for all $\alpha\in (\alpha_{\inf},\min\frat)$ we have $b(\alpha)<\delta$.
\end{prop}
\begin{proof}
    For a contradiction, we assume that there exists $\alpha\in (\alpha_{\inf},\min\frat)$ such that $b(\alpha)=\delta$. Then, by \eqref{eq nice infinity}, there are two possible cases: (1) $\limsup_{q\to+0}\mu_{\delta,q}(\phi)=\infty$. (2) $\limsup_{q\to+0}\mu_{\delta,q}(\phi)<\infty$. 

We first assume that we are in the case (1). Then, there exists a sequence $\{q_n\}_{n\in\mathbb{N}}\subset(0,\infty)$ such that $\lim_{n\to\infty}q_n=0$ and $\lim_{n\to\infty}\mu_{\delta,q_n}(\phi)=\infty$. 
By Theorem \ref{thm regularity of non induced pressure} and the convexity of the function $q\in [0,\infty)\mapsto p(\delta,q)$, we obtain 
\[
(\palpha)^+_q(\delta,0)=\lim_{n\to\infty}\frac{\partial}{\partial q}p(\delta,q_n)+\alpha=-\lim_{n\to\infty}\mu_{\delta,q_n}(\phi)+\alpha=-\infty.
\]
Since $\palpha(\delta,0)=p(\delta,0)+0=0$ this implies that there exists $q_0\in(0,\infty)$ such that $\palpha(\delta,q_0)<0$. This contradicts Lemma \ref{lemma positivity p alpha}. 

Next, we assume that we are in the case (2). Let $\{q_n\}_{n\in\mathbb{N}}$ be a sequence in $(0,\infty)$ such that for all $n\in\mathbb{N}$ we have $q_n>q_{n+1}$ and $\lim_{n\to\infty}q_n=0$. We will show that $\{(\delta,q_n)\}_{n\in\mathbb{N}}$ satisfies (T1.1), (T1.3), (T2) and (T3) in Section \ref{sec thermodynamic}. 
Since 
$\lim_{n\to\infty}\indupressure(\delta,q_n,p(\delta,q_n))=\indupressure(\delta,0,0)=0$
and $(\delta,q_n)\in \mathcal{N}$ for all $n\in\mathbb{N}$
, $\{(\delta,q_n)\}_{n\in\mathbb{N}}$ satisfies (T2) and (T3). 
Since for all $n\in\mathbb{N}$ we have $q_n>0$, \eqref{eq finite pressure geometric pote} yields that $\{(\delta,q_n)\}_{n\in\mathbb{N}}$ satisfies (T1.3). 
Since $\lim_{y\to\infty}e^{-y}y=0$ and $\lim_{y\to0}e^{-y}y=0$, we have $\sup_{y\in [0,\infty)}e^{-y}y<\infty$. Therefore, we obtain
\[
\sup_{i\in\edge}\sup_{n\in\mathbb{N}}e^{-q_n\phi\circ\oricodingmap([i])}q_n \phi\circ \oricodingmap([i])<\infty.
\]
Therefore, by the acceptability of $\phi$ and $\log|f'|$, for all $n\in\mathbb{N}$ we have
\begin{align*}
    &\sum_{i\in\edge}e^{\multipote_{\delta,q_n}([i])}|\multipote_{\delta,q_n}|([i])
    \ll\sum_{i\in\edge}e^{(-\delta\log|f'|\circ\oricodingmap)([i])}\log|f'|\circ\oricodingmap([i])
    \\&+
    \sum_{i\in\edge}e^{(-\delta\log|f'|\circ\oricodingmap)([i])}
    e^{-q_n\phi\circ\oricodingmap([i])}q_n \phi\circ \oricodingmap([i])
    \ll\sum_{i\in\edge}e^{(-\delta\log|f'|\circ\oricodingmap)([i])}\log|f'|\circ\oricodingmap([i]).
\end{align*}
Hence, by (G), $\{q_n\}_{n\in\mathbb{N}}$ satisfies (T1.1). 
Thus, by Lemma \ref{lemma tight}, there exist a subsequence $\{q_{n_{k}}\}_{k\in\mathbb{N}}$ of $\{q_n\}_{n\in\mathbb{N}}$ and $(\mu_\infty')^*\in M(\sigma)$ such that we have $\lim_{k\to\infty}\mu_{\delta,q_{n_k}}'=(\mu_\infty')^*$
Moreover, by \eqref{eq lift limit measure invariant} and Lemma \ref{lemma uniformly integrable}, we have $\mu_\infty^*:=(\mu_\infty')^*\circ \pi^{-1}\in M(f)$ and   $\lim_{k\to\infty}\lambda(\mu_{\delta,q_{n_k}})=\lambda(\mu_\infty^*)$.
Then, $\mu_\infty^*\in M_\delta$. Indeed, 
since $\limsup_{q\to+0}\mu_{\delta,q}(\phi)<\infty$, we have $\lim_{k\to\infty}q_{n_k}\mu_{\delta,q_{n_k}}(\phi)=0$. Therefore, 
by Lemma \ref{lemma}, we have 
    \begin{align*}
        &p(\delta,0)=
        \limsup_{k\to\infty}p(\delta,q_{n_k})=
        \limsup_{k\to\infty}(h(\mu_{\delta,q_{n_k}})-q_{n_k}\mu_{\delta,q_{n_k}}(\phi)-\delta\lambda(\mu_{\delta,q_{n_k}}))
        \\&\leq h(\mu_\infty^*)-\delta\lambda(\mu_\infty^*). 
    \end{align*}
    Hence, by Proposition \ref{prop equilibrium state Lambda}, $\mu_\infty^*(\phi)\in\frat$.
    Since $\phi\circ \pi$ is bounded from below, the Portmanteau theorem and \eqref{eq special lift} yield that $\mu_\infty^*(\phi)= (\mu_\infty')^*(\phi\circ \pi)\leq \liminf_{k\to\infty}\mu_{\delta,q_{n_k}}'(\phi\circ \pi)\leq \liminf_{k\to\infty}\mu_{\delta,q_{n_k}}(\phi)$. Thus, 
    by the convexity of $q\mapsto p(\delta,q)$ on $[0,\infty)$, we obtain
\begin{align*}
    p^+_q(\delta,0)=\lim_{k\to\infty}\frac{\partial}{\partial q}p(\delta,q_{n_k})=-\lim_{k\to\infty}\mu_{\delta,q_{n_k}}(\phi)\leq-\mu_{\infty}^*(\phi).
\end{align*}
Therefore, since $\alpha<\min\frat$,  we obtain
\begin{align}\label{eq proof only frat infinity}
    (\palpha)_{q}^{+}(\delta,0)\leq-\mu_{\infty}^*(\phi)+\alpha\leq-\min\frat+\alpha<0.
\end{align}
Since $\palpha(\delta,0)=p(\delta,0)+0=0$, this implies that there exists $q_0\in(0,\infty)$ such that $\palpha(\delta,q_0)<0$. This contradicts Lemma \ref{lemma positivity p alpha}. Therefore, we are done.
\end{proof}

Let $F$ be a finite subset of $\edge$. We define the topological pressure of $-q\phi-b\log|f'|$ ($(b,q)\in\mathbb{R}^2$) with respect to the dynamical system $(f_F,\Lambda_F)$ by 
\[
p_F(b,q):=\sup\{h(\mu)+\mu(-q\phi-b\log|f'|):\mu\in M(f,F)\}.
\]
Then, by Remark \ref{rem measurable bijection oricoding}, for all $(b,q)\in\mathbb{R}^2$ we have $p_F(b,q)=(P_\sigma)_F(\multipote_{b,q})$. Therefore, by Theorem \ref{thm compact approximation}, for all $(b,q)\in\mathbb{R}^2$ we obtain
\begin{align}\label{eq compact appoximation}
p(b,q)=\sup\{p_F(b,q):F\subset \edge,\ \#F<\infty\}.
\end{align}
For all finite subsets $F$ of $\edge$, $\alpha\in (\alpha_{\inf},\alpha_{\sup})$ and $(b,q)\in\mathbb{R}^2$ we set $p_{\alpha,F}(b,q):=p_F(b,q)+q\alpha$.
The following lemma follows by exactly the same arguments as those used in the proof of \cite[Lemma 3.2]{transience} (see also \cite[Lemma 5.2]{A.}) involving the definition of the topological pressure (the variational principle) and the compact approximation property of the topological pressure \eqref{eq compact appoximation}.
\begin{lemma}\label{lemma compact approximation}
    If $\alpha\in (\alpha_{\inf},\alpha_{\sup})$, $b>0$ and $\inf\{\palpha(b,q):q\in\mathbb{R}\}>0$ then there exists a finite set $F\subset \edge$ with $\p \subset F$ and $F\cap\h\neq \emptyset$ satisfying the following properties:  
    \begin{itemize}
        \item[(C1)] For all $q\in\mathbb{R}$ we have $p_{\alpha,F}(b,q)>0$.
        \item[(C2)] We have $\lim_{|q|\to\infty} p_{\alpha,F}(b,q)=\infty$.
    \end{itemize} 
\end{lemma}
For all finite subsets $F$ of $\edge$ and $\alpha\in (\alpha_{\inf, F},\alpha_{\sup,F})$ we define
\[
b_F(\alpha):=\dim_H
\left(
\left\{
x\in \Lambda_F: \lim_{n\to\infty}\frac{1}{n}\sum_{i=0}^{n-1}\phi(f^{i}(x))=\alpha
\right\}
\right).
\]
The following proposition is shown in \cite[Proposition 4.9]{arimanonuniformly} by using the same strategy as in the proof of Proposition \ref{prop relationship 0}:
\begin{prop}\label{prop relationship compact}
    Let $F$ be a finite subset of $\edge$ such that $\p\subset F$.
    For each $\alpha\in (\alpha_{\inf, F},\alpha_{\sup,F})\setminus\frat$  there exists a unique number $q_F(\alpha)\in \mathbb{R}$ such that   
   \[
p_{\alpha,F}(b_F(\alpha),q_F(\alpha))=0.        
\]
\end{prop}

\begin{prop}\label{prop relationship infinity}
    Assume that $\ratio=\infty$
    . Then, we have $(\alpha_{\inf},\min\frat)\subset \nice$ and  $q(\alpha)\in (0,\infty)$. 
\end{prop}
\begin{proof}
We assume that $\ratio=\infty$
.
If $\alpha_{\inf}=\alpha_{\sup}$, there is nothing to prove. Thus, we assume that $\alpha_{\inf}<\alpha_{\sup}$.
Let $\alpha\in (\alpha_{\inf},\min\frat)$.
By Lemma \ref{lemma sufficent condition for nice}, \eqref{eq nice infinity} and Theorem \ref{thm regularity of non induced pressure}, it is enough to show that there exists $q_0\in(0,\infty)$ such that 
\begin{align}\label{eq proof relation deribative 0}
    \frac{\partial}{\partial q}\palpha(b(\alpha),q_0)=0.
\end{align}
To this end, for a contradiction, we assume that there is no $q_0\in (0,\infty)$ satisfying \eqref{eq proof relation deribative 0}.  Then, by Proposition \ref{prop only frat infinity} and Theorem \ref{thm bowen formula}, there are two possible cases: 
\begin{itemize}
    \item[(1)] $\palpha(b(\alpha),0)=\infty$.
    \item[(2)] $0<\palpha(b(\alpha),0)<\infty$ and there is no $q_0\in (0,\infty)$ satisfying \eqref{eq proof relation deribative 0}.
\end{itemize}
 However, if we are in the case (1) then by Lemma \ref{lemma positivity p alpha}, we can find $q_0\in (0,\infty)$ satisfying \eqref{eq proof relation deribative 0}. Hence, we assume that we are in the case (2). 
In this case, by Lemma \ref{lemma positivity p alpha},  for all $\tilde q\in (0,\infty)$ we have $(\partial/\partial q)\palpha(b(\alpha),\tilde q)\geq 0$. 
Hence, by Lemma \ref{lemma infinite delta*}, for all $q\in \mathbb{R}$ we have $\palpha(b(\alpha),q)>0$. 
Thus, by \eqref{eq compact range}, Lemma \ref{lemma weak lower bound} and Lemma \ref{lemma compact approximation}, there exists a finite set $F\subset \edge$ with $\p\subset F$  
such that $\alpha\in (\alpha_{\inf,F},\min A)$ and the conditions (C1) and (C2) in Lemma \ref{lemma compact approximation} are satisfied. 
{
However, by Proposition \ref{prop relationship compact}, there exists $q_F(\alpha)\in\mathbb{R}$ such that $p_{\alpha,F}(b_F(\alpha),q_F(\alpha))=0$. Since for all $x\in \Lambda_F$ we have $-\log|f'(x)|\leq 0$ and $b_F(\alpha)\leq b(\alpha)$, we obtain $p_F(b(\alpha),q_F(\alpha))\leq p_F(b_F(\alpha),q_F(\alpha))$, which yields that 
\[
0<p_{\alpha,F}(b(\alpha),q_F(\alpha))\leq p_{\alpha,F}(b_F(\alpha),q_F(\alpha))=0.
\]
This is a contradiction. Hence, we conclude that there exists $q_0\in(0,\infty)$ satisfying \eqref{eq proof relation deribative 0} and we are done.
}
\end{proof}

\subsection{Analysis of the set $\nice$ under the condition  $0<\ratio<\infty$}
Recall that if $\phi$ satisfies (L) then we have $\ratio=\theta$.
\begin{lemma}\label{lemma alpha delta*}
    We assume that $\phi$ satisfies (L). Then, for each $b\in \mathbb{R}$ we have $p(b,q)=\infty$ if $q\leq (s_\infty-b)/\theta$ and $p(b,q)<\infty$ if $q>(s_\infty-b)/\theta$. 
    In particular, we have $\delta^*=0$. Moreover, for all $b\in\mathbb{R}$ we have $\lim_{q\to (s_\infty-b)/\theta}p(b,q)=\infty$. 
\end{lemma}
\begin{proof}
We assume that $\phi$ satisfies (L). Then, by (G), for every compact set $C\subset \mathbb{R}^2$ and $(b,q)\in C$ we have
\begin{align}\label{eq proof alpha delta*}
 \sum_{i\in \edge}e^{\multipote_{b,q}([i])}\asymp \sum_{i\in\edge}e^{((-q\theta-b)\log|f'|)\circ\oricodingmap([i])}\asymp\sum_{i=1}^{\infty}\frac{1}{i^{\growth(q\theta+b)}}.   
\end{align}
Hence, for all $b\in \mathbb{R}$ we have $p(b,q)=\infty$ if $q\leq (s_\infty-b)/\theta$ and $p(b,q)<\infty$ if $q>(s_\infty-b)/\theta$. Moreover, for all $b\in\mathbb{R}$ we obtain 
$\lim_{q\to (s_\infty-b)/\theta}\sum_{i\in\edge}e^{\multipote_{b,q}([i])}=\infty.$
Since $\oricodingsp$ is a full-shift, this implies that $\lim_{q\to (s_\infty-b)/\theta}p(b,q)=\infty$.
\end{proof}

Combining Lemma \ref{lemma alpha delta*} with Lemma \ref{lemma weak lower bound}, we can see that if $\phi$ satisfies (L) then for all $\alpha\in (\alpha_{\inf},\alpha_{\sup})$ we have 
\[b(\alpha)>\delta^*.\]
 Moreover, by Remark \ref{rem cheaking} and Lemma \ref{lemma alpha delta*}, if $\phi$ satisfies (L) then we have
\begin{align}\label{eq nice alpha}
&\{b(\alpha)\}\times\left(
\max\left\{({s_\infty-b(\alpha)})/{\theta},0\right\},\infty\right)\subset \mathcal{N} \text{ if }\alpha\in(\alpha_{\inf},\min\frat) \text{ and }
\\&\{b(\alpha)\}\times\left(({s_\infty-b(\alpha)})/{\theta},0\right)\subset \mathcal{N} \text{ if }\alpha\in(\max\frat,\alpha_{\sup}).   \nonumber 
\end{align}

{
\begin{rem}
Even if condition (L) is replaced by the weaker condition that $0<\ratio<\infty$, one can still obtain the following result by slightly modifying the proof of Lemma \ref{lemma alpha delta*}: Assume that $0<\ratio <\infty$. Then, for each $b\in\mathbb{R}$ we have $p(b,q)=\infty$ if $q<(s_\infty-b)/\ratio$ and $p(b,q)<\infty$ if $q>(s_\infty-b)/\ratio$.    
\end{rem}
The above remark enables us to use the same argument as in the proof of Proposition \ref{prop only frat} under the assumption $0<\ratio<\infty$. Thus, we obtain the following result:
\begin{prop}\label{lemma only frat alpha}
    We assume that $0<\ratio<\infty$. Then, for all $\alpha\in (\alpha_{\inf},\alpha_{\sup})\setminus \frat$ we have $b(\alpha)<\delta$.
\end{prop}
}

\begin{prop}\label{prop relation alpha}
    We assume that $\phi$ satisfies 
    (L). Then, $ (\alpha_{\inf},\alpha_{\sup})\setminus \frat\subset \nice$ and for all $\alpha\in (\max\frat,\alpha_{\sup})$ we have $b(\alpha)>s_\infty$. Moreover, if $\alpha\in (\alpha_{\inf},\min\frat)$ then $q(\alpha)\in (0,\infty)$ and if $\alpha\in (\max\frat,\alpha_{\sup})$ then $q(\alpha)\in ((s_\infty-b(\alpha))/\theta,0)$. 
\end{prop}
\begin{proof}
    We assume that $\phi$ satisfies 
    (L). 
    If $\alpha_{\inf}=\alpha_{\sup}$, there is nothing to prove. Thus, we assume that $\alpha_{\inf}<\alpha_{\sup}$.
    We first consider the case $\alpha\in (\alpha_{\inf},\min\frat)$. Then, there are two possible cases: (1) $b(\alpha)\leq s_\infty$. (2) $b(\alpha)>s_\infty$. 
    
    We first assume that we are in the case (1). Then, we have $s_\infty-b(\alpha)\geq 0$. Thus, by Lemma \ref{lemma positivity p alpha}, Lemma \ref{lemma alpha delta*} and \eqref{eq nice alpha}, there exists $q(\alpha)\in ((s_\infty-b(\alpha))/\theta, \infty)$ such that $(\partial/\partial q)\palpha(b(\alpha),q(\alpha))=0$. By our assumption that $\alpha_{\inf}<\alpha_{\sup}$ and Theorem \ref{thm regularity of non induced pressure}, such a number $q(\alpha)$ is uniquely determined. 
    Hence, by Lemma \ref{lemma sufficent condition for nice}, we obtain $\alpha\in \nice$.

    Next, we assume that we are in the case (2). By the same argument as in the proof of \eqref{eq proof of relation} in Proposition \ref{prop relationship 0}, one can show that $(\partial/\partial q)\palpha(b(\alpha),0)<0$. Thus, by Lemma \ref{lemma positivity p alpha}, there exists $q(\alpha)\in (0, \infty)$ such that $(\partial/\partial q)\palpha(b(\alpha),q(\alpha))=0$. By repeating the argument as in the case (1), we obtain $\alpha\in \nice$.

    We next consider the case $\alpha\in (\max\frat,\alpha_{\sup})$. We first show that $b(\alpha)>s_\infty$. For a contradiction, we assume that $b(\alpha)\leq s_\infty$.  
    Notice that for all $q\in (0,\infty)$ we have $\palpha(b(\alpha),q)\geq q(-\alpha_{\underline{i}}+ \alpha)>0$. 
    Therefore, combining the assumption that $b(\alpha)\leq s_\infty$ with Lemma \ref{lemma positivity p alpha} and Lemma \ref{lemma alpha delta*},
    we obtain $\inf\{\palpha(b(\alpha),q):q\in\mathbb{R}\}>0$. Thus, by \eqref{eq compact range} and Lemma \ref{lemma compact approximation}, there exists a finite set $F\subset \edge$ with $\p\subset F$ 
    such that $\alpha\in (\max A,\alpha_{\sup,F})$ and the conditions (C1) and (C2) in Lemma \ref{lemma compact approximation} are satisfied.
{
However, by Proposition \ref{prop relationship compact}, there exists $q_F(\alpha)\in\mathbb{R}$ such that $p_{\alpha,F}(b_F(\alpha),q_F(\alpha))=0$. Since $b_F(\alpha)\leq b(\alpha)$, we obtain $p_F(b(\alpha),q_F(\alpha))\leq p_F(b_F(\alpha),q_F(\alpha))$, which yields that 
\[
0<p_{\alpha,F}(b(\alpha),q_F(\alpha))\leq p_{\alpha,F}(b_F(\alpha),q_F(\alpha))=0.
\]
This is a contradiction. Hence, $b(\alpha)>s_\infty$ which yields that $(s_\infty-b)/\theta<0$.
}


By the same argument as in the proof of \eqref{eq proof of relation} in Proposition \ref{prop relationship 0}, one can show that
$({\partial}/{\partial q})\palpha(b(\alpha),0)>0.$    
 Hence,
by Lemma \ref{lemma alpha delta*}, there exists the unique number $q(\alpha)\in ((s_\infty-b(\alpha))/\theta,0)$ such that $(\partial/\partial q)\palpha(b(\alpha),q(\alpha))=0$. Thus, by Lemma \ref{lemma sufficent condition for nice}, we obtain $\alpha\in \nice$. 
\end{proof}


\emph {Proof of Theorem \ref{thm main}.}
By Proposition \ref{prop frat}, for all $\alpha\in \frat$ we have $b(\alpha)=\delta$.

We prove (B1) of Theorem \ref{thm main} under the assumption that $\ratio=0$. We first show that there exists $a^*\in [\alpha_{\inf},\min\frat]$ such that for all $\alpha\in (a^*,\min\frat)$ we have 
\[
s_\infty<b(\alpha)<\delta
\]
and there exists the unique ergodic measure $\mu\in M(f)$ such that $0<\lambda(\mu)<\infty$, $\int\phi d\mu=\alpha$ and $b(\alpha)=h(\mu)/\lambda(\mu)$. Moreover, $b$ is real-analytic and strictly increasing on $(a^*,\min\frat)$ and for all $\alpha\in( \alpha_{\inf},a^*]\setminus \frat$ we have $b(\alpha)=s_\infty$.
By Propositions \ref{prop only frat} and \ref{prop monotone}, \[
b(\alpha)<\dimension\]
for all $\alpha\in (\alpha_{\inf} ,\min \frat)$ and the function $b$ is monotone increasing on $(\alpha_{\inf},\min \frat)$. 
Moreover, by Lemmas \ref{lemma lower bound} and \ref{lemma 0 delta*}, we have \[
b(\alpha)\geq s_\infty\] for all $\alpha\in (\alpha_{\inf},\min \frat)$.
Thus, by the continuity of $b$ on $(\alpha_{\inf},\min \frat)$ (see Lemma \ref{lemma conti tilde b} and Theorem \ref{thm stong conditional variational principle}), we are in one of the following three cases.
\begin{itemize}
    \item[(1)] We have $b(\alpha)=s_\infty$ for all $\alpha\in (\alpha_{\inf},\min \frat)$. In this case, we set $a^*:=\min \frat$.
    \item[(2)] We have $s_\infty<b(\alpha)<\dimension$ for all $\alpha\in (\alpha_{\inf},\min \frat)$. In this case, we set $a^*:=\alpha_{\inf}$.
    \item[(3)] There exists $a^*\in (\alpha_{\inf},\min \frat)$ such that $s_\infty<b(\alpha)<\dimension$ for all $\alpha\in (a^*,\min \frat)$ and $b(\alpha)=s_\infty$ for all $\alpha\in (\alpha_{\inf},a^*]$.
\end{itemize}
Combining this with Propositions \ref{prop relationship 0}, \ref{prop equilibrium measure conditional}, \ref{prop real analytic}, and \ref{prop monotone} we obtain the desired result. 
By a similar argument, one can show that there exists $b^*\in [\max\frat,\alpha_{\sup}]$ such that for all $\alpha\in (\max \frat,b^*)$ we have $s_\infty<b(\alpha)<\delta$ and there exists the unique measure $\mu\in M(f)$ such that $0<\lambda(\mu)<\infty$, $\int\phi d\mu=\alpha$ and $b(\alpha)=h(\mu)/\lambda(\mu)$. Moreover, $b$ is real-analytic and strictly decreasing on  $(\max\frat,b^*)$, and for all $\alpha\in[b^*,\alpha_{\sup})\setminus \frat$ we have $b(\alpha)=s_\infty$.

(B1) of Theorem \ref{thm main} follows from 
Propositions \ref{prop only frat}, \ref{prop relationship 0}, \ref{prop equilibrium measure conditional}, \ref{prop real analytic}, and \ref{prop monotone}. 

(B2) of Theorem \ref{thm main} follows from  Lemma \ref{lemma weak lower bound}, Propositions \ref{prop frat under some assumption}, \ref{prop only frat infinity}, \ref{prop relationship infinity}, \ref{prop equilibrium measure conditional}, \ref{prop real analytic} and \ref{prop monotone}. 

(B3) of Theorem \ref{thm main} follows from Lemma \ref{lemma weak lower bound}, Propositions \ref{lemma only frat alpha}, \ref{prop relation alpha},  \ref{prop equilibrium measure conditional},  \ref{prop real analytic} and  \ref{prop monotone}.
\qed

\emph {Proof of Theorem \ref{thm main second}.} Theorem \ref{thm main second} follows from Theorem \ref{thm main} and Proposition  \ref{lemma only frat alpha}.

\appendix
\section{}
In this section, we prove Theorem \ref{thm uniquness and existence of the equilibrium state} and Theorem \ref{thm regularity of non induced pressure}. The details of the technical calculations in the proofs of Theorem \ref{thm uniquness and existence of the equilibrium state} and Theorem \ref{thm regularity of non induced pressure} can be found in the proofs of \cite[Theorems 3.3 and 3.4]{arimanonuniformly}.

\emph{Proof of Theorem \ref{thm uniquness and existence of the equilibrium state}.}
Let $(b,q)\in\mathcal{N}$. We first show that $\indupressure(b,q,p(b,q))=0$. Let $\epsilon>0$ be a small number with $ \LB+\epsilon<p(b,q)$. By \cite[Theorem 2.1.8]{mauldin2003graph}, there exists an ergodic measure $\mu\in M(f)$ such that $\mu (-q\phi-b\log|f'|)>-\infty$ and 
$
h(\mu)+\mu (-q\phi-b\log|f'|)>p(b,q)-\epsilon.
$
Then, $\mu\notin \conv$. Indeed, if $\mu\in \conv$ then we have 
$h(\mu)+\mu (-q\phi-b\log|f'|)+\epsilon\leq  \LB+\epsilon<p(b,q)$ which yields a contradiction. Thus, by Remark \ref{rem measurable bijection}, Lemma \ref{lemma equivalent condition not liftable} and \eqref{eq classical Abramov-Kac's formula}, we obtain 
\begin{align}\label{eq proof of uniquness existence larger than zero}
    &\indupressure(b,q,p(b,q)-\epsilon)\geq \measure( \return )\left(h(\mu)+\mu (-q\phi-b\log |f'|)-(p(b,q)-\epsilon)\right)>0, 
\end{align}
where $\measure=\mu|_{\indulimit}/\mu(\indulimit)$. 
By Lemma \ref{lem finite pressure on a special parameter}, 
the function $s\mapsto\indupressure(b,q,s)$ is continuous on $( \LB,\infty)$. Hence, by \eqref{eq proof of uniquness existence larger than zero} and \eqref{eq induced pressure is less than zero}, we obtain $\indupressure(b,q,p(b,q))=0$. Moreover, by \eqref{eq induced pressure is less than zero}, the measure $\mu_{b,q}$ is an equilibrium measure for $-q\phi-b\log|f'|$.

Next, we shall show the uniqueness of the equilibrium measure. Let $\nu$ be an equilibrium measure for $-q\phi-b\log|f'|$.
By the ergodic decomposition theorem (see \cite[Theorem 5.1.3]{viana}), we may assume that $\nu$ is ergodic.  As above, we have $\nu\notin \conv$ and thus, $\nu(\indulimit)>0$. 
Let $\tilde\nu=\nu|_{\indulimit}/\nu(\indulimit)$.
By Remark \ref{rem measurable bijection}, there exists $\tilde \nu'\in M(\shift)$ such that $\tilde \nu=\tilde \nu'\circ\codingmap^{-1}$ and $h(\tilde \nu)=h(\tilde \nu')$.
Then, $\tilde\nu'$ is an equilibrium measure for $\indumultipote_{b,q}$.
Indeed, by Theorem \ref{thm variational principle induce}, Remark \ref{rem measurable bijection}, \eqref{eq classical Abramov-Kac's formula} and \eqref{eq induced pressure is less than zero}, we have 
\begin{align*}
   &0\geq\indupressure(b,q,p(b,q))
   \geq\tilde\nu( \return)\left(h(\nu)+\nu (-q\phi-b\log |f'|)-p(b,q)\right)=0. \nonumber    
\end{align*}
Therefore, by Theorem \ref{thm equilibrium state induce}, we obtain $\tilde \nu'=\codingmeasure_{b,q}$ and thus, $\nu=\mu_{b,q}$.\qed

For two functions $\psi_1,\psi_2:\inducoding\rightarrow\mathbb{R}$ and $(b,q)\in \mathcal{N}$ we define the asymptotic variance of $\psi_1$ and $\psi_2$ by
\[
\sigma^2_{b,q}(\psi_1,\psi_2):=    \lim_{n\to\infty}\frac{1}{n}
    \codingmeasure_{b,q}\left( S_n\left(\psi_1-\codingmeasure_{b,q}(\psi_1) \right)
    S_n\left(\psi_2-\codingmeasure_{b,q}(\psi_2)\right)\right)
\]
when the limit exists. If $\psi_1=\psi_2$ then we write $\sigma^2_{b,q}(\psi_1):=\sigma^2_{b,q}(\psi_1,\psi_2)$.

\emph{Proof of Theorem \ref{thm regularity of non induced pressure}.}
Let $(b_0,q_0)\in\mathcal{N}$. By Lemma \ref{lem finite pressure on a special parameter}, there exists an open neighborhood $O\subset \mathbb{R}^3$ of $(b_0,q_0,p(b_0,q_0))$ such that for all $(b,q,s)\in O$ we have $\indupressure(b,q,s)<\infty$. Also, by Theorem \ref{thm regularity of induced pressure}, we have 
$
\left.
({\partial}/{\partial s})\indupressure(b,q,s)\right|_{(b,q,s)=(b_0,q_0,p(b_0,q_0))}=-\codingmeasure_{b_0,q_0}( \return\circ\codingmap )<0.
$
Therefore, by the implicit function theorem and Theorem \ref{thm regularity of induced pressure}, the function $p$ is real-analytic at $(b_0,q_0)$ and \eqref{eq Abramov-Kac's formula} gives \eqref{eq ruell's formula noninduced}.
Also, by \eqref{eq ruell's formula noninduced}, the implicit function theorem and Ruelle's formula for the second derivative of the pressure function \cite[Proposition 2.6.14]{mauldin2003graph}, we obtain
\begin{align}\label{eq important formula second derivative}
    \frac{\partial^2}{\partial q^2}p(b,q)
    =\frac{\sigma^2_{b,q}(\indupote\circ\codingmap-\mu_{b,q}(\phi)\return\circ\codingmap)}{\measure_{b,q}( \return ) }.
\end{align}
We shall show the last statement in Theorem \ref{thm regularity of non induced pressure}. If $\alpha_{\inf}=\alpha_{\sup}$ then by \eqref{eq ruell's formula noninduced}, $\frac{\partial^2}{\partial q^2}p(b,q)=0$. 
Conversely, we assume that  $\frac{\partial^2}{\partial q^2}p(b,q)=0$. Then, by \eqref{eq important formula second derivative}, we have $\sigma^2_{b,q}(\indupote\circ\codingmap-\mu_{b,q}(\phi)\return\circ\codingmap)=0$. Thus, by Lemma \ref{lemma finite return time} and \cite[Lemma 4.8.8]{mauldin2003graph}, there exists a bounded continuous function $\tilde u':\inducoding\rightarrow \mathbb{R}$ such that 
\begin{align}\label{eq coboundary}
    (\indupote-\mu_{b,q}(\phi)\return)\circ\codingmap=\tilde u'-\tilde u'\circ\shift.
\end{align}
Recall that, by Remark \ref{rem measurable bijection}, $\codingmap|_{\inducoding\setminus\codingmap^{-1}(J_0)}$ is one-to-one and $\codingmap|_{\inducoding\setminus\codingmap^{-1}(J_0)}^{-1}$ is continuous.
For 
$x\in J_0\cap \indulimit$ we fix $\tau_x\in \inducoding$ with $x=\codingmap(\tau_x)$ and define $\tilde u:\indulimit\rightarrow \mathbb{R}$ by $\tilde u|_{\indulimit\setminus J_0}=\tilde u'\circ\codingmap|_{\inducoding\setminus\codingmap^{-1}(J_0)}^{-1}$ and $\tilde u(x)=\tilde u'(\tau_x)$ for $x\in J_0\cap \indulimit$. Since $J_0$ is a countable set, $\tilde u$ is a Borel measurable bounded function satisfying 
\begin{align}\label{eq proof of cohomologous indu coboundary}
    \indupote-\mu_{b,q}(\phi)\return=\tilde u-\tilde u\circ\indumap \text{ on }\indulimit \setminus J_0. 
\end{align}
From this, it is not difficult to see that for all $i\in\p$ we have 
\begin{align}\label{eq proof of cohomologous boundary}
\alpha_i=\mu_{b,q}(\phi).    
\end{align}
We set $J_0'=\bigcup_{n=0}^{\infty}f^{-n}(\bigcup_{\omega\in \edge}\partial_{\Lambda}I_{\omega})$,
\[
N:=\bigcup_{i\in\p}\{x_i\}\setminus J_0',\ 
Z:=\bigcup_{i\in\p}\bigcup_{n\in\mathbb{N}}f^{-n}(x_i)\setminus (N\cup J_0')
\text{ and }
P:=\bigcup_{i\in\p}I_{ii}\setminus(N\cup Z\cup\indulimit\cup J_0').
\]
Note that by the definition of $J_0$, we have $J_0\subset J_0'$.
We will inductively construct a Borel measurable function $u:\Lambda\rightarrow \mathbb{R}$ such that for all $x\in \Lambda\setminus (Z\cup J_0')$ we have $\phi(x)=u(x)-u(f(x))+\mu_{b,q}(\phi)$.
Note that we have the direct decomposition
$    \Lambda=\indulimit\setminus J_0'\cup P\cup N\cup Z\cup J_0'.$
Define 
\begin{align}\label{eq proof of cohomologous construction 1}
    u(x):=\tilde u(x)
    \text{ for all $x\in \indulimit\setminus J_0'$}
    \text{ and }u(x)=0
    \text{ for all $x\in N\cup Z\cup(J_0'\cap\Lambda)$}.
\end{align}
For $i\in\p$ we define,  for $k\geq 2$,
\[
P_{i,k}:=\bigcup_{\omega\in \edge\setminus\{i\}}I_{i^k\omega}\setminus(N\cup Z\cup J_0')
.
\]
Then, we obtain the direct decomposition $P=\bigcup_{i\in\p}\bigcup_{k=2}^\infty P_{i,k}$.
Let $i\in\p$.
Since for $x\in P_{i,2}$
we have $f(x)\in \indulimit$, $u(f(x))$ is already defined by \eqref{eq proof of cohomologous construction 1}. Thus,  the following definition is well-defined:
$
u(x):=u(f(x))+\phi(x)-\mu_{b,q}(\phi)
\text{ for } x\in P_{i,2}.    $
Let $k\geq 3$.
Assume that for all $2\leq \ell\leq k$ and $x\in P_{i,\ell}$
we have already defined $u(x)$ by 
$u(x):=u(f(x))+\phi(x)-\mu_{b,q}(\phi).$
Since for $x\in P_{i,k+1}$
we have $f(x)\in P_{i,k}$, $u(f(x))$ is already defined. For $x\in P_{i,k+1}$ we define $u(x):=u(f(x))+\phi(x)-\mu_{b,q}(\phi)$. Therefore, by induction, the following definition is well-defined:
\begin{align}\label{eq proof of cohomologous def u on P}
    u(x)=u(f(x))+\phi(x)-\mu_{b,q}(\phi) \text{ for }x\in P.
\end{align}
We shall show that $u:\Lambda\rightarrow\mathbb{R}$ defined by \eqref{eq proof of cohomologous construction 1} and \eqref{eq proof of cohomologous def u on P} satisfies 
\begin{align}\label{eq proof of cohomologous coboundary}
    \phi(x)=u(x)-u(f(x))+\mu_{b,q}(\phi) \text{ for all } x\in \Lambda\setminus (Z\cup J_0').
\end{align} 
By \eqref{eq proof of cohomologous boundary}, \eqref{eq proof of cohomologous construction 1} and \eqref{eq proof of cohomologous def u on P}, for $x\in P\cup N$ we have \eqref{eq proof of cohomologous coboundary}. By \eqref{eq proof of cohomologous indu coboundary}, for $x\in \codingmap(\bigcup_{\omega\in \tilde E_1}[\omega])\setminus J_0'$ we have \eqref{eq proof of cohomologous coboundary}.  Let $n\geq 2$ and let $x\in \codingmap(\bigcup_{\omega\in \tilde E_n}[\omega])\setminus J_0'$.
Then, there exists $i\in\p$ such that for all $1\leq k\leq n-1$ we have $f^{k}(x)\in P_{i,n-(k-1)}$.
By \eqref{eq proof of cohomologous indu coboundary}, \eqref{eq proof of cohomologous construction 1} and \eqref{eq proof of cohomologous def u on P}, we have 
\begin{align*}
    &\sum_{k=0}^{n-1}(\phi-\mu_{b,q}(\phi))(f^k(x))=\tilde u (x)-\tilde u(f^n(x))
    \\&= u(x)-u(f(x))+\sum_{k=1}^{n-1}\left(u(f^k(x))-u(f^{k+1}(x))\right)
    \\&=u(x)-u(f(x))+\sum_{k=1}^{n-1}(\phi-\mu_{b,q}(\phi))(f^k(x)).
\end{align*}
Hence, we obtain $\phi(x)=u(x)-u(f(x))+\mu_{b,q}(\phi)$. This completes the proof of \eqref{eq proof of cohomologous coboundary}. 
 Since $Z$ is a countable set and there are no periodic orbits in $Z$, for all $\mu\in M(f)$ we have $\mu(Z)=0$. 
 Moreover, by (NERI1) and (NERI2), if $J_0'$ contains a periodic point then that periodic point must be $0$. By \eqref{eq coboundary} or \eqref{eq proof of cohomologous boundary} and \eqref{eq proof of cohomologous construction 1}, we have $\phi(0)=u(0)-u(f(0))+\mu_{b,q}(\phi)$.
Thus, by \eqref{eq proof of cohomologous coboundary}, it suffices to show that $u$ is bounded to conclude that $\mu(\phi)=\mu_{b,q}(\phi)$ for all $\mu\in M(f)$. We now show that $u$ is bounded. Since $\tilde u$ is bounded, \eqref{eq proof of cohomologous construction 1} implies that $u$ is bounded on $\indulimit\setminus J_0'\cup N\cup Z\cup (J_0'\cap \Lambda)$. Let $x\in P$ and let $h\in H$. Then there exist $\tau(x)\in \Sigma$, $i\in \p$ and $k\geq 2$ such that $x=\pi(\tau(x))$ and $x\in P_{i,k}$. We set $x_{-1}:=\pi(h\tau(x))$. Then we have $\return(x_{-1})=k$, $f^{k-1}(x)\in \indulimit\setminus J_0'$, $x_{-1}\in \indulimit\setminus J_0'$ and $f^{\ell+1}(x_{-1})=f^{\ell}(x)$ for all $\ell\in \mathbb{N}$.
Thus, by \eqref{eq proof of cohomologous coboundary} and \eqref{eq proof of cohomologous indu coboundary}, we obtain 
\begin{align*}
   &u(x)=u(f(x))+\phi(x)-\mu_{b,q}(\phi)=u(f^{2}(x)) +\phi(f(x))+\phi(x)-2\mu_{b,q}(\phi)=\cdots
   \\&=u(f^{k-1}(x))+\sum_{\ell=0}^{k-2}\phi(f^\ell(x))-(k-1)\mu_{b,q}(\phi)
   \\&=\tilde u(f^{k-1}(x))+\indupote(x_{-1})-\mu_{b,q}(\phi)\return(x_{-1})-\phi(x_{-1})+\mu_{b,q}(\phi)
   \\&=\tilde u(f^{k}(x_{-1}))+\tilde u(x_{-1})-\tilde u(f^{k}(x_{-1}))-\phi(x_{-1})+\mu_{b,q}(\phi)
   \\&=\tilde u(x_{-1})-\phi\circ\pi(h\tau(x))+\mu_{b,q}(\phi).
\end{align*}
Since $\phi\circ \pi$ is acceptable by (H), the function $\phi\circ \pi$ is bounded on $[h]$. Therefore, we obtain $\sup_{x\in \Lambda}|u|\leq \sup_{x\in \indulimit}|\tilde u|+\sup_{\tau\in [h]}|\phi\circ \pi|+\mu_{b,q}(\phi)<\infty$. This implies that $u$ is bounded. Consequently, for every
$\mu\in M(f)$, we have $\mu(u)=\mu(u\circ f)$.
This completes the proof.
\qed

\subsection*{Acknowledgments}
This work was supported by the JSPS KAKENHI 25KJ1382.

\bibliographystyle{abbrv}
\bibliography{reference}

\end{document}